\documentclass[11pt]{article}

\usepackage[utf8]{inputenc}
\usepackage[T1]{fontenc}
\usepackage{lmodern}
\usepackage{authblk}
\usepackage[letterpaper, margin=1in]{geometry}
\usepackage{microtype}

\usepackage{amsmath,amssymb,mathtools}
\usepackage[overload]{keytheorems}

\usepackage{enumitem}
\usepackage{booktabs}

\usepackage[small]{complexity}
\usepackage[ruled,vlined,linesnumbered]{algorithm2e}

\usepackage[hidelinks]{hyperref}
\usepackage{orcidlink}
\usepackage[capitalise,noabbrev]{cleveref}

\usepackage{tikz}

\newtheorem{theorem}{Theorem}[section]
\newtheorem{lemma}[theorem]{Lemma}
\newtheorem{proposition}[theorem]{Proposition}
\newtheorem{corollary}[theorem]{Corollary}
\theoremstyle{definition}
\newtheorem{definition}[theorem]{Definition}
\newtheorem{remark}[theorem]{Remark}

\newtheorem{example}[theorem]{Example}
\numberwithin{equation}{section}

\crefname{claim}{claim}{claims}
\Crefname{claim}{Claim}{Claims}

\newcommand{\Alt}{\operatorname{Alt}}
\newcommand{\Sym}{\operatorname{Sym}}

\newcommand{\Stab}{\operatorname{Stab}}

\newcommand{\diam}{\operatorname{diam}}
\newcommand{\supp}{\operatorname{supp}}
\newcommand{\F}{\mathbb F}
\newcommand{\N}{\mathbb N}
\newcommand{\Z}{\mathbb Z}
\newcommand{\one}{\mathbf{1}}
\newcommand{\BLength}{\textnormal{\textsc{Binary Length}}}
\newcommand{\BDiameter}{\textnormal{\textsc{Binary Diameter}}}
\newcommand{\ULength}{\textnormal{\textsc{Unary Length}}}
\newcommand{\UDiameter}{\textnormal{\textsc{Unary Diameter}}}
\newcommand{\BGLP}{\textnormal{\textsc{BGLP}}}
\newcommand{\wt}{\operatorname{wt}}
\newcommand{\rest}{\mathord\restriction}
\newcommand{\UT}{\operatorname{UT}}
\newcommand{\mbinom}[2]{\Bigl(\genfrac{}{}{0pt}{1}{\textstyle #1}{\textstyle #2}\Bigr)}
\newcommand{\fwd}{\triangleright}
\newcommand{\bwd}{\triangleleft}
\newcommand{\bidir}{\diamond}

\DeclarePairedDelimiter\abs{\lvert}{\rvert}
\DeclarePairedDelimiter\Abs{\lVert}{\rVert}

\providecommand\given{}
\makeatletter
\newcommand\mathset@symbol[1][]{%
  \nonscript\:#1\vert\allowbreak\nonscript\:\mathopen{}%
}
\DeclarePairedDelimiterX\Set[1]{\{}{\}}{%
  \renewcommand\given{\mathset@symbol[\delimsize]}%
  #1%
}

\newcommand{\sse}{\subseteq}

\DeclarePairedDelimiterX\Gen[1]{\langle}{\rangle}{%
  \renewcommand\given{\mathset@symbol[\delimsize]}%
  #1%
}

\makeatother

\hypersetup{
 pdftitle={Word Length and Diameter in Permutation Groups},
 pdfkeywords={permutation groups, diameter, word length, complexity}
}

\title{Word Length and Diameter in Permutation Groups}
\author{Markus Lohrey\,\orcidlink{0000-0002-4680-7198} and Alexander Thumm\,\orcidlink{0009-0005-4240-2045}}
\affil{University of Siegen, Germany}
\date{}

\usepackage{xcolor}
\newcommand{\Length}{\textnormal{\textsc{Length}}}
\newcommand{\Diameter}{\textnormal{\textsc{Diameter}}}

\begin{document}
\maketitle

\begin{abstract}
In the binary diameter problem, the input consists of explicitly given permutations generating a finite group~$G$ and a binary-encoded number~$k$, and the question is whether every element of~$G$ is a product of at most~$k$ of the given generators.
In the binary length problem, the input additionally contains a permutation $g \in G$, and the question is whether $g$ is such a product.
We show that the binary diameter problem is $\PSPACE$-complete, thereby confirming a conjecture of Lohrey and Rosowski (2023).
Moreover, we show that $\PSPACE$-hardness already holds for solvable groups of derived length at most~$3$, whereas the problem is $\Pi_2^{\P}$-complete for $2$-step nilpotent groups.
For the binary length problem, we establish a sharp complexity threshold: the problem is \NP-complete for $2$-step nilpotent groups, which extends a result of Lohrey and Rosowski for abelian groups, and \PSPACE-complete for groups that are simultaneously $3$-step nilpotent and metabelian, which strengthens the main result of Jerrum (1985).
Finally, we show that the unary diameter problem, where $k$ is encoded in unary, is $\Pi_2^{\P}$-complete for generating sets of the full symmetric group, confirming another conjecture of Lohrey and Rosowski.
\end{abstract}

\clearpage
\setcounter{tocdepth}{2}
\tableofcontents
\clearpage

\section{Introduction}
\label{sec:introduction}

Algorithmic problems for finite groups, and for permutation groups in particular, form an active research area at the interface of mathematics and theoretical computer science.
The algorithmic theory of permutation groups is well developed~\cite{Seress2003}, and its methods have found applications in a variety of areas: for instance, they are used in coding theory~\cite{Huffman1998CodesAndGroups} and in graph isomorphism testing~\cite{Babai16}, for the exhaustive generation of combinatorial objects~\cite{McKay1998}, in integer programming and combinatorial optimization~\cite{Margot2002}, as well as for Boolean satisfiability~\cite{AloulRamaniMarkovSakallah2003} and for symmetry reduction in model checking~\cite{ClarkeEmersonJhaSistla1998}.
In this paper, we study the complexity of computing lengths and diameters in finite permutation groups.

\subsection{Length and Diameter in Permutation Groups}
\label{sec:intro-motivation}

Let $G$ be a finite group that is generated by a set~$\Sigma$, so that every element of~$G$ can be written as a product of elements of~$\Sigma$.
A natural question is how long such products need to be.
This is measured by the \emph{length} $\ell(g, \Sigma)$ of an element $g \in G$, which is the smallest number of generators whose product is~$g$, and by the \emph{diameter} $\diam(G, \Sigma)$ of~$G$, which is the largest length of any of its elements.
Geometrically, these notions are distances in the (directed) \emph{Cayley graph} of~$G$ with respect to the generating set~$\Sigma$, whose vertices are the elements of~$G$ and in which each element~$g$ has a directed edge to~$ga$ for every generator $a \in \Sigma$~--~in this setting, the length of~$g$ is its distance from the identity, and the diameter of~$G$ is the diameter of this (strongly connected) graph.

These notions have been studied intensively from both a combinatorial and a computational point of view.
Upper and lower bounds on the diameters of various finite groups can be found, for instance, in~\cite{BabaiBS04,BabaiH05,BabaiH92, BabaiHKLS90,BabaiKL89,BabaiS88,BabaiS92,EberhardMainiSabatiniTracey2026,HeSe14,KMS84,McKe84}.

A famous open conjecture of Babai and Seress~\cite{BabaiS88} states that $\diam(G, \Sigma) \leq \mathcal{O}((\log \abs{G})^c)$ for every finite non-abelian simple group~$G$ and every generating set~$\Sigma$ of~$G$, where $c$ is a universal constant.
In particular, this would imply that the diameters of the alternating group~$\Alt(d)$ and the symmetric group~$\Sym(d)$ of degree~$d$ are bounded by a polynomial in~$d$ for all generating sets; the best known upper bound for these groups is $\exp(\mathcal{O}(\log^4 d \log \log d))$~\cite{HeSe14}.
For a recent account of diameters of permutation groups, see~\cite{EberhardMainiSabatiniTracey2026}, where they are also related to Grigorchuk's gap conjecture on the growth of finitely generated groups.
Length and diameter in permutation groups also have a more playful side: many questions about mechanical puzzles such as Rubik's Cube are questions about lengths and diameters in permutation groups (see \cref{sec:related}).

In this paper, we study the complexity of the following two natural decision problems.
The problem \Length\ is defined as follows.
\begin{quote}
\textbf{Input:} A set $\Sigma \sse \Sym(d)$ generating a group $G$, an element $g\in G$, and a number~$k\in\N$.\\[1.5mm]
\textbf{Question:} Does $\ell(g,\Sigma)\leq k$ hold?
\end{quote}
The problem \Diameter\ is defined as follows.
\begin{quote}
\textbf{Input:} A set $\Sigma \sse \Sym(d)$ generating a group $G$ and a number $k\in\N$.\\[1.5mm]
\textbf{Question:} Does $\diam(G,\Sigma)\leq k$ hold?
\end{quote}
Note that the condition $g \in G$ in the input of \Length\ can be tested in polynomial time~\cite{FHL1980}.
Both problems come in two variants, depending on the encoding of the threshold~$k$: for \ULength\ and \UDiameter, the number~$k$ is encoded in unary, and for \BLength\ and \BDiameter, it is encoded in binary.
In all variants, permutations are given as lists of images, so that the degree~$d$ of the permutation group~$G$ is essentially given in unary (see \cref{sec:preliminaries-groups-2}).

The binary encoding of~$k$ is the more natural one, since the diameter of a permutation group can be superpolynomial in its degree~$d$.
For instance, a permutation whose cycle lengths are the first $n$ primes $p_1, \dotsc, p_n$ acts on $d = p_1 + \dotsb + p_n \leq \mathcal{O}(n^2 \log n)$ points, but generates a cyclic group of diameter $p_1 \dotsm p_n - 1 \geq \exp(\Omega(n \log n))$ with respect to this single generator.

\subsection{Previous Work}
\label{sec:intro-previous}

The problems \Length\ and \Diameter\ were first studied by Even and Goldreich~\cite{EvenG81}.
Although they do not explicitly specify the encoding of the bound~$k$, their results show that they have the unary encoding in mind: they proved that \ULength\ is \NP-complete and that \UDiameter\ is \NP-hard.
Both hardness results already hold for elementary abelian $2$-groups, that is, for finite direct products of cyclic groups of order two.
The precise complexity of \UDiameter\ was settled by McLoughlin~\cite{McLoughlin1984} and later, independently, by the first author and Rosowski~\cite{LohreyRosowski2023}: the problem is $\Pi_2^{\P}$-complete,\footnote{McLoughlin~\cite{McLoughlin1984} considers the covering radius of binary linear codes. The covering radius of a code $C \leq \F_2^m$ is exactly the diameter of the elementary abelian $2$-group $\F_2^m / C$ with respect to the images of the standard basis vectors.} and hardness again holds already for elementary abelian $2$-groups.

For the more natural binary encoding of~$k$, Jerrum~\cite{Jerrum1985} proved that \BLength\ is \PSPACE-complete, even for two generators.
It is easy to see that \BDiameter\ belongs to~\PSPACE, and the first author and Rosowski~\cite{LohreyRosowski2023} conjectured that it is \PSPACE-complete.
Moreover, they showed that the restrictions of \BLength\ and \BDiameter\ to abelian groups are \NP-complete and $\Pi_2^{\P}$-complete, respectively.

\subsection{Our Results}
\label{sec:results}

Our first main result confirms the conjecture of the first author and Rosowski~\cite{LohreyRosowski2023} on \BDiameter\ mentioned above.\footnote{All hardness results are with respect to polynomial-time many-one reductions.}

\begin{theorem}[Binary Diameter]\label{thm:diameter-intro}
  The problem \BDiameter\ is $\PSPACE$-complete.
\end{theorem}

Membership in~\PSPACE\ is easy and was already observed in~\cite{LohreyRosowski2023}.
For \PSPACE-hardness, we reduce \BLength\ to \BDiameter\ (\cref{thm:diameter-reduction}) and use Jerrum's result that \BLength\ is \PSPACE-hard~\cite{Jerrum1985}.

Next, we consider nilpotent and solvable permutation groups (see \cref{sec:preliminaries-groups-1} for definitions).
Note that one can test in polynomial time whether a given permutation group is nilpotent or solvable and, if so, compute its nilpotency class or derived length, respectively (\cref{lem:nilpotent-solvable-algorithms}).
For the restriction of \BLength\ to nilpotent groups, we obtain the following dichotomy between nilpotency class two and three.

\begin{theorem}[Binary Length for Nilpotent Groups]\label{thm:length-intro}
For the problem \BLength,
\begin{enumerate}[label=\textup{(\alph*)}]
  \item\label{thm:dichotomy-2step} the restriction to $2$-step nilpotent groups is \NP-complete, and 
  \item\label{thm:dichotomy-3step} the restriction to $3$-step nilpotent, metabelian groups is \PSPACE-complete.
\end{enumerate}
\end{theorem}

Part~\ref{thm:dichotomy-2step} generalizes the corresponding result for abelian groups~\cite{LohreyRosowski2023}, from which the lower bound carries over.
For the upper bound, we show that the following problem, which may be of independent interest, belongs to~\NP: given a finite alphabet~$\Sigma$ and binary-encoded numbers $c_a$ for $a \in \Sigma$ and $c_{a,b}$ for distinct $a, b \in \Sigma$, is there a word that contains exactly $c_a$ occurrences of each letter~$a$ and exactly $c_{a,b}$ scattered occurrences of~$ab$ for all distinct $a$ and~$b$?
Part~\ref{thm:dichotomy-3step} strengthens Jerrum's result~\cite{Jerrum1985}, whose groups are, in general, not even solvable (\cref{prop:jerrum-quotient}).

Since \BDiameter\ is $\Pi_2^{\P}$-hard already for abelian groups~\cite{LohreyRosowski2023}, part~\ref{thm:dichotomy-2step} immediately yields the following.

\begin{corollary}\label{cor:diameter-2step-intro}
  The restriction of \BDiameter\ to $2$-step nilpotent groups is $\Pi_2^{\P}$-complete.
\end{corollary}

The reduction in \cref{thm:diameter-reduction} does not preserve nilpotency, but it increases the derived length of solvable groups by at most one.
Together with part~\ref{thm:dichotomy-3step} of \cref{thm:length-intro}, we thus obtain the following, as metabelian groups are precisely the solvable groups of derived length at most two.

\begin{corollary}\label{cor:diameter-derl3-intro}
  The restriction of \BDiameter\ to solvable groups is $\PSPACE$-complete, and this holds even when restricted to solvable groups of derived length at most three.
\end{corollary}

\Cref{tab:results} summarizes these and prior results for the binary variants of both problems.
In particular, the complexity of \BDiameter\ for $3$-step nilpotent groups and for metabelian groups remains open; in both cases, the problem is $\Pi_2^{\P}$-hard and belongs to~\PSPACE.

Finally, recall that \UDiameter\ is $\Pi_2^{\P}$-complete~\cite{McLoughlin1984,LohreyRosowski2023}.
We confirm another conjecture of the first author and Rosowski~\cite{LohreyRosowski2023}, stating that this remains true for generating sets of the full symmetric group.

\begin{theorem}[Unary Diameter of the Symmetric Group]\label{thm:symmetric-intro}
  The restriction of \UDiameter\ to sets generating the full symmetric group, meaning that $G = \Sym(d)$, is $\Pi_2^{\P}$-complete.
\end{theorem}

In fact, hardness holds even for generating sets that consist of involutions (\cref{thm:symmetric:involutions}).
The complexity of \BDiameter\ for generating sets of $\Sym(d)$ remains open; if the conjecture of Babai and Seress~\cite{BabaiS88} mentioned above holds, then this problem is $\Pi_2^{\P}$-complete as well.

All our hardness results hold even for symmetric generating sets, that is, for sets~$\Sigma$ with $\Sigma = \Sigma^{-1}$, which correspond to undirected Cayley graphs; in particular, the reduction in \cref{thm:diameter-reduction} always outputs symmetric generating sets.
This restriction is natural, for instance, for mechanical puzzles such as Rubik's Cube, whose moves are invertible (see \cref{sec:related}).

\begin{table}[t]
\centering

  \begin{tabular}{ccc}
  \toprule
  Class of Groups & \BLength\ & \BDiameter\ \\
  \midrule
  abelian & \NP-complete~\cite{LohreyRosowski2023} & $\Pi_2^{\P}$-complete~\cite{LohreyRosowski2023} \\
  $2$-step nilpotent & \NP-complete (Thm.~\ref{thm:length-intro}) & $\Pi_2^{\P}$-complete (Cor.~\ref{cor:diameter-2step-intro}) \\
  $3$-step nilpotent, metabelian & $\PSPACE$-complete (Thm.~\ref{thm:length-intro}) & open ($\Pi_2^{\P}$-hard, in $\PSPACE$) \\
  derived length $\leq3$ & $\PSPACE$-complete (Thm.~\ref{thm:length-intro}) & $\PSPACE$-complete (Cor.~\ref{cor:diameter-derl3-intro}) \\
  arbitrary & $\PSPACE$-complete~\cite{Jerrum1985} & $\PSPACE$-complete (Thm.~\ref{thm:diameter-intro}) \\
  \bottomrule
  \end{tabular}

  \caption{Complexity of the binary length and diameter problems.}
  \label{tab:results}
\end{table}

\subsection{Technical Overview}
\label{sec:overview}

\subsubsection*{From Binary Length to Binary Diameter}
Our main technical contribution is a polynomial-time reduction from \BLength, which is known to be \PSPACE-complete by Jerrum's result~\cite{Jerrum1985}, to \BDiameter\ (\cref{thm:diameter-reduction}).
That such a reduction is possible is not obvious, since the diameter is a global property: in a positive instance, all elements must be short, but generators that shorten all other elements tend to shorten the target as well.
Our construction is therefore somewhat intricate and proceeds in three steps.

First, we create a \emph{length gap}: in the direct product of the given group with a large cyclic group that counts letters, we find a central element whose length is at most one more than the given bound if the original instance is positive, and exceeds a huge threshold otherwise.

Second, we pass to the wreath product with a cyclic group of polynomial order, which consists of copies of the previous group at the positions of a cycle, together with shifts of this cycle.
Generators that act at single positions (\emph{fillers}) or shift the cycle (\emph{bridges}) yield short words for all elements.
At a few special positions, the fillers are restricted such that certain bridges (\emph{entries} and \emph{exits}) can produce the central element from the first step at a distinguished position only via a word for it over the generators from the first step; the only remaining shortcut is to use two arbitrary bridges.
To keep the number of bridges polynomial, we ensure that the special positions have small overlaps under shifts using a \emph{Sidon set}, that is, a set of positions whose pairwise differences are all distinct.

Third, we remove the shortcut with the help of another cyclic direct factor that serves as a \emph{clock}: the generators advance the clock by different amounts, so that every word for a suitable central target that is not longer than the threshold consists of fillers together with exactly one entry and one exit, which we show to be impossible for a negative initial instance.

\subsubsection*{Binary Length in 3-Step Nilpotent and Metabelian Groups}
We show that \BLength\ remains \PSPACE-hard for groups that are simultaneously $3$-step nilpotent and metabelian, whereas the groups in Jerrum's reduction~\cite{Jerrum1985} are in general not even solvable (\cref{prop:jerrum-quotient}).
To this end, we simulate a polynomial-space Turing machine by products of upper unitriangular $4 \times 4$ matrices, which generate such groups.
Since every generator can be applied in every configuration, the group itself cannot prevent incorrect steps; instead, it records a trace of every incorrect step, which is ruled out by our choice of the target.

More precisely, we store configurations in registers, and testing whether a register with current value~$t$ holds~$x$ adds $(t - x)^2$ to the upper right corner entry.
As a sum of squares, this entry vanishes, as demanded by the target, only if all tests succeed.
To keep the degree polynomial despite exponentially large entries, we work modulo polynomially many small primes simultaneously.

\subsubsection*{Binary Length and Diameter in 2-Step Nilpotent Groups}
The restrictions of \BLength\ and \BDiameter\ to abelian groups are \NP-complete and $\Pi_2^{\P}$-complete, respectively, as shown by the first author and Rosowski~\cite{LohreyRosowski2023}, and we extend these results to $2$-step nilpotent groups.
Our new upper bounds rely on a classical result going back to Magnus~\cite{Magnus1937}, which implies that the value of a word in a $k$-step nilpotent group depends only on its \emph{$k$-profile}, that is, the numbers of occurrences of its scattered subwords of length at most~$k$.

Instead of a word for a target element, which may be exponentially long, it therefore suffices to guess its $k$-profile in binary encoding and to verify that the guessed profile is realized by some word.
Using compression-based techniques as well as a Carath\'eodory-type bound for integer cones due to Eisenbrand and Shmonin~\cite{ES2006}, we show that this realization problem is in \NP\ for $2$-profiles.

\subsubsection*{Unary Diameter of the Symmetric Group}
We show that \UDiameter\ remains $\Pi_2^{\P}$-hard for sets of involutions that generate the full symmetric group.
Our reduction starts from the diameter problem for elementary abelian $2$-groups~$\F_2^d$, which is $\Pi_2^{\P}$-hard~\cite{McLoughlin1984,LohreyRosowski2023}; crucially, hardness already holds for generating sets of vectors with bounded Hamming weight (see \cref{appendix-source}).
We represent each of the $d$ coordinates by a pair of disjoint transpositions, replace each generator by all of its polynomially many lifts to these transpositions, and add all transpositions that involve an additional point.
We derive an exact formula for the length of a permutation with respect to the latter transpositions in terms of its cycles, and we use it to show that this construction increases the diameter by exactly~$5d$.

\subsection{Further Related Work}
\label{sec:related}

Eberhard, Maini, Sabatini, and Tracey~\cite{EberhardMainiSabatiniTracey2026} recently obtained new diameter bounds for finite groups, which yield polynomial bounds for two classes of permutation groups that contrast with our hardness results.
Specifically, the following holds, where $\diam(G)$ denotes the maximum of $\diam(G, \Sigma)$ over all generating sets~$\Sigma$ of~$G$.
\begin{itemize}
  \item If $G \leq \Sym(d)$ is solvable and transitive, then $\diam(G) \leq \mathcal{O}(d^5)$~\cite[Thm.~1.7]{EberhardMainiSabatiniTracey2026}.
  \item If $G \leq \Sym(d)$ is a $p$-group, then $\diam(G) \leq \mathcal{O}(d^{11})$.
    Indeed, every cycle length of an element of~$G$ is a power of~$p$, so that the exponent of~$G$ satisfies $\exp(G) \leq d$.
    Hence,  
    \[\diam(G, \Sigma \cup \Sigma^{-1}) \leq \exp(G) \cdot (4 \log_2 \abs{G})^8 \leq \mathcal{O}(d^{10})
    \]
    holds for every generating set~$\Sigma$ of~$G$ by \cite[Thm.~1.2]{EberhardMainiSabatiniTracey2026}, and replacing each $a^{-1} \in \Sigma^{-1}$ by at most $\exp(G) - 1 \leq d$ copies of~$a$ yields $\diam(G, \Sigma) \leq \mathcal{O}(d^{11})$.
\end{itemize}
For these classes, it therefore suffices to consider polynomially bounded thresholds~$k$, so that the restrictions of \BLength\ and \BDiameter\ to solvable and transitive groups and to $p$-groups belong to \NP\ and~$\Pi_2^{\P}$, respectively.
In particular, unless $\NP = \PSPACE$, the groups in part~\ref{thm:dichotomy-3step} of \cref{thm:length-intro} can be neither transitive nor $p$-groups.

Many mechanical puzzles, such as Rubik's Cube, the Pyraminx, and the Skewb, admit a natural description by permutation groups: their elementary moves permute the pieces and change their orientations, and these moves are invertible.
Under this correspondence, the minimal number of moves that transforms one configuration into another is the length of an element of a permutation group~$G$ with respect to the set~$\Sigma$ of allowed moves, and the minimal number of moves that suffices to solve every configuration is the diameter of~$G$ with respect to~$\Sigma$, which is often called \emph{God's number} in the puzzle literature.
For the $3 \times 3 \times 3$ Rubik's Cube, the precise value of God's number was open for a long time, until Rokicki, Kociemba, Davidson, and Dethridge~\cite{RokickiEtAl2013} showed that it equals~$20$ in the half-turn metric (and~$26$ in the quarter-turn metric~\cite{RokickiDavidson2014}).
For the generalized $n \times n \times n$ Rubik's Cube, Demaine, Demaine, Eisenstat, Lubiw, and Winslow~\cite{DemaineEtAl2011} proved that God's number is $\Theta(n^2 / \log n)$.
Further connections between puzzles and permutation groups can be found in~\cite{Joyner2008,KMS84,Wilson1974}.

Our upper bounds for $2$-step nilpotent groups rest on the fact that \textsc{Binary $2$-Realization} belongs to~\NP\ (\cref{pro:nilpotent:profile-realization}).
The connection between $2$-step nilpotent groups and $2$-profiles is well known in combinatorics on words.
Two words $u, v \in \Sigma^\ast$ are \emph{$k$-binomially equivalent}, written $u \sim_k v$, if they have the same $k$-profile, that is, if every word of length at most~$k$ occurs equally often as a scattered subword in~$u$ and in~$v$.
Lejeune, Rigo, and Rosenfeld~\cite{LejeuneRigoRosenfeld2020} showed that, for an alphabet~$\Sigma$ with $n$ letters, the quotient $\Sigma^\ast / {\sim_2}$ is the positive submonoid of the free $2$-step nilpotent group of rank~$n$.
Moreover, Freydenberger, Gawrychowski, Karhum\"aki, Manea, and Rytter~\cite{FreydenbergerEtAl2015} showed that $k$-binomial equivalence of two given words can be decided in deterministic polynomial time.
Let us also note that \textsc{Binary $2$-Realization} can easily be solved in deterministic polynomial time for alphabets with two letters.

For infinite finitely generated groups, a close relative of \ULength\ is the \emph{Bounded Geodesic Length Problem}, \BGLP\ for short: for a fixed group~$G$ with a finite symmetric generating set~$\Sigma$, the input consists of a word $w \in \Sigma^\ast$ and a number $k < \abs{w}$, and the question is whether the element represented by~$w$ is also represented by a word of length at most~$k$.
Parry~\cite{Parry1992} showed that \BGLP\ is \NP-complete for the wreath product $C_2 \wr (\Z \times \Z)$, and Myasnikov, Roman'kov, Ushakov, and Vershik~\cite{MyasnikovRomankovUshakovVershik2010} initiated a systematic study of \BGLP, showing that it is \NP-complete for every free metabelian group of rank at least two.

\subsection{Organization}

In \cref{sec:prel}, we fix our notation and recall the necessary background on groups.
In \cref{sec:diameter}, we reduce \BLength\ to \BDiameter\ (\cref{thm:diameter-reduction}), which proves \cref{thm:diameter-intro} together with Jerrum's result~\cite{Jerrum1985}.
In \cref{sec:length,sec:nilpotent}, we prove the two parts of our dichotomy in \cref{thm:length-intro} as well as their consequences for \BDiameter\ (\cref{cor:diameter-2step-intro,cor:diameter-derl3-intro}).
In \cref{sec:symmetric}, we prove \cref{thm:symmetric-intro}, and we conclude with open problems in \cref{sec:conclusion}.
The appendices show that Jerrum's groups are, in general, not solvable (\cref{appendix-jerrum}), and that hardness for elementary abelian $2$-groups already holds for generating sets of bounded Hamming weight (\cref{appendix-source}).
Generative AI tools were used to assist with obtaining some of our results; the AI methodology statement at the end of the paper (\hyperref[sec:ai-methodology]{page~\pageref*{sec:ai-methodology}}) describes their use.

\section{Preliminaries}
\label{sec:prel}

In this section, we briefly fix our notation and recall the necessary background on groups.

\subsection{General Notation}
\label{sec:preliminaries-general}

We write $\Z$ for the set of integers and $\N = \{0, 1, 2, \ldots\}$ for the set of natural numbers.
For $a, b \in \Z$, we write $[a, b] = \{ n \in \Z : a \leq n \leq b \}$; this interval is empty if $a > b$.

We write $A \uplus B$ for the disjoint union of sets $A$ and $B$.
For a set $A$ and $k \in \N$, a \emph{$k$-set} in $A$ is a subset $B \sse A$ with $\abs{B} = k$, and we write $\tbinom{A}{k}$ for the set of all $k$-sets in~$A$.

\paragraph{Words.}

An \emph{alphabet} is a finite set $\Sigma$ whose elements are called \emph{letters}.
A \emph{word} over $\Sigma$ is a finite sequence $w = a_1 \cdots a_n$ of letters $a_1, \ldots, a_n \in \Sigma$; its \emph{length} is $\abs{w} = n$, and its $i$-th letter is $w[i] = a_i$ for $i \in [1, n]$.
The set $\Sigma^\ast$ of all words over $\Sigma$ forms a monoid with respect to concatenation, whose neutral element is the \emph{empty word}~$\varepsilon$.
In this paper, alphabets are mostly generating sets of groups, so that every word represents a group element (see \cref{sec:preliminaries-groups-1}).

\paragraph{Numbers.}

For $n \in \N$, we write $\Abs{n} = \lceil \log_2(n + 1) \rceil$, which is the number of bits in the binary representation of~$n$ provided that $n \geq 1$.
We also extend this notation to vectors $t = (t_i)_{i \in I} \in \N^I$ over any finite index set~$I$ via $\Abs{t} = \sum_{i \in I} \Abs{t_i}$; note that $\Abs{t} \geq \Abs{\sum_{i \in I} t_i}$.

\paragraph{Complexity Classes.}

We use standard notation from complexity theory; see, e.g.,~\cite{Papadimitriou1994,AroraBarak2009}.
Specifically, we use the complexity classes \[
  \P \sse \NP \sse \Pi_2^{\P} \sse \PSPACE,
\] 
namely deterministic polynomial time ($\P$), nondeterministic polynomial time ($\NP$), the second universal level of the polynomial-time hierarchy ($\Pi_2^{\P}$)~\cite{stockmeyer}, and polynomial space ($\PSPACE$).
Throughout, hardness refers to polynomial-time many-one reductions.

\subsection{Abstract Groups}
\label{sec:preliminaries-groups-1}

We assume familiarity with basic group theory; see, e.g., Robinson's book~\cite{Robinson1996}.
All groups in this paper are finite.
As usual, we write groups multiplicatively, and denote the identity element of any group as well as the trivial group by~$1$.
We write $C_n$ for the cyclic group of order~$n$ and $G_1 \times G_2$ for the direct product of groups $G_1$ and~$G_2$.
Moreover, for elements $g$ and $h$ of a group, we write $g^h = h^{-1} g h$ and $\prescript{h\!}{}{g} = h g h^{-1}$ for the \emph{conjugates} of~$g$ by~$h$, and we use the same notation for subsets.

\paragraph{Subgroups and Cosets.}
We write $H \leq G$ if $H$ is a subgroup of $G$, and $N \trianglelefteq G$ if $N$ is a normal subgroup of $G$, meaning that $N^g = N$ for all $g \in G$.
For $H \leq G$, the sets $gH$ and $Hg$ with $g \in G$ are the \emph{left} and \emph{right cosets} of $H$ in $G$, respectively, and the number $[G : H] = \abs{G} / \abs{H}$ of left (equivalently, right) cosets is the \emph{index} of $H$ in $G$.
A set $T \sse G$ containing exactly one element from each left coset of $H$ in $G$, that is, with $G = \biguplus_{g \in T} gH$, is a \emph{left transversal} of $H$ in $G$; \emph{right transversals} are defined analogously.

\paragraph{Generating Sets, Length, and Diameter.}
The subgroup of $G$ \emph{generated} by a set $\Sigma \sse G$ is the smallest subgroup $\langle \Sigma \rangle \leq G$ that contains $\Sigma$, and we call $\Sigma$ a \emph{generating set} of $G$ if $\langle \Sigma \rangle = G$; such a set may contain the identity.
We write $\Sigma^{-1} = \{ a^{-1} : a \in \Sigma \}$ and call $\Sigma$ \emph{symmetric} if $\Sigma = \Sigma^{-1}$.

A word $w = a_1 \cdots a_n$ over $\Sigma \sse G$ \emph{represents} the element $g = a_1 \cdots a_n \in G$, where the empty word~$\varepsilon$ represents~$1$; we also call $g$ the \emph{value} of $w$ in $G$.
Since $G$ is finite, every element of $\langle \Sigma \rangle$ is represented by some word over~$\Sigma$; in particular, inverses of the generators are not needed.

\begin{definition}[Length and Diameter]\label{def:length-diameter}
  Let $G$ be a finite group and let $\Sigma \sse G$ be a generating set.
  The \emph{length} of an element $g \in G$ and the \emph{diameter} of the group $G$ with respect to $\Sigma$ are given by
  \begin{gather*}
    \ell(g,\Sigma) = \min \{
      k \in \N : g = a_1 \cdots a_k \text{ with } a_1, \ldots, a_k \in \Sigma
    \},\\
    \diam(G,\Sigma) = \max \{
      \ell(g,\Sigma) : g \in G
    \}.
  \end{gather*}
\end{definition}

\paragraph{Abelian, Nilpotent, and Solvable Groups.}

The \emph{center} of a group $G$ is its normal subgroup 
\[
  Z(G) = \{ g \in G : gh = hg \text{ for all } h \in G \},
\]
and its elements are called \emph{central}.
The \emph{commutator} of $g, h \in G$ is $[g,h] = g^{-1} h^{-1} g h$.
For $A, B \subseteq G$, we write $[A, B]$ for the subgroup of $G$ generated by all commutators $[g,h]$ with $g \in A$ and $h \in B$.
The \emph{derived subgroup} (or \emph{commutator subgroup}) of $G$ is then the normal subgroup $G' = [G, G]$.

A group $G$ is \emph{abelian} if all its elements commute, meaning that $gh = hg$ for all $g,h \in G$.
Thus, a group $G$ is abelian if and only if $Z(G) = G$ if and only if $G' = 1$.

The \emph{derived series} of a group $G$ is the descending series of normal subgroups $G^{(i)} \trianglelefteq G$ where 
\[
  G^{(0)} = G
  \quad\text{and}\quad
  G^{(i+1)} = [G^{(i)}, G^{(i)}].
\]
A group $G$ is \emph{solvable} if its derived series reaches the trivial subgroup after finitely many steps, meaning that $G^{(k)} = 1$ for some $k \geq 0$; the smallest such $k$ is then called the \emph{derived length} of $G$.
Thus, solvable groups of derived length $k \leq 1$ are precisely the abelian groups, and those of derived length $k \leq 2$ are called \emph{metabelian} groups.
If, for a normal subgroup $N \trianglelefteq G$, both $N$ and $G/N$ are solvable, then so is $G$, and the derived length of $G$ is at most the sum of their derived lengths.

The \emph{lower central series}\footnote{The lower central series is a central series, meaning that $\gamma_{i}(G) / \gamma_{i+1}(G) \leq Z(G / \gamma_{i+1}(G))$. 
Note also that we start the indexing with $\gamma_0(G) = G$, whereas $\gamma_1(G) = G$ is common in parts of the literature, including \cite{Robinson1996}.} of $G$ is the descending series of normal subgroups $\gamma_i(G) \trianglelefteq G$ where
\[
  \gamma_0(G) = G
  \quad\text{and}\quad
  \gamma_{i+1}(G) = [\gamma_i(G), G].
\]
A group $G$ is \emph{nilpotent} if its lower central series reaches the trivial subgroup after finitely many steps, meaning that $\gamma_k(G) = 1$ for some $k \geq 0$; the smallest such $k$ is then called the \emph{nilpotency class} of~$G$.
We also say that $G$ is \emph{$k$-step nilpotent} provided that $\gamma_k(G) = 1$. 

The classes of all finite solvable/nilpotent groups (of unbounded or bounded derived length/nil\-potency class) are closed under taking subgroups, quotients, and finite direct products.

\paragraph{Wreath Products.}

Besides direct products, a key ingredient of our proof of \cref{thm:diameter-intro} is the notion of (regular) wreath products.
Formally, these are defined as follows.

\begin{definition}[Wreath Products]\label{def:wreath-products}
  Let $G$ and $H$ be groups. 
  The \emph{wreath product} $G \wr H = G^H \rtimes H$ consists of all pairs $(f,t)$ with $f\colon H \to G$ and $t \in H$, with the multiplication
  \[
    (f, t) \cdot (f', t') = (f \cdot \prescript{t\!}{}{\smash{f'}}, t \cdot t'),
    \quad\text{where}\quad
    (f \cdot \prescript{t\!}{}{\smash{f'}})(x)
      = f(x) \cdot f'(xt).
  \]
\end{definition}

In a wreath product $G \wr H$, the normal subgroup $\{(f,1) : f \in G^H \} \cong G^H$ is the \emph{base group}, and the subgroup $\{(1, t) : t \in H\} \cong H$ is the \emph{active group}.
We call an element of $H$ a \emph{position} when it serves as an argument of a function $f\colon H \to G$, and a (pure) \emph{shift} when it is regarded as an element of the active group.
Moreover, we say that $(f,t) \in G \wr H$ has \emph{shift} $t$ and \emph{value} $f(x)$ at $x \in H$.

The group $G$ admits several natural embeddings into the base group $G^H$ and hence into $G \wr H$.
First, for every position $x \in H$, we can identify~$G$ with the \emph{coordinate subgroup} $\{(\delta^x_g, 1) : g \in G\}$, where $\delta^x_g \colon H \to G$ is given by $\delta^x_g(x) = g$ and $\delta^x_g(y) = 1$ for all $y \neq x$.
Note that identifying $t \in H$ with the pure shift $(1, t) \in G \wr H$ and $\delta^x_g \in G^H$ with $(\delta^x_g, 1) \in G \wr H$ yields $(\delta^x_g)^t = \delta^{xt}_g$ for all $g \in G$ and $x, t \in H$.
Second, we can identify the group $G$ with the \emph{diagonal subgroup} $\{(g, 1) : g \in G\}$, where, by abuse of notation, we let $g \colon H \to G$ also denote the constant function with value~$g$.

If $G$ and $H$ are solvable, then so is $G \wr H$, and its derived length is at most the sum of the derived lengths of $G$ and $H$.
If $G \neq 1$, then the center $Z(G \wr H)$ consists precisely of those elements $(f,1) \in G \wr H$ with constant $f \colon H \to Z(G)$, i.e., $Z(G \wr H)$ is the diagonal copy of $Z(G)$. 

\subsection{Permutation Groups}
\label{sec:preliminaries-groups-2}

The \emph{symmetric group} $\Sym(\Omega)$ consists of all permutations of a (finite) set $\Omega$.
We also write $\Sym(d)$ for $\Sym(\Omega)$ if $d = |\Omega|$.
Permutations act on the right, meaning that we write $\omega g$ for the image of a point $\omega \in \Omega$ under a permutation $g \in \Sym(\Omega)$, so that $\omega (gh) = (\omega g) h$.
For distinct points $\omega_1, \ldots, \omega_m \in \Omega$, the \emph{cycle} $(\omega_1, \ldots, \omega_m)$ maps $\omega_i$ to $\omega_{i+1}$, with indices modulo~$m$, and it fixes all other points; cycles of length two are called \emph{transpositions}.

A \emph{permutation group} is a subgroup $G \leq \Sym(\Omega)$, and its \emph{degree} is $\abs{\Omega}$.
The \emph{orbit} of a point $\omega \in \Omega$ under a permutation group $G \leq \Sym(\Omega)$ is $\omega G = \{ \omega g : g \in G \}$, and its \emph{stabilizer} is the subgroup $\Stab_G(\omega) = \{ g \in G : \omega g = \omega \} \leq G$.
The right cosets of $\Stab_G(\omega)$ in $G$ correspond to the points of the orbit $\omega G$ via $\Stab_G(\omega) g \mapsto \omega g$; in particular, $[G : \Stab_G(\omega)] = \abs{\omega G} \leq \abs{\Omega}$.

\paragraph{Algorithms.}

As input, a permutation group $G \leq \Sym(\Omega)$ is given by a generating set $\Sigma \sse \Sym(\Omega)$.
To this end, we identify the domain $\Omega$ with the set $[1, d]$ where $d = \abs{\Omega}$, and we represent each permutation $g \in \Sym(\Omega)$ by the list $(1g, \ldots, dg)$ of its images.
In this encoding, the degree $d$ of the permutation group $G$ is essentially given in unary, whereas the order of $G$ may be exponential in the size of this input.
Nevertheless, many problems for permutation groups can be solved in polynomial time by standard techniques based on computing a \emph{base} and a \emph{strong generating set} of $G$ with the Schreier--Sims algorithm~\cite{Sims1970,FHL1980}; see Seress' book~\cite{Seress2003} for a comprehensive treatment.

We use the following standard consequences of this construction (where the diameter bounds follow from factorization along a stabilizer chain); see, e.g.,~\cite[Sec.~4.1, Thm.~4.2.4, Sec.~5.1.1]{Seress2003}.

\begin{lemma}\label{lem:bsgs-algorithms}
  On input of a generating set $\Sigma$ of a permutation group $G \leq \Sym(\Omega)$, the following tasks can be performed in polynomial time.
  \begin{itemize}
    \item Test membership in $G$ for any given permutation $g \in \Sym(\Omega)$.
    \item Compute a generating set $\bar \Sigma$ of $G$ with $\diam(G, \bar\Sigma) \leq \abs{\Omega}$.
    \item Compute, for each point $\omega \in \Omega$, a generating set $\Sigma_\omega$ of the corresponding stabilizer $\Stab_G(\omega)$ with $\diam(\Stab_G(\omega), \Sigma_\omega) \leq \abs{\Omega}$, as well as a left and a right transversal of $\Stab_G(\omega)$ in $G$.
  \end{itemize}
\end{lemma}

As is well known, the classes of groups considered in this paper can also be recognized efficiently.

\begin{lemma}\label{lem:nilpotent-solvable-algorithms}
  On input of a generating set $\Sigma$ of a permutation group $G \leq \Sym(\Omega)$, one can decide in polynomial time whether $G$ is solvable and whether $G$ is nilpotent, and, if so, compute in polynomial time its derived length and its nilpotency class, respectively.
\end{lemma}

\begin{proof}
  Using normal closures~\cite{FHL1980}, one can compute generating sets of the terms of the derived series and of the lower central series one after another, and compare consecutive terms for equality, e.g., via membership of generators.
  Since each proper inclusion of subgroups at least halves the order, both series become stationary after at most $\log_2 {\abs{G}} \leq \log_2 {\abs{\Sym(\Omega)}} \leq \mathcal{O}(\abs{\Omega} \cdot \log {\abs{\Omega}})$ steps.
  Moreover, all generating sets can be kept of size polynomial in $\abs{\Omega}$ by sifting.

  Finally, the group~$G$ is solvable (resp.\ nilpotent) if and only if its derived series (resp.\ its lower central series) becomes stationary at the trivial subgroup, and, in this case, the index of the first trivial term is the derived length (resp.\ nilpotency class) of~$G$.
\end{proof}

\paragraph{Direct Products and Wreath Products.}

Direct products and wreath products of permutation groups are again permutation groups in a natural way.
For $G_1 \leq \Sym(\Omega_1)$ and $G_2 \leq \Sym(\Omega_2)$, the direct product $G_1 \times G_2$ is a permutation group on $\Omega_1 \uplus \Omega_2$ via $\omega (g_1, g_2) = \omega g_i$ for $\omega \in \Omega_i$.

For a permutation group $G \leq \Sym(\Omega)$ with $\Omega \neq \emptyset$ and an arbitrary group $H$, the wreath product~$G \wr H$ is a permutation group on $\Omega \times H$ via
\[
  (\omega, x) (f, t) = (\omega f(x), xt).
\]

Both of these constructions are effective: given the involved permutations (and, for wreath products, the multiplication table of $H$), the resulting permutations can be computed in polynomial time.
In particular, direct products allow us to efficiently realize cyclic groups of large order as permutation groups of small degree.

\begin{lemma}\label{lem:cyclic-group-construction}
  Given pairwise coprime integers $n_1, \dotsc, n_k \geq 1$, one can compute in time polynomial in $d = n_1 + \dotsb + n_k$ a permutation $c \in \Sym(\Omega)$ of order $n_1 \dotsm n_k$ with $\abs{\Omega} = d$.
\end{lemma}

\begin{proof}
  Let $\Omega = \Omega_1 \uplus \dotsc \uplus \Omega_k$ with $\abs{\Omega_i} = n_i$, and let $c = c_1 \dotsm c_k$, where $c_i$ is a cycle of length~$n_i$ on~$\Omega_i$.
  Clearly, $c$ can be computed in time polynomial in $d$.
  Since the cycles $c_1, \dotsc, c_k$ are disjoint, we have $c^e = 1$ if and only if $c_i^e = 1$ for all $i$, that is, if and only if $e$ is divisible by $n_1, \dotsc, n_k$.
  As the latter are pairwise coprime, this holds if and only if $e$ is divisible by $n_1 \dotsm n_k$.
\end{proof}

For choosing suitable pairwise coprime numbers, we use the following standard estimate for products of small primes.

\begin{lemma}\label{lem:primes}
Let $p_1,p_2,\ldots$ be the odd primes in increasing order. 
Then, $p_n \leq \mathcal{O}(n\log n)$ holds for all $n \geq 2$ and, consequently, for every fixed $e \geq 0$,
\[
  \sum_{j=1}^{n}p_j^e\leq\mathcal{O}(n^{e+1}\log^e n)
  \qquad\text{and}\qquad
  3^n\leq\prod_{j=1}^{n}p_j\leq2^{\mathcal{O}(n\log n)}.
\]
Moreover, the primes $p_1,\ldots,p_n$ can be computed in time polynomial in $n$.
\end{lemma}

\begin{proof}
  The bound on $p_n$ is the standard estimate for the $n$-th prime; see, e.g., \cite{RosserSchoenfeld1962}.
  The two displayed bounds follow from the estimates $3 \leq p_j \leq p_n$ valid for all $j \in [1, n]$.
  The primes can be computed by testing successive integers with trial division.
\end{proof}

\section{The Binary Diameter Problem}
\label{sec:diameter}

In this section, we prove \cref{thm:diameter-intro} via a reduction from \BLength\ to \BDiameter.

\begin{theorem}\label{thm:diameter-reduction}
There is a polynomial-time computable reduction that maps every instance $(\Sigma_0, g_0, k_0)$ of \BLength\ to an instance $(\Delta, m)$ of \BDiameter\ with $\Delta = \Delta^{-1}$ and
\[
 \ell(g_0, \Sigma_0) \leq k_0 \iff \diam(\langle\Delta\rangle, \Delta) \leq m.
\]
Moreover, $\langle\Delta\rangle$ is isomorphic to $\bigl((G_0 \times C_{q^2})\wr C_r\bigr)\times C_q$ for some $q, r \geq 1$, where $G_0 = \langle\Sigma_0\rangle$.
\end{theorem}

By a result of Jerrum~\cite{Jerrum1985}, the problem \BLength\ for arbitrary permutation groups is $\PSPACE$-complete.
Hence, \cref{thm:diameter-reduction} yields the $\PSPACE$-hardness part of \cref{thm:diameter-intro}, including the statement on symmetric generating sets.
Membership of \BDiameter\ in $\PSPACE$, on the other hand, is straightforward: an instance $(\Sigma, k)$ is positive if and only if $(\Sigma, g, k)$ is a positive instance of \BLength\ for every $g \in G = \langle \Sigma \rangle$.
Thus, it suffices to enumerate all $g \in \Sym(\Omega)$ in polynomial space, test whether $g \in G$ in polynomial time (\cref{lem:bsgs-algorithms}), and, if so, decide whether $\ell(g, \Sigma) \leq k$ in polynomial space (\cref{sec:length}).
Altogether, \cref{thm:diameter-intro} is a corollary of \cref{thm:diameter-reduction}.

More generally, \cref{thm:diameter-reduction} transfers hardness from \BLength\ to \BDiameter\ for restricted classes of groups:
if \BLength\ restricted to a class $\mathcal{C}$ of groups is $\PSPACE$-hard, then so is \BDiameter\ restricted to any class that contains all groups $((G_0 \times C_{q^2}) \wr C_r) \times C_q$ with $G_0 \in \mathcal{C}$ and $q, r \geq 1$.
In \cref{sec:length}, we show that \BLength\ is $\PSPACE$-hard already for metabelian groups (\cref{thm:length-intro}\ref{thm:dichotomy-3step}); together with \cref{thm:diameter-reduction}, this yields $\PSPACE$-hardness of \BDiameter\ for solvable groups of derived length at most three (\cref{cor:diameter-derl3-intro}).
Note, however, that the construction preserves neither nilpotency nor being metabelian, so that the precise complexity of \BDiameter\ remains open in these cases (see \cref{sec:conclusion}).

\paragraph{Proof Outline.}
The main difficulty is that the diameter is a global property: in a positive instance, every element of the group must be short, whereas in a negative instance, a single long element suffices.
In particular, we cannot simply output $\Sigma_0$ with a suitable threshold, since $\diam(G_0, \Sigma_0)$ may be attained by elements other than $g_0$, and adding generators that shorten all other elements tends to shorten $g_0$ as well.
Instead, we construct a larger group with a generating set and a central target whose length alone decides whether the diameter is small.

In the first step (\cref{sec:diameter:length-gap}), we create a length gap in a direct product with a large cyclic group:
we compute a generating set $\Sigma_1$ of $G_1 = G_0 \times C_{q^2}$ and an element $g_1 \in Z(G_1)$ such that $\ell(g_1, \Sigma_1) \leq k_1$ if $\ell(g_0, \Sigma_0) \leq k_0$, and $\ell(g_1, \Sigma_1) > m$ otherwise, where $m$ is much larger than $k_1$.

In the second step (\cref{sec:diameter:wreath-product}), we pass to the wreath product $G_2 = G_1 \wr C_r$ with a generating set $\Sigma_2$ that consists of \emph{fillers}, which write values at single positions, and \emph{bridges}, which realize shifts; among the bridges, we single out certain \emph{entries} and \emph{exits}.
If $\ell(g_0, \Sigma_0) \leq k_0$, then every element of the wreath product $G_2$ has a word of length at most $m$ over~$\Sigma_2$ that contains, apart from fillers, either a single bridge or a single entry and a single exit.
At the same time, words that contain, apart from fillers, only a single entry and a single exit can produce the value $g_1$ at the position~$1$ only via a word representing $g_1$ over $\Sigma_1$.
In the third step (\cref{sec:diameter:clock}), we add another large cyclic direct factor that serves as a clock.
This keeps all elements short if $\ell(g_0, \Sigma_0) \leq k_0$, but forces every word of length at most $m$ for the central target~$g_3$ to contain, apart from fillers, only an entry and its matching exit, which is impossible if $\ell(g_0, \Sigma_0) > k_0$.

\medskip

For the remainder of this section, we fix an instance $(\Sigma_0, g_0, k_0)$ of the problem \BLength, where $\Sigma_0 \sse \Sym(\Omega_0)$ generates the permutation group $G_0$ of degree $d_0 = \abs{\Omega_0}$, and $g_0 \in G_0$.
Without loss of generality, we also assume that $1 \in \Sigma_0$, since adding the identity to a generating set does not change the length of any element.

\subsection{Step 1: The Length Gap}
\label{sec:diameter:length-gap}

In this step, we construct a new instance $(\Sigma_1, g_1, k_1)$ of \BLength\ with $k_1 = k_0 + 1$, whose group $G_1 = \langle \Sigma_1 \rangle$ is the direct product $G_0 \times C_{q^2}$ with a large cyclic group, and whose target $g_1$ is central in $G_1$.
Moreover, this instance has a robust length gap: for a number $m$ that is much larger than $k_1$, we have $\ell(g_1, \Sigma_1) \leq k_1$ if $\ell(g_0, \Sigma_0) \leq k_0$, and $\ell(g_1, \Sigma_1) > m$ otherwise.

Recall from \cref{lem:cyclic-group-construction} that cyclic groups of large order can be realized as permutation groups of small degree.
Since we need another such group $C_q$ in Step~3, we construct both at once, for a suitable number~$q$.
To this end, we take $q$ to be the product of the squares of the first $n$ odd primes for a suitable $n \geq 1$, using the estimates from \cref{lem:primes}.

\begin{lemma}\label{lem:diameter:cyclic-group}
Given $(\Sigma_0,g_0,k_0)$, one can compute in polynomial time the least $n\geq1$ such that 
\begin{equation}\label{eqn:diameter:cyclic-group-parameters}
 q = \prod_{j=1}^{n}p_j^2,
 \qquad
 m = \frac{q+3}{4},
 \qquad
 d_{\mathrm c} = \sum_{j=1}^{n}p_j^4
\end{equation}
satisfy 
\begin{equation}\label{eqn:diameter:cyclic-group-choices}
  m \geq k_0 + 10(d_0 + d_{\mathrm c})^4 + 3.
\end{equation}
Moreover, one can compute in polynomial time a permutation $c$ of order $q^2$ on a set $\Omega_{\mathrm c}$ of size $d_{\mathrm c}$.
\end{lemma}

\begin{proof}
Since $q$ is an odd square, we have $q \equiv 1 \pmod 4$; hence, $m$ is an integer and $q = 4m - 3$.
By \cref{lem:primes}, we have $q \geq 9^n$ and, thus, $m \geq 9^n/4$, whereas $d_{\mathrm c} \leq \mathcal{O}(n^5 \log^4 n)$ and $\log_2 q \leq \mathcal{O}(n \log n)$.
As the bit length of $k_0$ and the degree $d_0$ are bounded by the input size, the inequality~\eqref{eqn:diameter:cyclic-group-choices} therefore holds for some $n$ that is polynomial in the input size, and hence so is the least such $n$.
It can be found by generating the odd primes one after another (\cref{lem:primes}), maintaining $q$ and $d_{\mathrm c}$, and testing the inequality using binary arithmetic; all numbers involved have polynomially many bits.

Finally, the numbers $p_1^4, \dotsc, p_n^4$ are pairwise coprime with sum $d_{\mathrm c}$ and product $q^2$, so that \cref{lem:cyclic-group-construction} yields the permutation $c$ in time polynomial in $d_{\mathrm c}$ and, hence, in the input size.
\end{proof}

From now on, we use the numbers $q$, $m$, and $d_{\mathrm c}$ as well as the permutation $c$ on the set $\Omega_{\mathrm c}$ provided by \cref{lem:diameter:cyclic-group}.
We let $c_1 = c$ and $c_3 = c^q$, which generate cyclic groups $C_{q^2} = \langle c_1 \rangle$ and $C_q = \langle c_3 \rangle$ of orders $q^2$ and $q$, respectively.
Now, we define the group $G_1 = G_0 \times C_{q^2}$, which is a permutation group on $\Omega_1 = \Omega_0 \uplus \Omega_{\mathrm c}$ of degree $d_1 = d_0 + d_{\mathrm c}$, as well as
\[
 \Sigma_1 = \{(a,c_1) : a\in\Sigma_0\} \cup \{(g_0^{-1},c_1^q)\} \cup \{(1,1)\},
 \qquad
 g_1 = (1,c_1^{q+k_0}),
 \qquad
 k_1 = k_0 + 1.
\]
Using repeated squaring for the powers $c_1^q$ and $c_1^{q+k_0}$, we can compute $\Sigma_1$ and $g_1$ in polynomial time.

\begin{lemma}\label{lem:diameter:length-gap}
The set $\Sigma_1$ generates $G_1$, and $g_1 \in Z(G_1) \setminus \{1\}$. 
Moreover,
\begin{equation*}
 \ell(g_0,\Sigma_0)\leq k_0 \Longrightarrow \ell(g_1,\Sigma_1)\leq k_1
 \quad\text{and}\quad
 \ell(g_0,\Sigma_0)>k_0 \Longrightarrow \ell(g_1,\Sigma_1)>m.
\end{equation*}
\end{lemma}

\begin{proof}
Since $1 \in \Sigma_0$, we have $(1, c_1) \in \Sigma_1$, so that the subgroup $\langle \Sigma_1 \rangle$ contains $(1, c_1)^{-1}(a, c_1) = (a, 1)$ for every $a \in \Sigma_0$.
Together with $(1, c_1)$, the latter elements generate $G_1$.
The element $g_1$ is central in $G_1$, as its first component is trivial and $C_{q^2}$ is abelian.
Moreover, we have $k_0 < m < q$ by \cref{lem:diameter:cyclic-group}, so that $0 < q + k_0 < 2q < q^2$ and, hence, $g_1 \neq 1$.

Suppose that $\ell(g_0, \Sigma_0) \leq k_0$; that is, $g_0 = a_1 \dotsm a_{k_0}$ for some $a_1, \dotsc, a_{k_0} \in \Sigma_0$.
(Indeed, any shorter word representing $g_0$ can be padded to length exactly $k_0$ using the identity $1 \in \Sigma_0$.)
Then
\[
  (a_1, c_1) \dotsm (a_{k_0}, c_1) \cdot (g_0^{-1}, c_1^q) = (1, c_1^{k_0 + q}) = g_1
\]
is a product of $k_0 + 1 = k_1$ elements of $\Sigma_1$, which shows $\ell(g_1, \Sigma_1) \leq k_1$.

Conversely, suppose that $g_1$ is the value of a word $w$ over $\Sigma_1$ of length at most $m$.
Let $m'$ and~$m''$ be the numbers of letters of $w$ of the form $(a, c_1)$ with $a \in \Sigma_0$ and of letters equal to $(g_0^{-1}, c_1^q)$, respectively.
Projecting onto the second component, we obtain
\[
 m' + q m'' \equiv q + k_0 \pmod{q^2}.
\]
Since $m' + m'' \leq m < q$, both sides of this congruence lie in $[0, q^2 - 1]$, so that equality holds.
Reducing modulo $q$ yields $m' = k_0$, as $m' < q$ and $k_0 < q$, and then $m'' = 1$.
Hence, projecting onto the first component, we obtain $1 = g' g_0^{-1} g''$, where $g'$ and $g''$ are products of altogether $k_0$ elements of $\Sigma_0$.
Thus, $g_0 = g'' g'$ is a product of $k_0$ elements of $\Sigma_0$, which shows $\ell(g_0, \Sigma_0) \leq k_0$.
\end{proof}

\subsection{Step 2: The Wreath Product}
\label{sec:diameter:wreath-product}

In this step, which is both the heart and the most technical part of our reduction, we pass from the group $G_1$ from the previous step to a wreath product $G_2 = G_1 \wr C_r$ with a cyclic group~$C_r$ of polynomial order~$r$.
Throughout, we use the terminology introduced in \cref{sec:preliminaries-groups-1}.

Our generating set~$\Sigma_2$ of~$G_2$ consists of \emph{fillers} and \emph{bridges}.
The fillers generate the coordinate subgroups with small diameter, except at a few \emph{special positions}, where they generate only certain subgroups of~$G_1$ or correspond to the original generating set~$\Sigma_1$.
The bridges have nontrivial shifts and connect the coordinate subgroups at different positions (\cref{lem:diameter:coordinate-completion}), so that every element of~$G_2$ is represented by a short word over~$\Sigma_2$ (\cref{lem:diameter:wreath-generators}).
However, bridges can also produce the value~$g_1$ at the position~$1$ cheaply, regardless of the original instance.
To prevent this in Step~3, we single out special bridges, called \emph{entries} and \emph{exits}, which can produce the value~$g_1$ at the position~$1$ only by means of a word representing~$g_1$ over~$\Sigma_1$ (\cref{lem:diameter:entry-exit-lb}).
Nevertheless, if $\ell(g_0, \Sigma_0) \leq k_0$, then every element of~$G_2$ is still represented by a short word that contains, apart from fillers, either a single bridge or a single entry and a single exit (\cref{pro:diameter:coverage}).

We first choose the special positions in \cref{sec:diameter:positions}, then construct the fillers and bridges in \cref{sec:diameter:generators-1}, and finally introduce the entries and exits in \cref{sec:diameter:generators-2}.

\subsubsection{Special Positions}
\label{sec:diameter:positions}

Our construction distinguishes a few special positions in the group $C_r$: the position~$1 \in C_r$, one \emph{stabilizer position}~$s_k \in C_r$ for every point $\omega_k \in \Omega_1$, and the quotient $x_\ast \in C_r$ of two stabilizer positions.
We choose these positions such that their translates have small overlaps; this will be crucial in the following, as it allows us to keep the number of bridges polynomial (\cref{sec:diameter:generators-1}).

\begin{lemma}\label{lem:diameter:positions}
   Given a number $d \geq 2$, one can compute in time polynomial in $d$ a number $r \leq 10d^3$ and positions $s_1, \dotsc, s_d \in C_r$ such that, for all distinct $i, j \in [1, d]$, the position~$x_\ast = s_i^{-1} s_j$ and the resulting set of special positions $S = \{1, s_1, \dotsc, s_d, x_\ast\}$ satisfy the following properties.
  \begin{enumerate}[label=\textup{(\roman*)}]
    \item The special positions are pairwise distinct, that is, $\abs{S} = d+2$.
    \item We have $S \cap S^{-1} = \{1\}$ and $\abs{S \cap St} \leq 3$ for all $t \in C_r$ with $t \neq 1$.
  \end{enumerate}
\end{lemma}

In the proof of \cref{lem:diameter:positions}, the positions $1, s_1, \dotsc, s_d$ form a Sidon set in $C_r$, meaning that all quotients~$x^{-1}y$ of distinct elements $x, y$ of this set are pairwise distinct or, equivalently, that each translate of this set by some $t \neq 1$ intersects it in at most one element.
We use a simple quadratic construction of such a set; for a survey of constructions of Sidon sets, we refer to~\cite{OBryant2004}.

\begin{proof}
We identify $C_r$ with the additive group of the ring $\Z_r = \Z / r\Z$ of residues modulo~$r$ via the isomorphism that maps a fixed generator of $C_r$ to the residue of~$1$.
Thus, the position $1 \in C_r$ corresponds to $0 \in \Z_r$, inverses $x^{-1}$ to negatives $-x$, and quotients $x^{-1}y$ to differences $y - x$.

Let $\lambda = 2d^2 + 1$, let $a_i = \lambda i + i^2$ for $i \in [0, d]$, and let $r = 2a_d + 1 \leq 10d^3$.
We let $s_i \in \Z_r$ be the residue of $a_i$ modulo~$r$.
Clearly, $r$ and $s_1, \dotsc, s_d$ can be computed in time polynomial in $d$.

The integers $a_j - a_i$ with $i \neq j$ are pairwise distinct.
Indeed, if $a_j - a_i = a_{j'} - a_{i'}$, then 
\[
  \lambda\bigl((j - i) - (j' - i')\bigr) = (j'^2 - i'^2) - (j^2 - i^2),
\]
where the right-hand side has absolute value at most $2d^2 < \lambda$; hence, $j - i = j' - i'$, and dividing the equality $j^2 - i^2 = j'^2 - i'^2$ by this nonzero number yields $j + i = j' + i'$ and, hence, $(i, j) = (i', j')$.
Moreover, since $a_j - a_i \in [-a_d, a_d]$, such differences remain pairwise distinct modulo~$r = 2a_d + 1$, showing that the set $\{0, s_1, \dotsc, s_d\}$ is indeed a Sidon set in the group $\Z_r$.

Let us now fix distinct $i, j \in [1, d]$, which define $x_\ast = s_j - s_i$ as well as $S = \{0, s_1, \dotsc, s_d, x_\ast\}$.
For~(i), note that $0, s_1, \dotsc, s_d$ are already pairwise distinct as $0 = a_0 < a_1 < \dotsb < a_d < r$.
Moreover, we have $x_\ast \neq 0$ as $i \neq j$, and $x_\ast = s_j - s_i$ differs from every $s_k = s_k - 0$ by the Sidon property.

For~(ii), suppose first that both $u$ and $-u$ lie in $S \setminus \{0\} = \{s_1, \dotsc, s_d, x_\ast\}$.
Since $r$ is odd, we have $u \neq -u$.
The two elements $u$ and $-u$ are not both of the form $s_{i'}$ and $s_{j'}$, since $s_{i'} + s_{j'}$ is the residue of $a_{i'} + a_{j'} \in [1, r - 1]$ and, hence, nonzero.
Hence, $\{u, -u\} = \{s_k, x_\ast\}$ for some $k \in [1, d]$, meaning that $s_k - 0 = -x_\ast = s_i - s_j$.
By the Sidon property, the latter implies $(0, s_k) = (s_j, s_i)$, which contradicts $s_j \neq 0$.
Thus, we have $S \cap (-S) = \{0\}$.
Finally, for $t \neq 0$, every $y \in S \cap (S + t)$ yields a pair $(y - t, y)$ of elements of $S$ with difference~$t$; at most one such pair consists of elements of the Sidon set $\{0, s_1, \dotsc, s_d\}$, and at most two such pairs contain the element~$x_\ast = s_j - s_i$.
\end{proof}

We now choose the special positions for the instance $(\Sigma_1, g_1, k_1)$ from Step~1, whose group $G_1$ is a permutation group on~$\Omega_1$ of degree~$d_1$.
To this end, we fix an enumeration $\omega_1, \dotsc, \omega_{d_1}$ of $\Omega_1$.
Since $g_1 \neq 1$ by \cref{lem:diameter:length-gap}, there is a point $\omega_i$ moved by $g_1$, and we let $\omega_j = \omega_i g_1$; in particular, we have $i \neq j$ and $d_1 \geq 2$.
Applying \cref{lem:diameter:positions} with $d = d_1$ yields a number $r \leq 10 d_1^3$ and positions $s_1, \dotsc, s_{d_1} \in C_r$; we call~$s_k$ the \emph{stabilizer position} of the point~$\omega_k$, and write $S' = \{s_1, \dotsc, s_{d_1}\}$.

For the indices $i$ and $j$ chosen above, we let $x_\ast = s_i^{-1} s_j$, and we write $S = \{1, s_1, \dotsc, s_{d_1}, x_\ast\}$ for the set of \emph{special positions}, which has the properties stated in \cref{lem:diameter:positions}.
Note also that $S \subsetneq C_r$, since $S \setminus \{1\}$ and $S^{-1} \setminus \{1\}$ are disjoint by \cref{lem:diameter:positions}.
Moreover, by~\eqref{eqn:diameter:cyclic-group-choices} and $d_1 = d_0 + d_{\mathrm c}$, we have
\begin{equation}\label{eq:r-bound}
  d_1 r + k_1 + 2 \leq 10 d_1^4 + k_0 + 3 \leq m.
\end{equation}
For $t \in C_r$ with $t \neq 1$, we write $I_t = S \cap St^{-1} = \{x \in S : xt \in S\}$, so that $\abs{I_t} \leq 3$ by \cref{lem:diameter:positions}.

\subsubsection{Fillers and Bridges}
\label{sec:diameter:generators-1}

We now pass from $G_1$ to the wreath product $G_2 = G_1 \wr C_r$ with the cyclic group~$C_r$ in which we have chosen the special positions.
As explained in \cref{sec:preliminaries-groups-2}, we regard~$G_2 = G_1 \wr C_r$ as a permutation group on $\Omega_2 = \Omega_1 \times C_r$ of degree $d_2 = d_1 r$.
Note that~$d_2$ is polynomial in the input size of the original instance $(\Sigma_0, g_0, k_0)$ and that, given any $(f, t) \in G_2$ by its values $f(x) \in G_1$ at all positions $x \in C_r$ (as corresponding permutations of $\Omega_1$) and its shift~$t \in C_r$, we can also compute its corresponding permutation of $\Omega_2$ in polynomial time.

Next, we construct a generating set $\Sigma_2 = \Gamma \cup \Gamma_{\bidir}$ of $G_2$, which consists of two types of generators.
The \emph{fillers} in $\Gamma$ have trivial shift and a nontrivial value at most at one position, meaning that each of them belongs to a coordinate subgroup of $G_2$.
In particular, $\Gamma$ generates a subgroup of the base group of $G_2 = G_1 \wr C_r$.
By contrast, the \emph{bridges} in $\Gamma_{\bidir}$ have nontrivial shifts and serve as connections between the different coordinate subgroups.
For the construction in Step~3, we will single out special bridges with further restrictions, called \emph{entries} and \emph{exits}, in \cref{sec:diameter:generators-2}.

\paragraph{Fillers.}
For each position~$x \in C_r$, we choose a set $\Lambda_x \sse G_1$ of \emph{local generators}, depending on the role of~$x$, and we let $H_x = \langle \Lambda_x \rangle \leq G_1$ denote the subgroup that they generate.
The fillers at the position~$x$ are the corresponding elements of the coordinate subgroup at~$x$, that is,
\[
  \Gamma = \bigcup_{x \in C_r} \{ (\delta^x_a, 1) : a \in \Lambda_x \}.
\]
The local generators, and the subgroups they generate, are given as follows.
\begin{itemize}
  \item For the special position $x = 1$, we simply let $\Lambda_x = \{1\}$ and, hence, $H_x = 1$.
  \item For the special position $x = x_\ast$, we let $\Lambda_{x} = \Sigma_1$ and, hence, $H_{x} = G_1$ by \cref{lem:diameter:length-gap}.
  \item For each stabilizer position $x = s_k$, we let $\Lambda_x$ be the generating set from \cref{lem:bsgs-algorithms} of the stabilizer $H_x = \Stab_{G_1}(\omega_k)$ of the corresponding point $\omega_k \in \Omega_1$ with $\diam(H_x, \Lambda_x) \leq d_1$.
  \item For each generic position $x \notin S$, we let $\Lambda_x$ be the generating set from \cref{lem:bsgs-algorithms} of $H_x = G_1$ with $\diam(H_x, \Lambda_x) \leq d_1$, where we use the same set for all generic positions.
\end{itemize}
By \cref{lem:bsgs-algorithms}, the set $\Gamma$ of fillers can be computed in polynomial time.
We also note that for every position $x \neq x_\ast$, we have $\diam(H_x, \Lambda_x) \leq d_1$, so that every element $(\delta^x_h, 1)$ with $h \in H_x$ is a product of at most $d_1$ fillers at the position~$x$.
Finally, for the stabilizer positions $s_i$ and $s_j$ with $x_\ast = s_i^{-1}s_j$ chosen at the end of \cref{sec:diameter:positions}, we have $g_1 \notin H_{s_i}$ and $g_1 \notin H_{s_j}$, since $g_1$ moves both $\omega_i$ and $\omega_j = \omega_i g_1$.
This will become important in the proof of \cref{pro:diameter:coverage}.

\paragraph{Bridges.}
As mentioned above, the bridges connect the coordinate subgroups at different positions. 
To motivate their definition, let us first note that, for all $(f, t) \in G_2$, all $x \in C_r$, and all $g', g'' \in G_1$, we have
\begin{equation}\label{eqn:diameter:bridges-product}
  (\delta^x_{g'}, 1)\,(f, t)\,(\delta^{xt}_{g''}, 1) = (\delta^x_{g'} f \delta^x_{g''}, t).
\end{equation}
Thus, fillers at the positions~$x$ and~$xt$ can multiply the value of $(f, t)$ at~$x$ by elements of~$H_x$ from the left and by elements of~$H_{xt}$ from the right, respectively.
If $t \neq 1$ and $x \notin I_t$, then at least one of the positions $x$ and $xt$ is a generic position, so that every value at~$x$ can be produced in this way, regardless of the value $f(x)$.
For $t \neq 1$ and $x \in I_t$, on the other hand, it suffices to prescribe the possible values $f(x)$ up to these multiplications, for which we use the following representatives. 

For each special position $x \in S \setminus \{1\}$, we fix a left transversal $\mathit{LT}_x$ and a right transversal $\mathit{RT}_x$ of the group $H_x$ in~$G_1$ as in \cref{lem:bsgs-algorithms}, where we assume without loss of generality that both contain~$1$ as the representative of~$H_x$; in particular, we have $\mathit{LT}_{x_\ast} = \mathit{RT}_{x_\ast} = \{1\}$ since $H_{x_\ast} = G_1$.
For the special position~$1$, where $H_1 = 1$, we let $\mathit{LT}_1 = \mathit{RT}_1 = \emptyset$ by convention, instead of using the entire group $G_1$.
For distinct positions $x, y \in S$, we then let
\[
  T_{x,y} = \mathit{RT}_x \cup \mathit{LT}_y.
\]

Note that each of the transversals has at most $d_1$ elements, since the index of a point stabilizer $H_x$ with $x \in S'$ in~$G_1$ is the length of the orbit of the corresponding point; hence, we have $\abs{T_{x,y}} \leq 2d_1$.
Moreover, we have $1 \in T_{x,y}$ for all distinct $x, y \in S$, since $x \neq 1$ or $y \neq 1$.

\medskip

The set $\Gamma_{\bidir}$ of \emph{bridges} is the union of the sets $\Gamma_{\bidir,t}$ with $t \in C_r \setminus \{1\}$, where $\Gamma_{\bidir,t}$ consists of all elements $(h, t) \in G_2$ with shift~$t$ and values $h(x) \in T_{x,xt}$ for $x \in I_t$ and $h(x) = 1$ for $x \notin I_t$.
Since~$\abs{I_t} \leq 3$ by \cref{lem:diameter:positions}, there are at most $8d_1^3$ bridges with shift~$t$ and, hence, at most $8 d_1^3 r$ bridges in total.
By \cref{lem:bsgs-algorithms}, the set $\Gamma_{\bidir}$ can be computed in polynomial time.

\medskip

The following technical lemma makes the role of the bridges precise: up to few fillers on both sides, they provide all elements with a given shift~$t \neq 1$, unless fillers at the position~$x_\ast$ are required.

\begin{lemma}\label{lem:diameter:coordinate-completion}
  Let $(f, t) \in G_2$ with $t \neq 1$.
  Suppose further that there is a set of positions $F \sse C_r$ such that, firstly, $f(x) = 1$ holds for all $x \notin F$ and, secondly, $\{x, xt\} \cap S' \neq \emptyset$ for all $x \in F \cap I_t$.

  Then $(f, t)$ is represented by a word $w \in \Gamma^\ast \, \Gamma_{\bidir} \, \Gamma^\ast$ of length $\abs{w} \leq d_1 \abs{F} + 1 \leq d_1r + 1$.
  Moreover, this word can be chosen such that, firstly, the bridge $(h, t)$ in $w$ has values $h(x) = 1$ for all $x \notin F$, and, secondly, the fillers before and after the bridge in $w$ are at positions in $F$ and $Ft$, respectively.
\end{lemma}

\begin{proof}
  We construct functions $f', h, f'' \colon C_r \to G_1$ with $f(x) = f'(x)\,h(x)\,f''(xt)$ for all $x \in C_r$ position by position.
  For $x \notin F$, we let $f'(x) = h(x) = f''(xt) = 1$.
  For $x \in F$, we distinguish between the following two cases according to the types of the positions $x$ and $xt$.
  \begin{itemize}
    \item If at least one of $x$ and $xt$ is a generic position, meaning that $x \notin S$ or $xt \notin S$, then we let
      \[
        f'(x) = f(x),\, h(x) = f''(xt) = 1
        \quad\text{or}\quad
        f'(x) = h(x) = 1,\, f''(xt) = f(x),
      \]
      respectively.
      (If both $x \notin S$ and $xt \notin S$, then we can choose arbitrarily.)
    \item If $x$ and $xt$ are both special positions, meaning that $x \in F \cap I_t$, then one of $x$ and $xt$ is a stabilizer position by our assumption on $F$; that is, $x \in S'$ or $xt \in S'$.
      Accordingly, we let
      \[
        f'(x)h(x) = f(x),\, f''(xt) = 1
        \quad\text{or}\quad
        f'(x) = 1,\, h(x)f''(xt) = f(x)
      \]
      with $f'(x) \in H_x$ and $h(x) \in \mathit{RT}_x \sse T_{x, xt}$, or with $h(x) \in \mathit{LT}_{xt} \sse T_{x,xt}$ and $f''(xt) \in H_{xt}$, respectively.
      (If both $x \in S'$ and $xt \in S'$, then we can again choose arbitrarily.)
  \end{itemize}
  By construction, we have $h(x) \in T_{x,xt}$ for all $x \in F \cap I_t$ and $h(x) = 1$ for all other positions~$x$; and, since $1 \in T_{x,xt}$ for all $x \in I_t$, the element $(h, t)$ is thus a bridge with $h(x) = 1$ for all $x \notin F$.

  Moreover, for each $x \in F$, at most one of the values $f'(x)$ and $f''(xt)$ is nontrivial, and it belongs to $H_x$ for a generic or stabilizer position~$x$, or to $H_{xt}$ for a generic or stabilizer position~$xt$, respectively.
  In either case, the element $(\delta^x_{f'(x)}, 1)$ or $(\delta^{xt}_{f''(xt)}, 1)$ is a product of at most $d_1$ fillers at the respective position.
  Since fillers at different positions commute, the elements $(f', 1)$ and $(f'', 1)$ are therefore represented by words $w'$ and $w''$ over the fillers at the positions in~$F$ and in~$Ft$, respectively, and their total length can be bounded by $\abs{w'} + \abs{w''} \leq d_1 \abs{F}$.
  Finally, since $(f,t) = (f',1) \,(h,t) \,(f'',1)$ holds by construction, the word $w = w' \, (h,t) \, w''$ has the claimed properties.
\end{proof}

A simple consequence of the above is that the fillers and bridges generate the group~$G_2$.

\begin{lemma}\label{lem:diameter:wreath-generators}
The set $\Sigma_2 = \Gamma \cup \Gamma_{\bidir}$ generates the group $G_2 = G_1 \wr C_r$.
\end{lemma}

\begin{proof}
Let $(f, t) \in G_2$ be arbitrary, and write $t = t' t''$ with $t' \notin \{1, x_\ast, x_\ast^{-1}\}$ and $t'' \neq 1$.
This is possible since $S \subsetneq C_r$ and $\abs{S} = d_1 + 2 \geq 4$ yield $r \geq 5$, so that we can choose $t' \notin \{1, x_\ast, x_\ast^{-1}, t\}$.

We now apply \cref{lem:diameter:coordinate-completion} to the element $(f, t')$ and the set $F = C_r$; its assumption holds, since $\{x, xt'\} \sse S \setminus S' = \{1, x_\ast\}$ would imply $t' \in \{x_\ast, x_\ast^{-1}\}$.
This shows that $(f, t')$ is represented by a word over~$\Sigma_2$.
Moreover, the pure shift $(1, t'')$ is a bridge, since $t'' \neq 1$ and $1 \in T_{x,y}$ for all distinct positions $x, y \in S$.
Hence, the element $(f, t) = (f, t')\,(1, t'')$ is represented by a word over~$\Sigma_2$.
\end{proof}

A close inspection of the above proof reveals that every element of $G_2$ is represented by a short word over $\Sigma_2$ containing exactly two bridges; in particular, $\diam(G_2, \Sigma_2) \leq d_1 r + 2 = d_2 + 2 < m$.
This bound holds regardless of whether the original instance is positive or negative, since bridges with shifts in~$S'$ can produce the value~$g_1$ at position~$1$ cheaply.
The remainder of our reduction is aimed at preventing precisely this shortcut.

\subsubsection{Entries and Exits}
\label{sec:diameter:generators-2}

We now single out the special bridges that are needed for the construction in \cref{sec:diameter:clock}, where they will be treated differently from the other bridges: the \emph{entries} and \emph{exits}, which connect the position~$1$, at which there are no nontrivial fillers, with some other special position $y$.

\medskip

The set $\Gamma_{\fwd}$ of \emph{entries} consists of all bridges $(h, y) \in \Gamma_{\bidir}$ with shift $y \in S \setminus \{1\}$ and the additional restriction that $h(1) \notin g_1H_y$ if $y \in S'$.
Note also that for $y = x_\ast$, the requirement that $(h, y)$ is a bridge already forces $h(1) = 1$, since $1 \in I_{x_\ast}$ and $T_{1, x_\ast} = \{1\}$.
The set $\Gamma_{\bwd}$ of \emph{exits} consists of the elements $(1, y^{-1})$ with $y \in S \setminus \{1\}$, which are bridges since $1 \in T_{x,z}$ for all distinct $x,z \in S$.

\medskip

As the following lemma shows, the restrictions on the entries are chosen such that an entry, followed by fillers and an exit, can produce the value~$g_1$ at position~$1$ only if the entry has shift~$x_\ast$ and there are at least $\ell(g_1, \Sigma_1)$ fillers at the position~$x_\ast$ in between.
Thus, permitting only a single entry and exit~--~instead of two arbitrary bridges~--~closes the shortcut mentioned after \cref{lem:diameter:wreath-generators}.

\begin{lemma}\label{lem:diameter:entry-exit-lb}
  Let $(f, 1) \in G_2$ be represented by a word $w \in \Gamma^\ast \, \Gamma_\fwd \, \Gamma^\ast \, \Gamma_\bwd \, \Gamma^\ast$.
  If $f(1) = g_1$, then the entry in $w$ has shift $x_\ast$, and $w$ contains at least $\ell(g_1, \Sigma_1)$ fillers at the position~$x_\ast$.
\end{lemma}

\begin{proof}
  Let $(h, y) \in \Gamma_\fwd$ and $(1, y^{-1}) \in \Gamma_\bwd$ be the entry and the exit in $w$, which have inverse shifts as the element represented by $w$ has trivial shift.
  Further, let $(f', 1)$, $(f'', 1)$, $(f''', 1)$ be the elements represented by the fillers before the entry, between the entry and the exit, and after the exit, respectively.
  Then these elements satisfy \[
    (f, 1) = (f', 1) \, (h, y) \, (f'', 1) \, (1, y^{-1}) \, (f''', 1),
  \]
  so that $g_1 = f'(1) h(1) f''(y) f'''(1) = h(1) f''(y)$; here, the last equality follows from $f'(1) = f'''(1) = 1$, since there are no nontrivial fillers at the position~$1$.

  Now note that $f''(y)$ is the product of the values of the fillers at the position $y$ between the entry and the exit.
  In particular, $f''(y) \in H_y$ and, in turn, this implies $h(1) = g_1 f''(y)^{-1} \in g_1 H_y$.
  Since $h(1) \notin g_1 H_y$ for $y \neq x_\ast$, we have $y = x_\ast$.
  But then $h(1) = 1$ and, hence, $g_1 = f''(x_\ast)$.
  Finally, since the fillers at the position $x_\ast$ have values in $\Sigma_1$, the word $w$ contains at least $\ell(g_1, \Sigma_1)$ of them.
\end{proof}

On the other hand, words of the same form can produce arbitrary values at all positions other than~$1$, while the value at the position~$1$ is controlled by the entry and by the fillers at two positions.

\begin{lemma}\label{lem:diameter:entry-ub}
  Let $(f, t) \in G_2$ with $t = y z^{-1}$ for some $y, z \in S \setminus \{1\}$.
  Suppose that $f(1) \in \mathit{LT}_y \, g'' g'''$ for some $g'' \in H_y$ and $g''' \in H_{t}$, where we regard $t$ also as a position, and, in the case $y \in S'$, additionally $f(1) \notin g_1 H_y \, g'''$.
  Then the element $(f, t)$ is represented by a word $w \in \Gamma^\ast \, \Gamma_\fwd \, \Gamma^\ast \, \Gamma_\bwd \, \Gamma^\ast$ of length 
  \[
    \abs{w} \leq d_1 (r - 1) + \ell(g'', \Lambda_y) + \ell(g''', \Lambda_{t}) + 2.
  \]
\end{lemma}

\begin{proof}
  Let $(\tilde f, y) \in G_2$ be the element with $\tilde f(1) = 1$ and $\tilde f(x) = f(x)$ for all positions $x \neq 1$.
  We then apply \cref{lem:diameter:coordinate-completion} to $(\tilde f, y)$ with $F = C_r \setminus \{1\}$; its assumption holds, since $\{x, xy\} \sse S \setminus S' = \{1, x_\ast\}$ with $x \neq 1$ would imply $x = x_\ast$ and $xy = 1$, that is, $y = x_\ast^{-1}$, contradicting $S \cap S^{-1} = \{1\}$.
  This shows that $(\tilde f, y)$ is represented by a word $\tilde w \in \Gamma^\ast \, \Gamma_\bidir \, \Gamma^\ast$ of length $\abs{\tilde w} \leq d_1 (r - 1) + 1$ such that, firstly, the bridge $(\tilde h, y)$ in $\tilde w$ has value $\tilde h(1) = 1$ and, secondly, the fillers before and after the bridge are at positions other than~$1$ and other than~$y$, respectively.

  Let $(h, y) \in G_2$ be the element with $h(1) \, g'' g''' = f(1)$ and $h(x) = \tilde h(x)$ for all positions $x \neq 1$.
  Our assumptions on $f(1)$ imply that the element $(h, y)$ is an entry, since $h(1) \in \mathit{LT}_y = T_{1,y}$ and, in the case $y \in S'$, additionally $h(1) \notin g_1 H_y$ as $g'' \in H_y$.
  Further, let $w'$ be the word obtained from $\tilde w$ by replacing $(\tilde h, y)$ with $(h, y)$.
  By construction and the choice of the positions of the fillers in~$\tilde w$, the word $w'$ represents the element $(f', y)$ with $f'(1) = h(1)$ and $f'(x) = f(x)$ for all $x \neq 1$.

  Finally, let $w''$ and $w'''$ be shortest words over the fillers at the positions $y$ and $t$ that represent the elements $(f'', 1)$ and $(f''', 1)$ with $f'' = \smash{\delta^y_{g''}}$ and $f''' = \smash{\delta^{t}_{g'''}}$, respectively.
  Then, by~\eqref{eqn:diameter:bridges-product}, the word $w = w' \, w'' \, (1, z^{-1}) \, w''' \in \Gamma^\ast \, \Gamma_\fwd \, \Gamma^\ast \, \Gamma_\bwd \, \Gamma^\ast$ represents the element $(f, t)$; its length satisfies 
  \[
    \abs{w} = \abs{w'} + \abs{w''} + 1 + \abs{w'''} \leq d_1 (r - 1) + 1 + \ell(g'', \Lambda_y) + 1 + \ell(g''', \Lambda_{t}).\qedhere
  \]
\end{proof}

Finally, we show that every element of~$G_2$ is represented by a short word over~$\Sigma_2$ with either a single entry and a single exit or a single bridge, provided that $\ell(g_0, \Sigma_0) \leq k_0$.

\begin{proposition}\label{pro:diameter:coverage}
  Suppose that $\ell(g_0, \Sigma_0) \leq k_0$.
  Then every element of~$G_2$ is represented by a word 
  \[
    w \in \Gamma^\ast \, \Gamma_\fwd \, \Gamma^\ast \, \Gamma_\bwd \, \Gamma^\ast
    \quad\text{or}\quad 
    w \in \Gamma^\ast \, \Gamma_\bidir \, \Gamma^\ast
  \]
  of length $\abs{w}  \leq d_1 r + k_1 + 2 \leq m$.
\end{proposition}

\begin{proof}
Let $(f, t) \in G_2$ be arbitrary.
In the first three cases below, we apply \cref{lem:diameter:entry-ub} for suitable special positions $y, z \in S \setminus \{1\}$ with $t = y z^{-1}$ and suitable elements $g'' \in H_y$ and $g''' \in H_t$.

\begin{itemize}
  \item Suppose that $t = 1$.
    If $f(1) \neq g_1$, then $g_1^{-1} f(1) \neq 1$ moves some point $\omega_k \in \Omega_1$, which means that $f(1) \notin g_1H_y$ for the stabilizer position $y = s_k$.
  We let $z = y$ and $g''' = 1$, and we choose $g'' \in H_y$ with $f(1) \in \mathit{LT}_y \, g''$, which is possible since $\mathit{LT}_y$ is a left transversal of~$H_y$ in~$G_1$.
  Since we also have $\ell(g'', \Lambda_y) \leq d_1$, \cref{lem:diameter:entry-ub} yields a suitable word $w$ of length 
  \[
    \abs{w} \leq d_1 (r - 1) + d_1 + 2 = d_1 r + 2.
  \]
  In case $f(1) = g_1$, we use \cref{lem:diameter:entry-ub} with $y = z = x_\ast$ as well as $g'' = g_1$ and $g''' = 1$ instead.
  Since $\ell(g_1, \Lambda_{x_\ast}) = \ell(g_1, \Sigma_1) \leq k_1$ by \cref{lem:diameter:length-gap}, the resulting word $w$ has length 
  \[
    \abs{w} \leq d_1 (r - 1) + k_1 + 2.
  \]

  \item Suppose that $t = x_\ast$.
    Let $s_i$ and $s_j$ be the stabilizer positions with $x_\ast = s_i^{-1} s_j = s_j s_i^{-1}$ fixed at the end of \cref{sec:diameter:positions}.
    We intend to apply \cref{lem:diameter:entry-ub} with $y = s_j$ and $z = s_i$.
    To this end, we first choose some $g''' \in \Lambda_{x_\ast}$ with $f(1) \notin g_1 H_{y} \, g'''$: if $f(1) \notin g_1 H_{y}$, then we can take $g''' = 1$, and otherwise we can take any $g''' \in \Lambda_{x_\ast} \setminus H_{y}$, which exists since $\Lambda_{x_\ast}$ generates $G_1$ and $g_1 \notin H_{y}$.
    We then choose $g'' \in H_y$ with $f(1) \in \mathit{LT}_y \, g'' g'''$, which is possible since $\mathit{LT}_y$ is a left transversal of~$H_y$ in~$G_1$.
  Since $\ell(g'', \Lambda_y) \leq d_1$ and $\ell(g''', \Lambda_{x_\ast}) \leq 1$, we can bound the length of the resulting word $w$ as follows:
  \[
    \abs{w} \leq d_1 (r - 1) + d_1 + 1 + 2 = d_1 r + 3.
  \]

  \item Suppose that $t = x_\ast^{-1}$.
    With $s_i$ and $s_j$ as above, we now apply \cref{lem:diameter:entry-ub} with $y = s_i$, $z = s_j$, as well as $g'' = 1$ and $g''' = f(1)$.
    Its assumptions hold, since $f(1) \in \mathit{LT}_y \, g'' g'''$ as $1 \in \mathit{LT}_y$, and since $f(1) \notin g_1 H_y \, g'''$ as $g_1 \notin H_y$.
    Moreover, the position $t = x_\ast^{-1}$ is generic, since $x_\ast^{-1} \notin S$ by \cref{lem:diameter:positions}, so that $g''' \in H_t = G_1$ and $\ell(g''', \Lambda_t) \leq d_1$.
    Since also $\ell(g'', \Lambda_y) = 0$, the resulting word $w$ has length
    \[
      \abs{w} \leq d_1 (r - 1) + d_1 + 2 = d_1 r + 2.
    \]

  \item Finally, suppose that $t \notin \{1, x_\ast, x_\ast^{-1}\}$.
As in the proof of \cref{lem:diameter:wreath-generators}, we can apply \cref{lem:diameter:coordinate-completion} to the element $(f, t)$ with $F = C_r$, which yields a word $w \in \Gamma^\ast \, \Gamma_\bidir \, \Gamma^\ast$ of length $\abs{w} \leq d_1 r + 1$.

\end{itemize}

In all cases, the word~$w$ has one of the two forms in the statement, and we have $\abs{w} \leq d_1 r + k_1 + 2$, since $k_1 = k_0 + 1 \geq 1$.
Finally, we have $d_1 r + k_1 + 2 \leq m$ by~\eqref{eq:r-bound}.
\end{proof}

\subsection{Step 3: The Clock}
\label{sec:diameter:clock}

In this final step, we add the cyclic group $C_q = \langle c_3 \rangle$ (with $q = 4m - 3$ as provided by \cref{lem:diameter:cyclic-group}) as a direct factor.
That is, we let $G_3 = G_2 \times C_{q}$, which is a permutation group on $\Omega_3 = \Omega_2 \uplus \Omega_{\mathrm c}$ of degree $d_3 = d_2 + d_{\mathrm c}$ (see \cref{sec:preliminaries-groups-2}).
The factor $C_q$ serves as a clock.

Our generating set $\Sigma_3$ of $G_3$ consists of \emph{clock variants} $(a, c_3^i)$ of the generators $a \in \Sigma_2$, where the possible \emph{clock values} $i$ depend on the type of~$a$.
More precisely, we let
\[
  \Sigma_3
  = \{ (a, c_3^i) : a \in \Gamma, \, i \in \{0, 1, 2\} \}
  \cup \{ (a, c_3^i) : a \in \Gamma_\fwd \cup \Gamma_\bwd, \, i \in \{0, 1\} \}
  \cup \{ (a, 1) : a \in \Gamma_\bidir \}.
\]
Intuitively, the clock value $2m - 2$ of the central element $g_3$ defined below forces all but two letters of a word of length~$m$ representing~$g_3$ to be fillers with clock value~$2$, and the remaining two letters to be an entry and an exit with clock value~$1$ (\cref{pro:diameter:no}).
On the other hand, the available clock values suffice to use the short words supplied by \cref{pro:diameter:coverage} (\cref{pro:diameter:yes}).

The output of our reduction is the closure $\Delta = \Sigma_3 \cup \Sigma_3^{-1}$ of the set $\Sigma_3$ under taking inverses, together with the threshold~$m$ (in binary).
Both sets $\Sigma_3$ and $\Delta$ can be computed in polynomial time: the set $\Sigma_2 = \Gamma \cup \Gamma_\bidir$ can be computed in polynomial time by \cref{sec:diameter:generators-1}, the entries and exits can be singled out among the bridges (using membership tests for the subgroups $H_y$; see \cref{lem:bsgs-algorithms}) yielding $\Gamma_\fwd, \Gamma_\bwd \sse \Gamma_\bidir$, and the permutation $c_3$ has already been computed in \cref{sec:diameter:length-gap}.

\begin{lemma}\label{lem:diameter:generates}
  The sets $\Sigma_3$ and $\Delta$ generate the group $G_3$.
\end{lemma}

\begin{proof}
  The set $\Sigma_3$ contains the element $(a, 1)$ for every generator $a \in \Sigma_2 = \Gamma \cup \Gamma_\bidir$, and these elements generate the subgroup $G_2 \times \{1\}$ by \cref{lem:diameter:wreath-generators}.
  Moreover, the identity of $G_2$ is a filler, namely the only filler at the position~$1$; hence, we have $(1, c_3) \in \Sigma_3$, and this element generates the subgroup $\{1\} \times C_{q}$.
  Thus, $\Sigma_3 \sse \Delta$ generates $G_3 = G_2 \times C_{q}$.
\end{proof}

In a positive instance, the short words from \cref{pro:diameter:coverage} can be equipped with clock values such that every element of $G_3$ becomes short.

\begin{proposition}\label{pro:diameter:yes}
  If $\ell(g_0,\Sigma_0) \leq k_0$, then $\diam(G_3, \Delta) \leq m$.
\end{proposition}

\begin{proof}
  Let $(g, c_3^i) \in G_3$ be arbitrary, where $g \in G_2$ and $i \in \Z$.
  Since $c_3$ has the odd order $q = 4m - 3$, we may choose $i$ such that $\abs{i} \leq (q - 1)/2 = 2m - 2$.
  Moreover, since $\Delta$ is closed under taking inverses, we may assume that $0 \leq i \leq 2m - 2$ by replacing $(g, c_3^i)$ with its inverse if necessary.

  By \cref{pro:diameter:coverage}, the element $g$ is represented by a word~$w$ of length exactly~$m$ (using the filler $1 \in \Gamma$ for padding) that consists of fillers and either a single entry and a single exit or a single bridge.
  Replacing the letters of~$w$ by suitable clock variants, we may obtain a word of length~$m$ over $\Sigma_3 \sse \Delta$ that represents $(g, c_3^{i'})$ for any $i' \in [0, 2(m - 2) + 2]$ or for any $i' \in [0, 2(m - 1)]$, respectively. 
  Since both intervals equal $[0, 2m - 2]$ and, hence, contain $i$, this shows $\ell((g, c_3^i), \Delta) \leq m$.
\end{proof}

Recall from \cref{sec:preliminaries-groups-1} that we can regard $g_1 \in Z(G_1)$ also as an element $g_2 = (g_1, 1)$ of the diagonal subgroup of $G_2$, and that this element is central in $G_2$.
In the following, we consider the element $g_3 = (g_2, c_3^{2m-2}) \in Z(G_2) \times Z(C_q) = Z(G_3)$.

\begin{proposition}\label{pro:diameter:no}
  If $\ell(g_0,\Sigma_0) > k_0$, then $\ell(g_3, \Delta) > m$ and, hence, $\diam(G_3, \Delta) > m$.
\end{proposition}

\begin{proof}
  By \cref{lem:diameter:length-gap}, we have $g_1 \neq 1$ and $\ell(g_1, \Sigma_1) > m$.
  Suppose, for a contradiction, that $g_3$ is represented by a word over~$\Delta$ of length at most~$m$.
  Since the identity of $G_3$ belongs to $\Sigma_3 \sse \Delta$, we may assume that this word has length exactly~$m$, and we write it as
  \[
    w = (a_1, c_3^{i_1}) \cdots (a_m, c_3^{i_m})
    \quad\text{with}\quad
    a_k \in \Sigma_2 \cup \Sigma_2^{-1}
    \text{ and }
    i_k \in \{-2, \dotsc, 2\},
  \]
  where the clock values $i_k$ are uniquely determined since $q > 4$.
  By the definition of~$\Sigma_3$, we have $a_k \in \Gamma$ if $i_k = 2$, $a_k \in \Gamma^{-1}$ if $i_k = -2$, and $a_k \in \Gamma \cup \Gamma_\fwd \cup \Gamma_\bwd$ if $i_k = 1$.
  Moreover, recall that the fillers and their inverses have trivial shift, whereas all bridges and their inverses have nontrivial shift.

  We first note that at least two letters $a_k$ are not in $\Gamma \cup \Gamma^{-1}$.
  Indeed, the product $a_1 \cdots a_m = g_2$ has the value $g_1 \neq 1$ at the position~$1$, whereas every product of fillers and their inverses has the value~$1$ at the position~$1$, since there are no nontrivial fillers at the position~$1$.
  Moreover, if exactly one letter $a_k$ were not in $\Gamma \cup \Gamma^{-1}$, then the product $a_1 \cdots a_m$ would have the same nontrivial shift as $a_k$, contradicting the fact that $g_2$ has trivial shift.

  Next, comparing the second components of $w$ and $g_3$ yields $i_1 + \dotsb + i_m \equiv 2m - 2 \pmod{q}$.
  Since $\abs{i_1 + \dotsb + i_m} \leq 2m$, $q = 4m - 3$, and $m \geq 2$, this leaves only two possibilities.
  \begin{itemize}
    \item If $i_1 + \dotsb + i_m = 2m - 2 - q = -(2m - 1)$, then $(i_1 + 2) + \dotsb + (i_m + 2) = 1$, where all summands are nonnegative.
      Hence, we have $i_k = -2$ and, thus, $a_k \in \Gamma^{-1}$ for all but one index~$k$, which contradicts the observation above.
    \item If $i_1 + \dotsb + i_m = 2m - 2$, then $(2 - i_1) + \dotsb + (2 - i_m) = 2$, where all summands are nonnegative.
      Hence, either $i_k = 0$ for exactly one index~$k$, or $i_k = 1$ for exactly two indices~$k$, and $i_k = 2$, i.e., $a_k \in \Gamma$, for all other indices.
      The first alternative leaves at most one letter outside $\Gamma \cup \Gamma^{-1}$ and is thus excluded by the observation above.
      In the second alternative, the observation above implies that both letters $a_k$ with $i_k = 1$ lie outside $\Gamma \cup \Gamma^{-1}$, which means that they are entries or exits.
      Since the shifts of these two letters multiply to~$1$ and $S \cap S^{-1} = \{1\}$, one of them is an entry and the other one is an exit.
  \end{itemize}
  Thus, only the second case is possible, and all letters of~$w$ belong to~$\Sigma_3$.

  Finally, since $g_3$ is central in~$G_3$, every cyclic rotation of~$w$ represents $g_3$ as well.
  Thus, every cyclic rotation of the word $a_1 \cdots a_m$ represents $g_2$.
  But some such rotation belongs to $\Gamma^\ast \, \Gamma_\fwd \, \Gamma^\ast \, \Gamma_\bwd \, \Gamma^\ast$ and, thus, contains at least $\ell(g_1, \Sigma_1) > m$ fillers by \cref{lem:diameter:entry-exit-lb}~--~which is a contradiction.
\end{proof}

\begin{proof}[Proof of \cref{thm:diameter-reduction}]
  On input of an instance $(\Sigma_0, g_0, k_0)$ of \BLength, our reduction outputs the instance $(\Delta, m)$ of \BDiameter\ constructed above, which can be computed in polynomial time as argued along the way, and which satisfies $\Delta = \Delta^{-1}$ by definition.
  Moreover, $\Delta$ generates the group $G_3 \cong \bigl((G_0 \times C_{q^2}) \wr C_r\bigr) \times C_q$ by \cref{lem:diameter:generates} and, finally, 
  \[
    \ell(g_0, \Sigma_0) \leq k_0 \iff \diam(\langle\Delta\rangle, \Delta) \leq m
  \]
  holds by \cref{pro:diameter:yes,pro:diameter:no}, which shows correctness of the reduction.
\end{proof}

\section{Three-Step Nilpotent and Metabelian Groups}
\label{sec:length}

In this section, we improve Jerrum's result~\cite{Jerrum1985} on the complexity of \BLength\ as follows; this establishes part~\ref{thm:dichotomy-3step} of our dichotomy in \cref{thm:length-intro}.

\begin{theorem}\label{thm:length:pspace}
  The problem \BLength\ is \PSPACE-complete, even when restricted to groups that are simultaneously $3$-step nilpotent and metabelian.
\end{theorem}

Assuming $\NP \neq \PSPACE$, neither the nilpotency class nor the derived length in \cref{thm:length:pspace} can be lowered.
Indeed, we will show in \cref{sec:nilpotent} that \BLength\ is in~\NP\ for $2$-step nilpotent groups (\cref{thm:nilpotent:length}), and hence also for abelian groups (i.e., groups of derived length at most one).

Together with the reduction from \cref{thm:diameter-reduction}, we also obtain the following statement for the restriction of \BDiameter\ to solvable groups, which is \cref{cor:diameter-derl3-intro}.

\begin{corollary}\label{cor:length:diameter}
  The restriction of \BDiameter\ to solvable groups is \PSPACE-complete, and this holds even when restricted to solvable groups of derived length at most three.
\end{corollary}

Like Jerrum~\cite{Jerrum1985}, we reduce from the acceptance problem of a polynomial-space Turing machine and simulate its computations by products of generators.
However, Jerrum's groups have the symmetric group on the states of the simulated machine as a quotient (\cref{prop:jerrum-quotient}) and are therefore, in general, neither nilpotent nor solvable.
We instead simulate computations by unitriangular $4 \times 4$ matrices, which generate $3$-step nilpotent and metabelian groups.
Since each generator can be applied in every configuration, the group itself cannot prevent incorrect steps; instead, it records a trace of every incorrect step, which is ruled out by our choice of the target.

More precisely, we first represent configurations by registers, one holding the state together with the head position and one for each tape cell holding its symbol, and steps by operations that test some registers for given values and overwrite them (\cref{sec:length:machines}).
We then encode an operation that tests whether a register holds the value~$x$ and writes~$y$ by an upper unitriangular $4 \times 4$ matrix with the following property: multiplying a matrix that encodes the current value~$t$ of the register from the right by this matrix replaces~$t$ by $t + y - x$ and adds the square $(t - x)^2$ to the upper right corner entry (\cref{sec:length:UT,sec:length:reduction}).
The target prescribes the final register values, the number of tests of each register (which we fix by padding with dummy tests), and a vanishing corner entry, which, as a sum of squares, vanishes only if all tests succeed.

Since the entries of these matrices can be exponentially large, we work modulo polynomially many small primes simultaneously, which keeps the degree polynomial, whereas the product of the primes exceeds all values that can occur.
Finally, we prove \cref{thm:length:pspace} and \cref{cor:length:diameter} in \cref{sec:length:proofs}, where we also show that hardness holds for symmetric generating sets.

\subsection{Unitriangular Groups}
\label{sec:length:UT}

We simulate the register operations introduced in \cref{sec:length:machines} by upper unitriangular matrices over finite fields.
For an odd prime~$p$, we write $\F_p$ for the field with $p$ elements and
\[
  U_p = \UT_4(\F_p) = \{ 1 + X : X \in \F_p^{4 \times 4} \text{ is strictly upper triangular} \}
\]
for the group of upper unitriangular $4 \times 4$ matrices over~$\F_p$, that is, of upper triangular matrices whose diagonal entries are all equal to~$1$.
Note that $U_p$ is indeed a group, since every strictly upper triangular $4 \times 4$ matrix~$X$ satisfies $X^4 = 0$ and, hence, $(1 + X)^{-1} = 1 - X + X^2 - X^3$.

\begin{lemma}\label{lem:length:matrix-group}
For every odd prime $p$, the group $U_p$ is $3$-step nilpotent and metabelian.
\end{lemma}

\begin{proof}
  The $3$-step nilpotency of~$U_p$ is a special case of the standard result for unitriangular groups; see, e.g., \cite[p.~127]{Robinson1996}.
  The fact that $U_p$ is metabelian can be found, e.g., in~\cite{Glasby1999}.
\end{proof}

We will use upper unitriangular matrices of a particular form.
For $t, \lambda, \mu, \nu \in \Z$, let
\[
 M(t,\lambda,\mu,\nu)=
 \begin{pmatrix}
  1&t&t^2&\nu\\
  0&1&2t&\mu\\
  0&0&1&\lambda\\
  0&0&0&1
 \end{pmatrix}.
\]
For an odd prime~$p$, we write $M_p(t, \lambda, \mu, \nu) \in U_p$ for the matrix that is obtained by reducing all entries of $M(t, \lambda, \mu, \nu)$ modulo~$p$.
Products of such matrices are again of this form, which allows us to keep track of products of generators over the integers.

\begin{lemma}\label{lem:length:matrix-product}
For all $t, t', \lambda, \lambda', \mu, \mu', \nu, \nu' \in \Z$, we have
\[
 M(t,\lambda,\mu,\nu)\,M(t',\lambda',\mu',\nu') = M\bigl(t+t',\,\lambda+\lambda',\,\mu+\mu'+2t\lambda',\,\nu+\nu'+t\mu'+t^2\lambda'\bigr).
\]
In particular, for all $x,y\in\Z$,
\[
 M(t,\lambda,\mu,\nu)\,M(y-x,1,-2x,x^2) = M\bigl(t+y-x,\,\lambda+1,\,\mu+2(t-x),\,\nu+(t-x)^2\bigr).
\]
\end{lemma}

\begin{proof}
  The first equality follows by direct matrix multiplication.
  The second one is the special case $t' = y - x$, $\lambda' = 1$, $\mu' = -2x$, and $\nu' = x^2$, since $-2x + 2t = 2(t - x)$ and $x^2 - 2tx + t^2 = (t - x)^2$.
\end{proof}

Since $M(0, 0, 0, 0) = 1$, \cref{lem:length:matrix-product} also shows that $M(t, 0, 0, 0)^{-1} = M(-t, 0, 0, 0)$ for all $t \in \Z$, and the same holds for the matrices $M_p(t, 0, 0, 0)$ with an odd prime~$p$.

\subsection{Register Machines}
\label{sec:length:machines}

For the remainder of this section, we fix a \PSPACE-complete language and a deterministic one-tape Turing machine~$\mathcal M$ with state set~$Q$ and tape alphabet~$A$ deciding it in polynomial space.
Moreover, we fix an input word~$w$ of length~$n$ for the machine~$\mathcal M$.
Let $s(n) \geq \max(n, 1)$ be a polynomial space bound for~$\mathcal M$ that is computable in time polynomial in~$n$, so that the computation of~$\mathcal M$ on every input of length~$n$ only visits the tape cells $1, \dotsc, s(n)$, and write $s = s(n)$.
Further, we assume that~$\mathcal M$ erases its tape and moves its head to cell~$1$ before entering its unique accepting state, so that there is a unique accepting configuration for every input length.
This can be achieved by adding a cleanup phase and increasing the polynomial space bound (e.g., the right boundary of the permitted tape segment can be marked in an initialization phase).

\paragraph{Registers and Operations.}
We simulate~$\mathcal M$ by registers indexed by $I = [0, s]$, where register~$0$ holds the current state together with the head position, and register $i \in [1, s]$ holds the symbol in tape cell~$i$.
In particular, register~$0$ takes values in the set $D_0 = Q \times [1, s]$, and register~$i$ takes values in $D_i = A$ for $i \in [1, s]$.
Thus, the configurations of~$\mathcal M$ using only the tape cells $1, \dotsc, s$ correspond to the tuples $\mathbf t = (t_i)_{i \in I}$ with $t_i \in D_i$ for all $i \in I$, which we call \emph{register tuples}.
We write $\mathbf u$ and $\mathbf v$ for the register tuples of the initial configuration on~$w$ and of the accepting configuration, respectively.
Clearly, both $\mathbf u$ and $\mathbf v$ can be computed in polynomial time from~$w$.

An \emph{operation}~$\theta$ tests a set $I_\theta \sse I$ of registers for values $x_{\theta,i} \in D_i$ and writes values $y_{\theta,i} \in D_i$ for $i \in I_\theta$.
More precisely, $\theta$ is \emph{applicable} to a register tuple~$\mathbf t$ if $t_i = x_{\theta,i}$ for all $i \in I_\theta$, and it then transforms~$\mathbf t$ into the register tuple~$\mathbf t'$ with $t'_i = y_{\theta,i}$ for $i \in I_\theta$ and $t'_i = t_i$ for $i \notin I_\theta$.
A sequence of operations is \emph{applicable} to a register tuple~$\mathbf t$ if each of its operations is applicable to the register tuple obtained from~$\mathbf t$ by the preceding operations; it then \emph{transforms}~$\mathbf t$ into the register tuple obtained by all of its operations.
Moreover, the sequence \emph{tests} a register $i \in I$ exactly $m$ times if $i \in I_\theta$ holds for exactly $m$ of its operations~$\theta$, counted with multiplicity.
We use the following operations.
\begin{itemize}
  \item For every transition of~$\mathcal M$ that, in state~$q$ reading symbol~$a$, enters state~$q'$, writes symbol~$a'$, and moves the head by $e \in \{-1, 0, 1\}$, and for every head position $i \in [1, s]$ with $i + e \in [1, s]$, there is an operation~$\theta$ with $I_\theta = \{0, i\}$ and
  \[
    x_{\theta,0} = (q, i), \quad x_{\theta,i} = a, \quad y_{\theta,0} = (q', i + e), \quad y_{\theta,i} = a'.
  \]
  We write $\Theta_{\mathcal M}$ for the set of these operations, which can be computed in time polynomial in~$n$.
  \item For every register $i \in I$, there is an operation $\varepsilon_i$ with $I_{\varepsilon_i} = \{i\}$ and $x_{\varepsilon_i, i} = y_{\varepsilon_i, i} = v_i$; that is, $\varepsilon_i$ tests whether register~$i$ holds the value~$v_i$ and leaves all registers unchanged.
\end{itemize}
We let $\Theta = \Theta_{\mathcal M} \uplus \{\varepsilon_i : i \in I\}$.
By construction, an operation from~$\Theta_{\mathcal M}$ is applicable to the register tuple of a configuration if and only if it corresponds to the step that $\mathcal M$ performs in this configuration, and it then transforms this register tuple into that of the successor configuration.
The operations $\varepsilon_i$, on the other hand, do not change any register tuple at all.

\paragraph{Computations.}
Finally, let
\[
  N = \abs{Q} \cdot s \cdot \abs{A}^s
  \quad\text{and}\quad
  k = (s+2) \cdot N.
\]
Here, $N$ is the number of configurations of~$\mathcal M$ using only the tape cells $1, \dotsc, s$.
As $\mathcal M$ is deterministic and halts on~$w$, its computation does not repeat a configuration and, hence, consists of fewer than~$N$ steps.
Note that the binary representations of $N$ and~$k$ have length polynomial in~$n$.

\begin{lemma}\label{lem:length:simulation}
If $\mathcal M$ accepts~$w$, then there is an applicable sequence of at most $k$ operations from $\Theta$ that transforms $\mathbf u$ into $\mathbf v$ and tests every register exactly $N$ times. Conversely, every applicable sequence of operations from $\Theta$ that transforms $\mathbf u$ into $\mathbf v$ yields an accepting computation of $\mathcal M$ on~$w$.
\end{lemma}

\begin{proof}
  Suppose that $\mathcal M$ accepts~$w$ after $m < N$ steps, and, for $i \in I$, let $m_i \leq m$ be the number of these steps whose corresponding operations from~$\Theta_{\mathcal M}$ test register~$i$.
  The sequence of these operations is applicable to~$\mathbf u$ and transforms it into~$\mathbf v$.
  Appending $N - m_i$ copies of the operation $\varepsilon_i$ for every $i \in I$ yields a sequence that is still applicable to~$\mathbf u$, since none of the appended operations changes the register tuple~$\mathbf v$, and this sequence tests every register exactly $N$ times.
  Its length is 
  \[
    m + \sum_{i \in I} (N - m_i) \leq N + (s + 1) N = k.
  \]

  Conversely, deleting all operations $\varepsilon_i$ from an applicable sequence of operations from~$\Theta$ that transforms~$\mathbf u$ into~$\mathbf v$ yields an applicable sequence of operations from~$\Theta_{\mathcal M}$ that still transforms~$\mathbf u$ into~$\mathbf v$.
  By construction, the latter sequence corresponds to a computation of~$\mathcal M$ from the initial configuration on~$w$ to the accepting configuration.
\end{proof}

\subsection{The Reduction}
\label{sec:length:reduction}

We now simulate the register machine from \cref{sec:length:machines} by products of the matrices from \cref{sec:length:UT}.
To this end, we first fix injective maps from the sets $D_i$ into the interval $[0, N - 1]$ for all $i \in I$, which exist since $\abs{D_i} \leq N$ and which can clearly be computed in polynomial time.
From now on, we identify the values of registers with their images under these maps, so that the entries of all register tuples as well as all values $x_{\theta,i}$ and $y_{\theta,i}$ of operations $\theta \in \Theta$ are integers in $[0, N - 1]$.

Since the matrices from \cref{sec:length:UT} that we use here can have exponentially large entries, we work modulo many small primes simultaneously; the use of many small primes to obtain a large modulus also appears in Jerrum's construction~\cite[Lem.~2.1]{Jerrum1985}.
More precisely, let $B = k^3 N^2$, and let $\mathcal P$ be the set of the first $\Abs{B}$ odd primes; note that $\Abs{B}$ is polynomial in~$n$.
By \cref{lem:primes}, the product $P$ of these primes satisfies $P > B$, and the set~$\mathcal P$ can be computed in time polynomial in~$n$.

We work in the direct product 
\[
  U = \prod_{(i,p) \in I \times \mathcal P} U_p,
\]
which is $3$-step nilpotent and metabelian by \cref{lem:length:matrix-group}, since these classes are closed under taking finite direct products (\cref{sec:preliminaries-groups-1}).

\begin{lemma}\label{lem:length:permutation-group}
The group $U$ can be efficiently realized as a permutation group of degree $(s+1)\sum_{p\in\mathcal P}p^4$.
\end{lemma}

\begin{proof}
  For every odd prime~$p$, the group~$U_p$ acts on the set~$\F_p^4$ of row vectors by multiplication from the right.
  This action is faithful, so that we can regard~$U_p$ as a permutation group on~$\F_p^4$ of degree~$p^4$.
  Hence, the direct product~$U$ is a permutation group on the disjoint union~$\Omega$ of the sets~$\F_p^4$ for all $(i, p) \in I \times \mathcal P$ (see \cref{sec:preliminaries-groups-2}), and its degree is $d = (s + 1) \sum_{p \in \mathcal P} p^4$.
  
  Given $\mathcal P$, the set~$\Omega$ can be computed in time polynomial in~$d$, and, for every element of~$U$ given by its components, the corresponding permutation of~$\Omega$ can be computed in time polynomial in~$d$ by multiplying, for each $(i, p) \in I \times \mathcal P$, all $p^4$ vectors in~$\F_p^4$ with the component at $(i, p)$.
\end{proof}

For each operation $\theta \in \Theta$, we define a generator $a_\theta \in U$ whose component at $(i, p) \in I \times \mathcal P$ is
\[
  (a_\theta)_{(i,p)} =
  \begin{cases}
    M_p(y_{\theta,i} - x_{\theta,i},\, 1,\, -2x_{\theta,i},\, x_{\theta,i}^2) & \text{if } i \in I_\theta, \\
    1 & \text{otherwise}.
  \end{cases}
\]
By \cref{lem:length:matrix-product}, multiplying a matrix $M_p(t, \lambda, \mu, \nu)$ from the right with the component of~$a_\theta$ at $(i, p)$ for $i \in I_\theta$ replaces the value~$t$ by $t + y_{\theta,i} - x_{\theta,i}$, which is $y_{\theta,i}$ if $t = x_{\theta,i}$, and adds $(t - x_{\theta,i})^2$ to the corner entry~$\nu$.
We let $\Sigma = \{a_\theta : \theta \in \Theta\}$, which generates a proper subgroup of~$U$ in general.

We define our target $g \in U$ based on the register tuples $\mathbf{u}$ and $\mathbf{v}$, which correspond to the initial and accepting configuration of $\mathcal M$, respectively.
To this end, we let 
\[
  g_- = \bigl(M_p(u_i, 0, 0, 0)\bigr)_{(i,p) \in I \times \mathcal P}
  \quad\text{and}\quad
  g_+ = \bigl(M_p(v_i, N, 0, 0)\bigr)_{(i,p) \in I \times \mathcal P},
\]
and then define
\[
  g = g_-^{-1} g_+ = \bigl(M_p(v_i - u_i,\, N,\, -2Nu_i,\, Nu_i^2)\bigr)_{(i,p) \in I \times \mathcal P},
\]
where the second equality follows from \cref{lem:length:matrix-product}.
Note that $g$ need not belong to the group generated by~$\Sigma$ in general; we deal with this issue in \cref{sec:length:proofs}.

\begin{proposition}\label{pro:length:correctness}
The machine $\mathcal M$ accepts its input $w$ if and only if the element $g \in U$ is represented by a word over~$\Sigma$ of length at most~$k$.
\end{proposition}

\begin{proof}
  Let $\theta_1, \dotsc, \theta_m \in \Theta$ be a sequence of $m \leq k$ operations.
  We first compute the corresponding products $g_- a_{\theta_1} \dotsm a_{\theta_j}$ for $j \in [0, m]$ over the integers and, to this end, we define $t_{i,j}, \lambda_{i,j}, \mu_{i,j}, \nu_{i,j} \in \Z$ for $i \in I$ and $j \in [0, m]$ as follows.
  We let $t_{i,0} = u_i$ and $\lambda_{i,0} = \mu_{i,0} = \nu_{i,0} = 0$, and 
  \begin{equation}\label{eqn:length:updates}
  \begin{aligned}
    t_{i,j+1} &= t_{i,j} + y_{\theta_{j+1},i} - x_{\theta_{j+1},i}, &
    \mu_{i,j+1} &= \mu_{i,j} + 2(t_{i,j} - x_{\theta_{j+1},i}), \\
    \lambda_{i,j+1} &= \lambda_{i,j} + 1, &
    \nu_{i,j+1} &= \nu_{i,j} + (t_{i,j} - x_{\theta_{j+1},i})^2,
  \end{aligned}
  \end{equation}
  for $j \in [0, m - 1]$ and $i \in I_{\theta_{j+1}}$, whereas all four integers remain unchanged for $i \notin I_{\theta_{j+1}}$.

  By \cref{lem:length:matrix-product} and since reducing the entries of integer matrices modulo~$p$ is compatible with matrix multiplication, induction on~$j$ yields
  \begin{equation}\label{eqn:length:invariant}
    g_- a_{\theta_1} \dotsm a_{\theta_j} = \bigl(M_p(t_{i,j}, \lambda_{i,j}, \mu_{i,j}, \nu_{i,j})\bigr)_{(i,p) \in I \times \mathcal P}
  \end{equation}
  for all $j \in [0, m]$.
  Here, $\lambda_{i,m}$ is the number of operations in the sequence that test register~$i$, and
  \begin{equation}\label{eqn:length:sum-of-squares}
    \nu_{i,m} = \sum_{j} (t_{i,j} - x_{\theta_{j+1},i})^2,
  \end{equation}
  where the sum ranges over all $j \in [0, m - 1]$ with $i \in I_{\theta_{j+1}}$.
  Induction on~$j$ shows that the sequence is applicable to~$\mathbf u$ if and only if $t_{i,j} = x_{\theta_{j+1},i}$ for all $j \in [0, m - 1]$ and $i \in I_{\theta_{j+1}}$; in this case, the register tuple obtained from~$\mathbf u$ by the first $j$ operations is $(t_{i,j})_{i \in I}$, and $\mu_{i,m} = \nu_{i,m} = 0$ for all $i \in I$.

  The integers $t_{i,j}$ may leave the interval $[0, N - 1]$ if the sequence is not applicable to~$\mathbf u$, but they cannot grow too fast. 
  Indeed, by~\eqref{eqn:length:updates}, each operation changes them by at most $N - 1$.
  Since, moreover, $u_i$, $v_i$, and all values $x_{\theta,i}$ lie in $[0, N - 1]$, and since $m \leq k$, each of the at most $k$ summands in~\eqref{eqn:length:sum-of-squares} is thus less than $(kN)^2$, and $t_{i,m}$ differs from~$v_i$ by at most $(k + 1)(N - 1)$.
  As $B = k^3 N^2$ and $P > B$, we obtain
  \begin{equation}\label{eqn:length:bounds}
    \abs{t_{i,m} - v_i} < P
    \quad\text{and}\quad
    0 \leq \nu_{i,m} < P
    \quad\text{for all } i \in I.
  \end{equation}

  Now suppose that $\mathcal M$ accepts its input $w$.
  By \cref{lem:length:simulation}, there is an applicable sequence $\theta_1, \dotsc, \theta_m$ of $m \leq k$ operations that transforms~$\mathbf u$ into~$\mathbf v$ and tests every register exactly $N$ times.
  By the above, we then have $t_{i,m} = v_i$, $\lambda_{i,m} = N$, and $\mu_{i,m} = \nu_{i,m} = 0$ for all $i \in I$, so that \eqref{eqn:length:invariant} yields $g_- a_{\theta_1} \dotsm a_{\theta_m} = g_+$.
  Hence, $g = a_{\theta_1} \dotsm a_{\theta_m}$ is represented by a word over~$\Sigma$ of length $m \leq k$.

  Conversely, suppose that $g = a_{\theta_1} \dotsm a_{\theta_m}$ for operations $\theta_1, \dotsc, \theta_m \in \Theta$ with $m \leq k$.
  We then have $g_- a_{\theta_1} \dotsm a_{\theta_m} = g_+$.
  Comparing the entries of the components at $(i, p) \in I \times \mathcal P$ in~\eqref{eqn:length:invariant} with those of~$g_+$ yields $t_{i,m} \equiv v_i \pmod p$ and $\nu_{i,m} \equiv 0 \pmod p$.
  Since the primes in~$\mathcal P$ are distinct, both congruences also hold modulo~$P$, and \eqref{eqn:length:bounds} yields $t_{i,m} = v_i$ and $\nu_{i,m} = 0$ for all $i \in I$.
  Moreover, since $\nu_{i,m}$ is a sum of squares by~\eqref{eqn:length:sum-of-squares}, all of its summands vanish; that is, we have $t_{i,j} = x_{\theta_{j+1},i}$ for all $j \in [0, m - 1]$ and $i \in I_{\theta_{j+1}}$.
  Hence, the sequence $\theta_1, \dotsc, \theta_m$ is applicable to the register tuple~$\mathbf u$ and transforms it into $\mathbf v$.
  By \cref{lem:length:simulation}, the machine~$\mathcal M$ therefore accepts its input.
\end{proof}

\subsection{Proofs of \texorpdfstring{\cref{thm:length:pspace} and \cref{cor:length:diameter}}{Theorem~\ref{thm:length:pspace} and Corollary~\ref{cor:length:diameter}}}
\label{sec:length:proofs}

The upper bound in \cref{thm:length:pspace} holds for arbitrary permutation groups, as already shown by Jerrum~\cite{Jerrum1985}; we include the short argument for completeness.

\begin{proof}[Proof of \cref{thm:length:pspace}]
  For membership in~\PSPACE, let $(\Sigma, g, k)$ be an instance of \BLength.
  Our algorithm first checks whether $\Sigma$ generates a group that is $3$-step nilpotent and metabelian, which is possible in polynomial time by \cref{lem:nilpotent-solvable-algorithms}, and rejects otherwise.
  It then guesses a word over~$\Sigma$ of length at most~$k$ letter by letter, storing the product of the letters guessed so far and their number in binary; it accepts if and only if this product is~$g$.
  Since this nondeterministic algorithm runs in polynomial space, Savitch's theorem yields membership in~\PSPACE.

  For \PSPACE-hardness, we reduce the \PSPACE-complete language fixed in \cref{sec:length:machines} to the restriction of \BLength\ to $3$-step nilpotent, metabelian groups.
  Given an input word~$w$ of length~$n$, our reduction computes the register tuples $\mathbf u$ and~$\mathbf v$, the set~$\Theta$ of operations, the binary representations of $N$ and~$k$, and the set~$\mathcal P$ of primes from \cref{sec:length:machines,sec:length:reduction}, all in time polynomial in~$n$.
  By \cref{lem:length:permutation-group}, the group~$U$ can be efficiently realized as a permutation group of degree $(s + 1) \sum_{p \in \mathcal P} p^4$, which is polynomial in~$n$ by \cref{lem:primes}.
  Hence, the generators~$a_\theta$ and the target~$g$ can be computed as permutations in time polynomial in~$n$: we obtain their components by reducing the entries of integer matrices with polynomially many bits modulo the primes $p \in \mathcal P$, and then compute the corresponding permutations as in \cref{lem:length:permutation-group}.
  The group $G$ generated by $\Sigma$ is $3$-step nilpotent and metabelian as a subgroup of~$U$ (\cref{sec:preliminaries-groups-1}), but the target~$g$ need not belong to~$G$.
  Our reduction therefore tests whether $g \in G$, which is possible in polynomial time by \cref{lem:bsgs-algorithms}.
  If so, it outputs the instance $(\Sigma, g, k)$, which is positive if and only if $\mathcal M$ accepts~$w$ by \cref{pro:length:correctness}.
  Otherwise, $g$ is not represented by any word over~$\Sigma$, so that $\mathcal M$ rejects~$w$ by \cref{pro:length:correctness}. 
  The reduction then outputs a fixed negative instance whose group is abelian, such as the instance $(\{(1, 2)\}, (1, 2), 0)$ with the transposition $(1, 2) \in \Sym(\{1, 2\})$.
\end{proof}

\begin{remark}\label{rem:length:symmetric}
  \PSPACE-hardness in \cref{thm:length:pspace} holds even for symmetric generating sets.
  Indeed, let $(\Sigma, g, k)$ be an instance of \BLength, where $\Sigma$ generates the group~$G$; without loss of generality, $1 \in \Sigma$.
  Let $q$ be the product of the first $\Abs{2k}$ odd primes, so that $q > 2k$, and let $c$ be a permutation of order~$q$ from \cref{lem:cyclic-group-construction}.
  By \cref{lem:cyclic-group-construction,lem:primes}, the permutation $c$ can be computed in polynomial time, and so can $c^k$ by repeated squaring.
  Let 
  \[
    \Sigma_c = \{ (a,c) : a \in \Sigma\}
    \quad\text{and}\quad 
    \Delta = \Sigma_c \cup \Sigma_c^{-1},
  \]
  which generates the group $G \times C_q$; the latter is $3$-step nilpotent and metabelian whenever $G$ is.

  Then $\ell(g, \Sigma) \leq k$ holds if and only if $\ell((g, c^k), \Delta) \leq k$.
  Indeed, if $\ell(g, \Sigma) \leq k$, then $g = a_1 \dotsm a_k$ with $a_1, \dotsc, a_k \in \Sigma$ since $1 \in \Sigma$; hence, $(g, c^k) = (a_1, c) \dotsm (a_k, c)$, so that $\ell((g, c^k), \Delta) \leq k$.
  Conversely, since $2k < q$, the target $(g, c^k)$ can be represented by a word of length at most~$k$ over~$\Delta$ only if all its letters belong to~$\Sigma_c$, and projecting onto the direct factor~$G$ then shows $\ell(g, \Sigma) \leq k$.
\end{remark}

\begin{proof}[Proof of \cref{cor:length:diameter}]
  Membership in~\PSPACE\ follows from \cref{thm:diameter-intro} and \cref{lem:nilpotent-solvable-algorithms}.
  For hardness, we compose the reduction from the proof of \cref{thm:length:pspace} with the reduction from \cref{thm:diameter-reduction}.
  The former outputs instances of \BLength\ with metabelian groups $G_0$, so that the resulting groups $\bigl((G_0 \times C_{q^2}) \wr C_r\bigr) \times C_q$ have derived length at most three (\cref{sec:preliminaries-groups-1}).
\end{proof}

\section{Two-Step Nilpotent Groups}
\label{sec:nilpotent}

In this section, we consider the restriction of \BLength\ to $2$-step nilpotent groups and prove the following result, which constitutes the remaining part of our dichotomy in \cref{thm:length-intro}.

\begin{theorem}\label{thm:nilpotent:length}
  The restriction of \BLength\ to $2$-step nilpotent groups is \NP-complete.
\end{theorem}

As an immediate consequence, we also obtain the upper bound in the following statement for the restriction of \BDiameter\ to $2$-step nilpotent groups, which is \cref{cor:diameter-2step-intro}.

\begin{corollary}\label{cor:nilpotent:diameter}
  The restriction of \BDiameter\ to $2$-step nilpotent groups is $\Pi_2^{\P}$-complete.
\end{corollary}

Both results extend the corresponding completeness results for abelian groups by the first author and Rosowski~\cite{LohreyRosowski2023}, from which the lower bounds carry over.
Our algorithms for the new upper bounds rest on the observation that, in a $2$-step nilpotent group, the element represented by a word is determined by its \emph{$2$-profile}, which records the numbers of occurrences of all letters and of all scattered subwords of length two (\cref{lem:nilpotent:profile-evaluation}).
Thus, instead of guessing a short word representing the target, it suffices to guess the $2$-profile of such a word in binary and to verify that some word realizes this profile.
Since words realizing a given profile can be exponentially long, the main technical ingredient is that the latter problem, \textsc{Binary $2$-Realization}, belongs to \NP\ (\cref{pro:nilpotent:profile-realization}), which may be of independent interest.

\subsection{Subword Profiles}
\label{sec:nilpotent:profiles}

A word $u \in \Sigma^\ast$ is a (scattered) \emph{subword} of a word $w \in \Sigma^\ast$ if it can be obtained from~$w$ by deleting some of its letters.
For instance, $u = aba$ is a subword of $w = \underline{a}b\underline{b}a\underline{a}$, whereas $u' = bab$ is not.

We are interested in the number $\tbinom{w}{u}$ of \emph{occurrences} of~$u$ as a subword of~$w$, that is, of sequences of positions $1 \leq i_1 < \dotsb < \smash{i_{\abs{u}}} \leq \abs{w}$ with $w[i_1] \dotsm w[i_{\abs{u}}] = u$.
Thus, the empty word~$\varepsilon$ occurs exactly once in every word $w \in \Sigma^\ast$, and $\tbinom{w}{u} = 4$ for $u = aba$ and $w = abbaa$ as above.
Splitting an occurrence of~$u$ in a product $w_1 w_2$ into the positions inside~$w_1$ and those inside~$w_2$ yields the well-known identity
\begin{equation}\label{eqn:nilpotent:subword-product}
  \mbinom{w_1 w_2}{u} = \sum_{u = u_1 u_2} \mbinom{w_1}{u_1} \mbinom{w_2}{u_2},
\end{equation}
where the sum ranges over all factorizations $u = u_1 u_2$ with $u_1, u_2 \in \Sigma^\ast$.

For $k \in \N$, we write $\Sigma^{\leq k}$ for the set of all words $u \in \Sigma^\ast$ of length $\abs{u} \leq k$, and $\Psi_k(\Sigma)$ for the set of all functions $\nu \colon \Sigma^{\leq k} \to \N$.
The latter set becomes a monoid with respect to the product
\begin{equation}\label{eqn:nilpotent:profile-composition}
  (\nu_1 \cdot \nu_2)(u) = \sum_{u = u_1 u_2} \nu_1(u_1) \, \nu_2(u_2),
\end{equation}
where, as in~\eqref{eqn:nilpotent:subword-product}, the sum ranges over all factorizations $u = u_1 u_2$ with $u_1, u_2 \in \Sigma^\ast$.
Its neutral element maps the empty word~$\varepsilon$ to~$1$ and all nonempty words to~$0$.

\medskip

For a word $w \in \Sigma^\ast$, we call the function $\psi_k(w) \in \Psi_k(\Sigma)$ given by $(\psi_k(w))(u) = \tbinom{w}{u}$ the \emph{$k$-profile} of~$w$, and we say that $w$ \emph{realizes} $\psi_k(w)$.
By~\eqref{eqn:nilpotent:subword-product}, the map $\psi_k \colon \Sigma^\ast \to \Psi_k(\Sigma)$ is a homomorphism.

Our interest in this statistic originates from the following fact, first observed by Magnus~\cite{Magnus1937}: for every $k$-step nilpotent group~$G$ and every homomorphism $\varphi \colon \Sigma^\ast \to G$, the image $\varphi(w)$ of a word $w \in \Sigma^\ast$ is completely determined by its $k$-profile $\psi_k(w)$; that is, $\psi_k(w) = \psi_k(w')$ implies $\varphi(w) = \varphi(w')$.
In fact, Magnus~\cite{Magnus1937} proved that two words have the same $k$-profile if and only if they cannot be separated by homomorphisms to $k$-step nilpotent groups.
Here, we only use the following explicit version of the much simpler case $k = 2$.

\begin{lemma}\label{lem:nilpotent:profile-evaluation}
Let $\varphi\colon\Sigma^\ast\to G$ be a homomorphism to a $2$-step nilpotent group $G$, and let $<$ be a linear order on $\Sigma$. Then every $w\in\Sigma^\ast$ with $\nu = \psi_2(w)$ satisfies
\[
 \varphi(w) = \prod_{a}\varphi(a)^{\nu(a)}\cdot\prod_{a<b}[\varphi(b),\varphi(a)]^{\nu(ba)},
\]
where the first product runs over the elements of $\Sigma$ in increasing order.\footnote{The factors of the second product are central in~$G$, so that their order is irrelevant.}
\end{lemma}

\begin{proof}
  Let $w'$ be the word that is obtained by sorting the letters of~$w$ in increasing order; that is, $w' = \prod_a a^{\nu(a)}$ with the factors in increasing order.
  We can transform~$w$ into~$w'$ by repeatedly swapping two adjacent letters $ba$ with $a < b$.
  Such a swap replaces a factor $\varphi(b) \varphi(a)$ of the corresponding product in~$G$ by $\varphi(a) \varphi(b)$, where
  \[
    \varphi(b) \varphi(a) = \varphi(a) \varphi(b) \, [\varphi(b), \varphi(a)],
  \]
  and this commutator is central in~$G$, since $[[G, G], G] = \gamma_2(G) = 1$.
  Hence, swapping $\varphi(b) \varphi(a)$ turns a product that equals some $g \in G$ into one that equals $g \, [\varphi(b), \varphi(a)]^{-1}$.
  Moreover, for all letters $a < b$, every swap of~$ba$ decreases the number of occurrences of the subword~$ba$ by one, whereas it does not affect the number of occurrences of the subword $b'a'$ for any other pair of letters $a' < b'$.
  Since the sorted word~$w'$ contains no such occurrences at all, the sorting process involves exactly $\nu(ba)$ swaps of~$ba$ for all letters $a < b$.
  Altogether, this yields the claimed equality.
\end{proof}

\subsection{The Realization Problem}
\label{sec:nilpotent:realization}

For fixed $k\geq0$, the problem \textsc{Binary $k$-Realization} is defined as follows:
\begin{quote}
\textbf{Input:} A finite alphabet $\Sigma$ and a function $\nu\in\Psi_k(\Sigma)$ with binary-encoded values.\\[1.5mm]
\textbf{Question:} Is $\nu = \psi_k(w)$ for some $w\in\Sigma^\ast$?
\end{quote}
For instances of this problem, we encode every $\nu \in \Psi_k(\Sigma)$ as the vector of its values, so that, in particular, the input size of \textsc{Binary $k$-Realization} is at least $\Abs{\nu}$.

For $k \geq 1$, there is a direct relation between the length of a word $w \in \Sigma^\ast$ and the values of its $k$-profile $\nu = \psi_k(w)$.
On the one hand, we have $\abs{w} = \sum_{a \in \Sigma} \nu(a)$, so that the binary representation of~$\abs{w}$ has at most $\Abs{\nu}$ bits.
On the other hand, every occurrence of~$u$ in~$w$ is determined by $\abs{u}$ positions of~$w$, which yields the bound
\begin{equation}\label{eqn:nilpotent:length-vs-profile}
  \nu(u) \leq \abs{w}^{\abs{u}}
  \quad\text{for all } u \in \Sigma^{\leq k}.
\end{equation}

The main technical result of this section, which we prove in \cref{sec:nilpotent:compression}, is the following.

\begin{proposition}\label{pro:nilpotent:profile-realization}
  The problem \textnormal{\textsc{Binary $2$-Realization}} is in \NP.
\end{proposition}

Since $\abs{w} = \sum_{a \in \Sigma} \nu(a)$, a word~$w$ realizing a given $2$-profile~$\nu$ can be exponentially long in the input size; hence, we cannot simply guess such a word letter by letter.
Instead, our proof of \cref{pro:nilpotent:profile-realization} relies on guessing a compressed witness.
We show that, firstly, $k$-profiles of run-length encoded words can be computed efficiently (\cref{lem:nilpotent:profile-verification-rle}) and, secondly, every $2$-profile realized by some word is also realized by a word with polynomially many runs (\cref{lem:nilpotent:profile-compression}) and, hence, with polynomial size run-length encoding.

\subsection{Run-Length Encoded Witnesses}
\label{sec:nilpotent:compression}

Every word $w \in \Sigma^\ast$ can be uniquely written as $w = a_1^{\lambda_1} \dotsm a_n^{\lambda_n}$ with $n \in \N$, letters $a_1, \dotsc, a_n \in \Sigma$ such that $a_i \neq a_{i+1}$ for all $i \in [1, n-1]$, and exponents $\lambda_1, \dotsc, \lambda_n \geq 1$.
We call the factors $a_i^{\lambda_i}$ the \emph{runs} of~$w$, and those runs with $a_i = a$ the \emph{$a$-runs} of~$w$.
The \emph{run-length encoding} of the word~$w$ is the sequence $(a_1, \lambda_1), \dotsc, (a_n, \lambda_n)$ with all exponents encoded in binary; its \emph{run-length size} is
\[
  \mathrm{rle}(w) = \Abs{(\lambda_1, \dotsc, \lambda_n)} = \Abs{\lambda_1} + \dotsb + \Abs{\lambda_n},
\]
that is, the total number of bits of the exponents in its run-length encoding.
Since $\lambda_1, \dotsc, \lambda_n \geq 1$, the run-length size of~$w$ is at least its number~$n$ of runs.
Moreover, $\mathrm{rle}(w)$ is also an upper bound for the number of bits in the binary representation of the length $\abs{w} = \lambda_1 + \dotsb + \lambda_n$.

\begin{lemma}\label{lem:nilpotent:profile-verification-rle}
Let $k\geq0$ be fixed. Given the run-length encoding of $w\in\Sigma^\ast$, the $k$-profile $\psi_k(w)$ can be computed in time polynomial in $\mathrm{rle}(w)$ and $\abs{\Sigma}$.
\end{lemma}

\begin{proof}
  Let us first note that, by~\eqref{eqn:nilpotent:profile-composition}, the product $\psi_k(w_1) \cdot \psi_k(w_2) = \psi_k(w_1 w_2)$ of the $k$-profiles of two words $w_1, w_2 \in \Sigma^\ast$ can be computed using at most $2(k+1) \cdot \abs{\Sigma^{\leq k}}$ arithmetic operations, where all integers occurring in this computation are bounded by the values of $\psi_k(w_1 w_2)$.

  Now let $w = w_1 \dotsm w_n$ be the decomposition of~$w$ into its runs $w_i = a_i^{\lambda_i}$.
  We compute the $k$-profile of~$w$ from the $k$-profiles of its runs as the product $\psi_k(w) = \psi_k(w_1) \dotsm \psi_k(w_n)$.
  This computation involves $n - 1 \leq \mathrm{rle}(w)$ products, and the values of all intermediate $k$-profiles are bounded by the values of $\psi_k(w)$.
  By~\eqref{eqn:nilpotent:length-vs-profile} and since the binary representation of~$\abs{w}$ has at most $\mathrm{rle}(w)$ bits, the bit lengths of all involved numbers are thus polynomially bounded in $\mathrm{rle}(w)$.
  It therefore suffices to compute the $k$-profile of a single run~$a^\lambda$, whose only possibly nonzero values are the numbers $\tbinom{\lambda}{j}$ of occurrences of the subwords~$a^j$ with $j \in [0, k]$.
  Since $k$ is fixed, these values can clearly be computed in time polynomial in $\Abs{\lambda}$.
\end{proof}

Next, we show that every realizable $2$-profile is realized by a word with few runs and, hence, with a run-length encoding of small size.
To this end, we use the following Carath\'eodory-type bound for integer cones due to Eisenbrand and Shmonin~\cite[Thm.~1(i)]{ES2006}, where the \emph{integer cone} of a finite set $X \sse \N^n$ is
\[
  \mathrm{cone}_\N(X) = \Set[\Big]{\;\sum_{x \in X} \lambda_x \, x : \lambda_x \in \N \text{ for all } x \in X\;} \sse \N^n.
\]

\begin{lemma}\label{lem:nilpotent:eisenbrand-shmonin}
Let $X\sse\N^n$ be finite. If $t\in\mathrm{cone}_\N(X)$, then $t\in\mathrm{cone}_\N(X')$ for some $X' \sse X$ with 
\[
  \abs{X'} \leq \sum_{i=1}^{n} \log_2(t_i + 1) \leq \Abs{t}.
\]
\end{lemma}

Most of the $2$-profile of a word depends linearly on how the occurrences of a single letter are distributed among the gaps between the remaining letters, allowing us to reduce the number of runs without changing the $2$-profile.

\begin{lemma}\label{lem:nilpotent:profile-compression}
For every $w\in\Sigma^\ast$, there is $w'\in\Sigma^\ast$ with $\psi_2(w') = \psi_2(w)$ and at most $\Abs{\psi_2(w)}$ runs.
\end{lemma}

\begin{proof}
  Let $\nu = \psi_2(w)$.
  Fix a letter $a \in \Sigma$, and let $\Gamma = \Sigma \setminus \{a\}$.
  We then write 
  \[
    w = w_0 a^{\lambda_1} w_1 \dotsm a^{\lambda_n} w_n
  \]
  with $n \in \N$, exponents $\lambda_1, \dotsc, \lambda_n \geq 1$, and words $w_0, \dotsc, w_n \in \Gamma^\ast$.
  We show below that the occurrences of~$a$ in~$w$ can be redistributed to obtain a word 
  \[
    w' = w_0 a^{\lambda'_1} w_1 \dotsm a^{\lambda'_n} w_n
  \]
  with the same $2$-profile $\psi_2(w') = \nu$ and at most $r_a  = \Abs{\nu(a)} + \sum_{b \in \Gamma} \Abs{\nu(ab)}$ runs of the letter $a$, meaning that at most $r_a$ of the exponents $\lambda'_1, \dotsc, \lambda'_n \in \N$ are nonzero.
  For every $b \in \Gamma$, the number of $b$-runs of~$w'$ is at most the number of $b$-runs of~$w$, since the only effect of the redistribution on them is that two of them merge if the $a$-run separating them vanishes.
  Since $r_a$ only depends on~$\nu$, which is preserved, we can repeat this process for every letter $a \in \Sigma$ to eventually obtain a word with the $2$-profile~$\nu$ and at most $\sum_{a \in \Sigma} r_a \leq \Abs{\nu}$ runs in total; here, the inequality holds since every term $\Abs{\nu(u)}$ with $u \in \Sigma^{\leq 2}$ occurs at most once in the sum $\sum_{a \in \Sigma} r_a$.

  To show that such a redistribution is possible, we apply \cref{lem:nilpotent:eisenbrand-shmonin}.
  To this end, consider the $1$-profiles $\xi_i = \psi_1(w_i \dotsm w_n) \in \Psi_1(\Gamma)$ for $i \in [1, n]$ and $\tau \in \Psi_1(\Gamma)$ given by $\tau(u) = \nu(au)$ for $u \in \Gamma^{\leq 1}$; note that $\Abs{\tau} = r_a$ by definition.
  Regarding these $1$-profiles as vectors, we have
  \[
    \tau = \lambda_1 \xi_1 + \dotsb + \lambda_n \xi_n.
  \]
  Indeed, we have $\xi_i(\varepsilon) = 1$ for all $i \in [1, n]$ and $\nu(a) = \lambda_1 + \dotsb + \lambda_n$, and each of the $\lambda_i$ occurrences of~$a$ in the factor $a^{\lambda_i}$ is followed by exactly $\xi_i(b)$ occurrences of each letter $b \in \Gamma$ in~$w$.
  Hence, $\tau$ lies in the integer cone of the set $X = \{\xi_1, \dotsc, \xi_n\}$.
  By \cref{lem:nilpotent:eisenbrand-shmonin}, there are coefficients $\lambda'_1, \dotsc, \lambda'_n \in \N$ such that, firstly, $\tau = \lambda'_1 \xi_1 + \dotsb + \lambda'_n \xi_n$ and, secondly, at most $\Abs{\tau} = r_a$ of them are nonzero.

  It remains to show that the resulting word $w' = w_0 a^{\lambda'_1} w_1 \dotsm a^{\lambda'_n} w_n$ satisfies $\psi_2(w') = \nu$.
  For words $u \in \Gamma^{\leq 2}$, we have $(\psi_2(w'))(u) = (\psi_2(w_0 w_1 \dotsm w_n))(u) = \nu(u)$, since deleting all occurrences of~$a$ from~$w'$ or from~$w$ yields the same word $w_0 w_1 \dotsm w_n$.
  Moreover, by the same counting argument as above, we have $(\psi_2(w'))(au) = \tau(u) = \nu(au)$ for all $u \in \Gamma^{\leq 1}$.
  Together, these equalities already imply $(\psi_2(w'))(u) = \nu(u)$ for all $u \in \Sigma^{\leq 2}$, since every word $v \in \Sigma^\ast$ satisfies
  \[
    \mbinom{v}{aa} = \frac{1}{2} \mbinom{v}{a}^2 - \frac{1}{2} \mbinom{v}{a}
    \quad\text{and}\quad
    \mbinom{v}{ba} = \mbinom{v}{a} \mbinom{v}{b} - \mbinom{v}{ab}
    \text{ for } b \in \Gamma.
    \qedhere
  \]
\end{proof}

\begin{proof}[Proof of \cref{pro:nilpotent:profile-realization}]
  Let $(\Sigma, \nu)$ be an instance of \textsc{Binary $2$-Realization}, and suppose that $\nu = \psi_2(w)$ for some $w \in \Sigma^\ast$.
  By \cref{lem:nilpotent:profile-compression}, we may assume that~$w$ has at most $\Abs{\nu}$ runs.
  Moreover, every exponent~$\lambda$ in the run-length encoding of~$w$ is at most~$\abs{w}$ and, hence, satisfies $\Abs{\lambda} \leq \Abs{\nu}$, as observed before \cref{pro:nilpotent:profile-realization}.
  Hence, the run-length encoding of~$w$ has size polynomial in the input size; in particular, $\mathrm{rle}(w) \leq \Abs{\nu}^2$.
  We can therefore guess the run-length encoding of such a word~$w$ and verify that $\nu = \psi_2(w)$ holds, which can be done in polynomial time by \cref{lem:nilpotent:profile-verification-rle}.
\end{proof}

\subsection{Proofs of \texorpdfstring{\cref{thm:nilpotent:length} and \cref{cor:nilpotent:diameter}}{Theorem~\ref{thm:nilpotent:length} and Corollary~\ref{cor:nilpotent:diameter}}}
\label{sec:nilpotent:proofs}

We now derive the upper bounds for the restrictions of \BLength\ and \BDiameter\ to $2$-step nilpotent groups.
In both cases, our algorithm first checks whether the group $G \leq \Sym(\Omega)$ generated by the input set $\Sigma$ is $2$-step nilpotent, which is possible in polynomial time by \cref{lem:nilpotent-solvable-algorithms}, and rejects otherwise.
Moreover, we fix an arbitrary linear order~$<$ on~$\Sigma$.

\begin{proof}[Proof of \cref{thm:nilpotent:length}]
  \NP-hardness holds already for abelian groups~\cite{LohreyRosowski2023}.
  For membership in~\NP, let $(\Sigma, g, k)$ be an instance such that $\Sigma \sse \Sym(\Omega)$ generates a $2$-step nilpotent group~$G$.

  If $g$ is represented by a word $w \in \Sigma^\ast$ of length $\abs{w} \leq k$, then its $2$-profile $\nu = \psi_2(w)$ satisfies $\nu(u) \leq k^{\abs{u}}$ for all $u \in \Sigma^{\leq 2}$ by~\eqref{eqn:nilpotent:length-vs-profile}, so that $\Abs{\nu}$ is polynomial in the input size.
  Our algorithm guesses a function $\nu \in \Psi_2(\Sigma)$ with $\nu(u) \leq k^{\abs{u}}$ for all $u \in \Sigma^{\leq 2}$, with binary-encoded values, and verifies that
  \[
    \text{(i)}\;\, 
    \sum_{a} \nu(a) \leq k,
    \qquad
    \text{(ii)}\;\, 
    g = \prod_{a} a^{\nu(a)} \cdot \prod_{a < b} [b, a]^{\nu(ba)},
    \qquad
    \text{(iii)}\;\,
    \exists w \in \Sigma^\ast : \psi_2(w) = \nu,
  \]
  where summation in~(i) and the products in~(ii) extend over $a,b \in \Sigma$ in increasing order.
  The first two conditions can be checked in polynomial time, using repeated squaring for the computation of powers in~(ii), whereas the third condition can be checked in~\NP\ by \cref{pro:nilpotent:profile-realization}.
  Altogether, our algorithm therefore runs in nondeterministic polynomial time.
  
  For correctness, note first that every word~$w$ with $\abs{w} \leq k$ representing~$g$ has a $2$-profile $\nu = \psi_2(w)$ satisfying all three conditions; here, condition~(i) holds since $\abs{w} = \sum_{a \in \Sigma} \nu(a)$, and condition~(ii) holds by \cref{lem:nilpotent:profile-evaluation}.
  Conversely, if $\nu$ satisfies all three conditions, then the word~$w$ from~(iii) has length $\abs{w} \leq k$ by~(i), and represents~$g$ by~(ii) and \cref{lem:nilpotent:profile-evaluation}; hence, $\ell(g, \Sigma) \leq k$.
\end{proof}

\begin{proof}[Proof of \cref{cor:nilpotent:diameter}]
  $\Pi_2^{\P}$-hardness holds already for abelian groups~\cite{LohreyRosowski2023}.
  For membership in~$\Pi_2^{\P}$, let $(\Sigma, k)$ be an instance of \BDiameter\ such that $\Sigma \sse \Sym(\Omega)$ generates a $2$-step nilpotent group~$G$.
  By definition, we have $\diam(G, \Sigma) \leq k$ if and only if
  \[
    \forall g \in \Sym(\Omega) : g \in G \implies \ell(g, \Sigma) \leq k.
  \]

  For a given permutation $g \in \Sym(\Omega)$, the premise $g \in G$ can be checked in polynomial time by \cref{lem:bsgs-algorithms}, whereas the conclusion $\ell(g, \Sigma) \leq k$ can then be checked in~\NP\ by \cref{thm:nilpotent:length}.
  Since permutations $g \in \Sym(\Omega)$ have polynomial size encodings, it follows that the problem is in~$\Pi_2^{\P}$.
\end{proof}

\section{Unary Diameter of the Symmetric Group}
\label{sec:symmetric}

In this section, we consider the restriction of \UDiameter\ to sets generating the full symmetric group and prove the following result, which strengthens \cref{thm:symmetric-intro}.

\begin{theorem}\label{thm:symmetric:involutions}
  The restriction of \UDiameter\ to sets generating the full symmetric group, meaning that $G = \Sym(\Omega)$, is $\Pi_2^{\P}$-complete, and this holds even when restricted to generating sets that consist of involutions, that is, of elements of order two.
\end{theorem}

Membership in~$\Pi_2^{\P}$ is straightforward, and we prove $\Pi_2^{\P}$-hardness by a reduction from a diameter problem for elementary abelian $2$-groups~$\F_2^d$ with generating sets~$\Gamma$ of vectors of bounded Hamming weight (\cref{pro:symmetric:source}).
To this end, we represent each coordinate of~$\F_2^d$ by a pair of disjoint transpositions, lift the generators in~$\Gamma$ accordingly, which yields only polynomially many permutations due to the bounded weight, and add all transpositions involving an additional point (\cref{sec:symmetric:construction}).
For the resulting generating set~$\Sigma$, we show that $\diam(\Sym(\Omega), \Sigma) = \diam(\F_2^d, \Gamma) + 5d$ (\cref{pro:symmetric:diameter-identity}).
For the upper bound, we pair up the $2d$ transpositions in a suitable way and show that, on the one hand, lifts can realize all prescribed parities on these pairs with at most $\diam(\F_2^d, \Gamma)$ letters (\cref{sec:symmetric:pairings}) and, on the other hand, suitable parities ensure that the remaining permutation is a product of at most $5d$ of the additional transpositions (\cref{sec:symmetric:star-transpositions}).
For the lower bound, we construct explicit permutations whose lengths reflect those of the elements of~$\F_2^d$ (\cref{sec:symmetric:proofs}).

\subsection{The Construction}
\label{sec:symmetric:construction}

Our reduction starts from a diameter problem for elementary abelian $2$-groups (\cref{pro:symmetric:source}), whose elements we regard as vectors.
For a finite set~$A$, we write $\F_2^A$ for the elementary abelian $2$-group of all maps $f \colon A \to \F_2$ with pointwise addition, where $\F_2$ is the field with two elements, and if $A = [1, d]$ for some $d \in \N$, then we abbreviate $\F_2^d = \F_2^A$.
In contrast to our convention for other groups (\cref{sec:preliminaries-groups-1}), we write this group additively, with neutral element~$0$; note that every element is its own inverse, that is, $f + f = 0$ for all $f \in \F_2^A$.

For $x \in A$, the \emph{unit vector} $e_x \in \F_2^A$ is given by $e_x(x) = 1$ and $e_x(y) = 0$ for all $y \neq x$.
More generally, for $B \sse A$, the \emph{indicator vector} $\one_B = \sum_{x \in B} e_x \in \F_2^A$ of~$B$ has value~$1$ at the elements of~$B$ and value~$0$ elsewhere; in particular, $\one_A$ is the \emph{all-one vector}.
Conversely, the \emph{support} of $f \in \F_2^A$, which we denote by $\supp(f)$, is the unique set $B \sse A$ with $f = \one_B$; that is, $\supp(f) = \{ x \in A : f(x) = 1 \}$.
The size of this set is the (Hamming) \emph{weight} of $f$, and is denoted by $\wt(f) = \abs{\supp(f)}$.

\medskip

We use the following hardness result as the basis of our reduction.

\begin{proposition}\label{pro:symmetric:source}
  It is $\Pi_2^{\P}$-hard to decide, on input of a generating set $\Gamma$ of $\F_2^d$ and a unary-encoded number $k \in \N$, whether $\diam(\F_2^d, \Gamma) \leq k$, even if $e_1, \dotsc, e_d \in \Gamma$ and $1 \leq \wt(a) \leq 5$ for all $a \in \Gamma$.
\end{proposition}

Without the additional conditions on~$\Gamma$, this result is due to McLoughlin~\cite[Thm.~4]{McLoughlin1984}, who phrased it in terms of the covering radius of binary linear codes, and it was later shown independently by the first author and Rosowski~\cite[Thm.~3.1]{LohreyRosowski2023}.
In \cref{appendix-source}, we observe that the reduction from~\cite{LohreyRosowski2023} already produces only sets~$\Gamma$ satisfying these conditions.

\medskip

Our construction replaces each coordinate $i$ of~$\F_2^d$ by two disjoint transpositions $t_{i,0}$ and~$t_{i,1}$, and each generator $a \in \Gamma$ by all of its lifts, which contain exactly one of the transpositions $t_{i,0}$ and~$t_{i,1}$ for each coordinate~$i$ in the support of~$a$.
Finally, we add all transpositions involving an additional point~$\omega_\star$, which ensures that the resulting set generates the full symmetric group.

For the remainder of this section, we fix a number $d \geq 1$ and a set $\Gamma \sse \F_2^d$ with $e_1, \dotsc, e_d \in \Gamma$ and $\wt(a) \geq 1$ for all $a \in \Gamma$, and we let $D = \diam(\F_2^d, \Gamma)$.

We then let
\[
  \Omega = \{\omega_\star\} \uplus (I \times \{0, 1\}),
  \quad\text{where}\quad
  I = [1, d] \times \{0, 1\},
\]
so that $\abs{\Omega} = 4d + 1$ and $\abs{I} = 2d$.
For each coordinate $i \in [1, d]$ of~$\F_2^d$, we consider the two transpositions $t_{i,\xi} = \bigl((i, \xi, 0), (i, \xi, 1)\bigr) \in \Sym(\Omega)$ with $\xi \in \{0, 1\}$.
These transpositions are pairwise disjoint and, hence, generate a subgroup $H \leq \Sym(\Omega)$ that we can identify with the elementary abelian $2$-group~$\F_2^I$ via $t_{i,\xi} \mapsto e_{(i,\xi)}$.
Accordingly, we write $h(i, \xi) \in \F_2$ for the coordinates of $h \in H$, so that, as a permutation, $h$ equals the product of those transpositions $t_{i,\xi}$ with $h(i, \xi) = 1$.

For $a \in \Gamma$, we call every element $\prod_{i \in \supp(a)} t_{i, \xi(i)} \in H$ with $\xi \colon \supp(a) \to \{0, 1\}$ a \emph{lift} of~$a$, and we write $\Sigma_\Gamma \sse H$ for the set of all lifts of all elements of~$\Gamma$.
The set $\Sigma_\Gamma$ generates the group $H$, since the set $\Gamma$ contains every unit vector $e_i$ and, hence, $\Sigma_\Gamma$ contains both of its lifts $t_{i,0}$ and~$t_{i,1}$.

The \emph{star transpositions} in $\Sigma_S = \{(\omega_\star, \omega) : \omega \in \Omega \setminus \{\omega_\star\}\}$ generate the symmetric group $\Sym(\Omega)$, since every transposition $(\omega, \omega')$ with $\omega, \omega' \neq \omega_\star$ satisfies $(\omega, \omega') = (\omega_\star, \omega)(\omega_\star, \omega')(\omega_\star, \omega)$.
Hence, the set $\Sigma = \Sigma_S \cup \Sigma_\Gamma$ generates $\Sym(\Omega)$ as well.
Note that $\abs{\Sigma} \leq 4d + \sum_{a \in \Gamma} 2^{\wt(a)}$; in particular, $\Sigma$ can be computed from~$\Gamma$ in polynomial time if $\wt(a) \leq 5$ for all $a \in \Gamma$, as in \cref{pro:symmetric:source}.
Moreover, all elements of~$\Sigma$ are involutions, as $\wt(a) \geq 1$ for all $a \in \Gamma$.

The following proposition, which we prove in \cref{sec:symmetric:proofs}, is the core of our reduction.

\begin{proposition}\label{pro:symmetric:diameter-identity}
  With the notation above, we have
  \[
    \diam(\Sym(\Omega),\Sigma) = \diam(\F_2^d, \Gamma) + 5d.
  \]
\end{proposition}

Both the upper and the lower bound in \cref{pro:symmetric:diameter-identity} rest on the following lemma, which describes how the two types of generators in~$\Sigma$ interact.

\begin{lemma}\label{lem:symmetric:normal-form}
For every $g\in\Sym(\Omega)$, we have \[
  \ell(g,\Sigma) = \min_{h\in H} \bigl(\ell(gh,\Sigma_S) + \ell(h,\Sigma_\Gamma)\bigr).
\]
\end{lemma}

\begin{proof}
  Let $w$ be a word over~$\Sigma$ that represents~$g$.
  Since every lift $\tilde a \in \Sigma_\Gamma$ fixes the point~$\omega_\star$, it normalizes the set~$\Sigma_S$.
  Hence, by replacing factors $\tilde a c$ with $\tilde a \in \Sigma_\Gamma$, $c \in \Sigma_S$ by $(\prescript{\tilde a\!}{}{c}) \tilde a$, we can move all star transpositions in~$w$ to the front.
  This yields a word $u v$ with $u \in \Sigma_S^\ast$ and $v \in \Sigma_\Gamma^\ast$ that still represents~$g$ and has the length $\abs{uv} = \abs{w}$.
  If $v$ represents $h \in H$, then $u$ represents $g h^{-1} = g h$, so that $\abs{w} = \abs{u} + \abs{v} \geq \ell(gh, \Sigma_S) + \ell(h, \Sigma_\Gamma)$.
  Conversely, for every $h \in H$, concatenating shortest words for~$gh$ over~$\Sigma_S$ and for~$h$ over~$\Sigma_\Gamma$ yields a word over~$\Sigma$ that represents $gh \cdot h = g$.
\end{proof}

\subsection{Pairings and Lifts}
\label{sec:symmetric:pairings}

We call a partition $\mathcal P \sse \tbinom{I}{2}$ of~$I$ into $2$-sets a \emph{pairing} of~$I$.
For every pairing~$\mathcal P$ of~$I$, we consider the homomorphism $\sigma_{\mathcal P} \colon H \to \F_2^{\mathcal P}$ that sums up the coordinates within each $2$-set, that is,
\[
  \sigma_{\mathcal P}(h)(P) = \sum_{(i, \xi) \in P} h(i, \xi)
  \quad\text{for all } h \in H \text{ and } P \in \mathcal P.
\]
In this subsection, we show that every vector in~$\F_2^{\mathcal P}$ is the image under~$\sigma_{\mathcal P}$ of an element of~$H$ that is represented by a word over~$\Sigma_\Gamma$ of length at most~$D = \diam(\F_2^d, \Gamma)$.
Together with \cref{lem:symmetric:short-star-fibre} from \cref{sec:symmetric:star-transpositions}, this will yield the upper bound in \cref{pro:symmetric:diameter-identity}.

For a finite set~$A$, the vectors of even weight form the \emph{even-weight subgroup} $E(A) \leq \F_2^A$, which is generated by the vectors $e_x + e_y$ with distinct $x, y \in A$, since the support of each of its elements can be partitioned into $2$-sets.
Moreover, we write $f\rest_B \in \F_2^B$ for the restriction of $f \in \F_2^A$ to $B \sse A$.

\begin{lemma}\label{lem:symmetric:short-fibres}
Let $\mathcal P$ be a pairing of~$I$.
Then, for every $b\in\F_2^{\mathcal P}$, there exists $h\in H$ with 
\[
  \sigma_{\mathcal P}(h) = b
  \quad\text{and}\quad 
  \ell(h,\Sigma_\Gamma)\leq D.
\]
\end{lemma}

\begin{proof}
  Consider the multigraph with vertex set~$\mathcal P$ and edge set~$[1, d]$, where the edge~$i$ joins the unique $2$-sets $P, P' \in \mathcal P$ with $(i, 0) \in P$ and $(i, 1) \in P'$; in particular, the edge~$i$ is a loop if $P = P'$.
  Since every $2$-set in~$\mathcal P$ contains exactly two elements of~$I$, every vertex of this multigraph has degree two (where loops are counted twice).
  Hence, every connected component of this multigraph is a cycle, which may also consist of a single loop or of two parallel edges (see \cref{ex:symmetric:short-words} below).

  For every connected component, with vertex set $\mathcal C \sse \mathcal P$ and edge set $K \sse [1, d]$, we choose a set $K' \sse K$ of edges that still connect all vertices in~$\mathcal C$ such that
  \begin{equation}\label{eqn:symmetric:component-parity}
    \abs{K'} \equiv \sum_{P \in \mathcal C} b(P) \pmod 2;
  \end{equation}
  namely, we let $K' = K$ if $\abs{K}$ already has the right parity, and we let $K' = K \setminus \{i\}$ for an arbitrary edge $i \in K$ otherwise.
  Let $L \sse [1, d]$ be the union of the sets~$K'$ for all connected components.
  By the definition of~$D$, there are $a_1, \dotsc, a_m \in \Gamma$ with $m \leq D$ and $\one_L = a_1 + \dotsb + a_m$.
  We choose arbitrary lifts $\tilde a_1, \dotsc, \tilde a_m \in \Sigma_\Gamma$ of $a_1, \dotsc, a_m$, respectively, and we let $h' = \tilde a_1 \dotsm \tilde a_m$ and $b' = \sigma_{\mathcal P}(h')$.

  For every $i \in [1, d]$ and $j \in [1, m]$, the lift~$\tilde a_j$ contains, as a permutation of $\Omega$, exactly one of the transpositions $t_{i,0}$ and~$t_{i,1}$ if $a_j(i) = 1$, and neither of them otherwise; hence,
  \[
    h'(i, 0) + h'(i, 1) = a_1(i) + \dotsb + a_m(i) = \one_L(i).
  \]
  Now consider a connected component with vertex set~$\mathcal C$ and edge set~$K$ as above.
  Since the $2$-sets in~$\mathcal C$ partition the set $\{(i, 0), (i, 1) : i \in K\}$, we obtain
  \[
    \sum_{P \in \mathcal C} b'(P)
    = \sum_{i \in K} \bigl(h'(i, 0) + h'(i, 1)\bigr)
    = \sum_{i \in K} \one_L(i)
    = \sum_{P \in \mathcal C} b(P),
  \]
  where the last equality follows from~\eqref{eqn:symmetric:component-parity}, since $K \cap L = K'$.
  In other words, the restriction $(b + b')\rest_{\mathcal C}$ belongs to the even-weight subgroup $E(\mathcal C)$ of the elementary abelian $2$-group $\F_2^{\mathcal C}$.

  Finally, let $i \in K'$ be an edge joining the vertices $P, P' \in \mathcal C$.
  Since $\one_L(i) = 1$, there is some $j \in [1, m]$ with $a_j(i) = 1$, and replacing $t_{i,0}$ by~$t_{i,1}$ in the lift~$\tilde a_j$, or vice versa, yields another lift of~$a_j$ and changes~$b'$ by $e_P + e_{P'}$.
  As the edges in~$K'$ connect all vertices in~$\mathcal C$, the latter vectors generate the even-weight subgroup $E(\mathcal C)$: every generator $e_{P_1} + e_{P_2}$ of $E(\mathcal C)$ is the sum of these vectors along a path from~$P_1$ to~$P_2$.
  Hence, by changing the lifts at the coordinates in a suitable subset of~$L$, we obtain lifts of $a_1, \dotsc, a_m$ whose product $h$ satisfies $\sigma_{\mathcal P}(h) = b$, which shows $\ell(h, \Sigma_\Gamma) \leq m \leq D$.
\end{proof}

\begin{figure}[htb]
\centering
\begin{tikzpicture}[
 vertex/.style={circle,draw,minimum size=8mm,inner sep=1pt},
 kept/.style={thick},
 omitted/.style={thick,densely dashed},
 every edge/.style={draw},
 every node/.style={font=\small}]
 \node[vertex] (P1) at (0,0) {$P_1$};
 \node[vertex] (P2) at (2.2,0) {$P_2$};
 \node[vertex] (P3) at (1.1,1.7) {$P_3$};
 \node[vertex] (P4) at (4.5,0.75) {$P_4$};
 \node[vertex] (P5) at (6.7,0.75) {$P_5$};
 \node[vertex] (P6) at (9.2,0.75) {$P_6$};
 \path (P1) edge[kept] node[below] {$1$} (P2)
       (P2) edge[omitted] node[above right] {$2$} (P3)
       (P3) edge[kept] node[above left] {$3$} (P1)
       (P4) edge[kept, bend left=34] node[above] {$4$} (P5)
       (P4) edge[omitted, bend right=34] node[below] {$5$} (P5)
       (P6) edge[kept, loop above, looseness=4, in=55, out=125]
             node[above] {$6$} (P6);
 \begin{scope}[thin, black!50, rounded corners=2mm]
   \draw (-0.85,-0.85) rectangle (3.05,2.6);
   \draw (3.65,-0.85) rectangle (7.55,2.6);
   \draw (8.15,-0.85) rectangle (10.25,2.6);
 \end{scope}
 \node[anchor=north west] at (-0.85,2.6) {$\mathcal C_1$};
 \node[anchor=north west] at (3.65,2.6) {$\mathcal C_2$};
 \node[anchor=north west] at (8.15,2.6) {$\mathcal C_3$};
\end{tikzpicture}
\caption{The multigraph from the proof of \cref{lem:symmetric:short-fibres} for the pairing~$\mathcal P$ in \cref{ex:symmetric:short-words}, where the edge~$i$ joins the $2$-sets containing $(i, 0)$ and~$(i, 1)$, and solid edges belong to the sets $K_1'$, $K_2'$, and~$K_3'$.}
\label{fig:symmetric:short-words}
\end{figure}

\begin{example}\label{ex:symmetric:short-words}
  Let $d = 6$ and $\Gamma = \{e_1, \dotsc, e_6, a_1, a_2\}$ with $a_1 = e_1 + e_2 + e_4$ and $a_2 = e_2 + e_3 + e_6$; note that $D \geq 2$, since $e_5 + e_6 \notin \Gamma$.
  We consider the pairing $\mathcal P = \{P_1, \dotsc, P_6\}$ of~$I$ with
  \begin{align*}
    P_1 &= \{(1,0),(3,1)\}, &
    P_2 &= \{(1,1),(2,0)\}, &
    P_3 &= \{(2,1),(3,0)\}, \\
    P_4 &= \{(4,0),(5,0)\}, &
    P_5 &= \{(4,1),(5,1)\}, &
    P_6 &= \{(6,0),(6,1)\},
  \end{align*}
  and the vector $b = e_{P_2} + e_{P_3} + e_{P_5} + e_{P_6} \in \F_2^{\mathcal P}$.
  As shown in \cref{fig:symmetric:short-words}, the multigraph from the proof of \cref{lem:symmetric:short-fibres} has three connected components: a triangle with vertex set $\mathcal C_1 = \{P_1, P_2, P_3\}$ and edge set $K_1 = \{1, 2, 3\}$, a cycle with vertex set $\mathcal C_2 = \{P_4, P_5\}$ and edge set $K_2 = \{4, 5\}$, and a loop with vertex set $\mathcal C_3 = \{P_6\}$ and edge set $K_3 = \{6\}$.
  Summing~$b$ over $\mathcal C_1$, $\mathcal C_2$, $\mathcal C_3$ yields $0$,~$1$,~$1$, respectively; hence, we may choose $K_1' = \{1, 3\}$, $K_2' = \{4\}$, $K_3' = \{6\}$, resulting in $L = \{1, 3, 4, 6\}$.
  Note that $\one_L = a_1 + a_2$, with $2 \in \supp(a_1)$ and $2 \in \supp(a_2)$ even though $2 \notin L$.

  For the lifts $\tilde a_1 = t_{1,0} t_{2,0} t_{4,0}$ and $\tilde a_2 = t_{2,0} t_{3,0} t_{6,0}$, we then obtain $h' = t_{1,0} t_{3,0} t_{4,0} t_{6,0}$ and 
  \[
    b' = \sigma_{\mathcal P}(h') = e_{P_1} + e_{P_3} + e_{P_4} + e_{P_6}.
  \]
  As expected, the sums of~$b'$ over $\mathcal C_1$, $\mathcal C_2$, $\mathcal C_3$ coincide with those of~$b$, and $b + b' = (e_{P_1} + e_{P_2}) + (e_{P_4} + e_{P_5})$ is the sum of the vectors corresponding to the edges $1$ and~$4$.
  Thus, we change $\tilde a_1$ at these coordinates, which yields the new lift $t_{1,1} t_{2,0} t_{4,1}$ of~$a_1$ and the element $h = t_{1,1} t_{3,0} t_{4,1} t_{6,0}$ with $\sigma_{\mathcal P}(h) = b$.
\end{example}

\subsection{Star Transpositions}
\label{sec:symmetric:star-transpositions}

In this subsection, we show that, for every $g \in \Sym(\Omega)$, there are a pairing~$\mathcal P$ of~$I$ and a vector $b \in \F_2^{\mathcal P}$ such that all permutations $gh$ with $h \in H$ and $\sigma_{\mathcal P}(h) = b$ are short with respect to~$\Sigma_S$.
Together with \cref{lem:symmetric:short-fibres}, this will yield the upper bound in \cref{pro:symmetric:diameter-identity}.
To this end, we first determine the length of every permutation with respect to~$\Sigma_S$ in terms of its cycles.

For $g \in \Sym(\Omega)$, we write $\mathrm{fix}(g)$ for the number of fixed points of~$g$ in $\Omega \setminus \{\omega_\star\}$, and $\mathrm{cyc}(g)$ for the number of nontrivial cycles of~$g$ that do not contain~$\omega_\star$.
Thus, $g$ has $\mathrm{fix}(g) + \mathrm{cyc}(g) + 1$ cycles in total, including its trivial cycles and the cycle containing~$\omega_\star$.

\begin{lemma}\label{lem:symmetric:star-length}
For every $g\in\Sym(\Omega)$, we have $\ell(g,\Sigma_S) = 4d-\mathrm{fix}(g)+\mathrm{cyc}(g)$.
\end{lemma}

\begin{proof}
  For the upper bound, recall that permutations act on the right, so that we have
  \[
    (\omega_\star, \omega_1, \dotsc, \omega_m) = (\omega_\star, \omega_1) \dotsm (\omega_\star, \omega_m)
    \quad\text{and}\quad
    (\omega_1, \dotsc, \omega_m) = (\omega_\star, \omega_1) \dotsm (\omega_\star, \omega_m) \, (\omega_\star, \omega_1)
  \]
  for distinct points $\omega_1, \dotsc, \omega_m \in \Omega \setminus \{\omega_\star\}$.
  Thus, a nontrivial cycle of~$g$ with $m$ points in $\Omega \setminus \{\omega_\star\}$ is a product of $m$ star transpositions if it contains~$\omega_\star$, and of $m + 1$ star transpositions otherwise.
  Since the nontrivial cycles of~$g$ are disjoint and contain exactly $4d - \mathrm{fix}(g)$ points of $\Omega \setminus \{\omega_\star\}$, multiplying these products yields a word of length $4d - \mathrm{fix}(g) + \mathrm{cyc}(g)$ over~$\Sigma_S$ that represents~$g$.

  For the lower bound, let $\ell(g) = 4d - \mathrm{fix}(g) + \mathrm{cyc}(g)$ for all $g \in \Sym(\Omega)$, and note that $\ell(1) = 0$.
  It suffices to show that $\abs{\ell(g \, (\omega_\star, \omega)) - \ell(g)} = 1$ for all $g \in \Sym(\Omega)$ and $\omega \in \Omega \setminus \{\omega_\star\}$, since then every word of length~$n$ over~$\Sigma_S$ represents a permutation~$g$ with $\ell(g) \leq n$.
  To this end, we distinguish the following cases.
  \begin{itemize}
    \item If $\omega$ is a fixed point of~$g$, then multiplying $g$ by $(\omega_\star, \omega)$ on the right inserts~$\omega$ into the cycle containing~$\omega_\star$, so that $\ell$ increases by one.
    \item If $\omega$ lies in a nontrivial cycle of~$g$ that does not contain~$\omega_\star$, then multiplying $g$ by $(\omega_\star, \omega)$ on the right merges this cycle with the cycle containing~$\omega_\star$, so that $\ell$ decreases by one.
    \item If $\omega$ lies in the cycle of~$g$ containing~$\omega_\star$, then multiplying $g$ by $(\omega_\star, \omega)$ on the right splits this cycle into a cycle containing~$\omega_\star$ and a cycle containing~$\omega$; thus, $\ell$ decreases by one if the latter is trivial, and it increases by one otherwise.
    \qedhere
  \end{itemize}
\end{proof}

Recall that right multiplication of $g \in \Sym(\Omega)$ by a transposition $t = (\omega, \omega')$ merges the cycles of~$g$ containing $\omega$ and~$\omega'$ if they are distinct, and splits the cycle containing both of them into two cycles otherwise.
In particular, the total number $\mathrm{fix}(g) + \mathrm{cyc}(g) + 1$ of cycles changes by exactly one.
Since $\ell(g, \Sigma_S) = 4d + \bigl(\mathrm{fix}(g) + \mathrm{cyc}(g)\bigr) - 2 \, \mathrm{fix}(g)$ by \cref{lem:symmetric:star-length}, we thus obtain
\begin{equation}\label{eqn:symmetric:star-change}
  \ell(gt, \Sigma_S) - \ell(g, \Sigma_S) =
  \begin{cases}
    -1 - 2\bigl(\mathrm{fix}(gt) - \mathrm{fix}(g)\bigr) & \text{if $t$ merges two cycles of~$g$}, \\
    +1 - 2\bigl(\mathrm{fix}(gt) - \mathrm{fix}(g)\bigr) & \text{if $t$ splits a cycle of~$g$}.
  \end{cases}
\end{equation}

\begin{lemma}\label{lem:symmetric:short-star-fibre}
For every $g\in\Sym(\Omega)$, there exist a pairing~$\mathcal P$ of~$I$ and a vector $b\in\F_2^{\mathcal P}$ such that we have $\ell(gh,\Sigma_S)\leq5d$ for every $h\in H$ with $\sigma_{\mathcal P}(h) = b$.
\end{lemma}

\begin{proof}
  We choose $h' \in H$ such that $g' = gh'$ satisfies $\ell(g', \Sigma_S) = \max\{\ell(gh, \Sigma_S) : h \in H\}$.
  It suffices to find a pairing~$\mathcal P$ of~$I$ and a vector $b' \in \F_2^{\mathcal P}$ such that
  \begin{equation}\label{eqn:symmetric:star-fibre-goal}
    \ell(g'h, \Sigma_S) \leq 5d
    \quad\text{for every } h \in H \text{ with } \sigma_{\mathcal P}(h) = b'.
  \end{equation}
  Indeed, the vector $b = \sigma_{\mathcal P}(h') + b'$ then has the desired property: every $h \in H$ with $\sigma_{\mathcal P}(h) = b$ satisfies $\sigma_{\mathcal P}(h'h) = b'$, as $\sigma_{\mathcal P}$ is a homomorphism, and $gh = g'(h'h)$, as $h'^{-1} = h'$.

  We first collect some consequences of the maximality of~$g'$.
  Firstly, the permutation $g'$ has no fixed points in $\Omega \setminus \{\omega_\star\}$, that is, $\mathrm{fix}(g') = 0$.
  Indeed, suppose that $g'$ fixes a point $\omega = (i, \xi, \zeta) \in \Omega \setminus \{\omega_\star\}$, and let $\omega' = (i, \xi, 1 - \zeta)$, so that $t_{i,\xi} = (\omega, \omega')$.
  If $\omega'$ is also fixed by~$g'$, then $t_{i,\xi}$ merges these two fixed points into a cycle of length two, so that $\mathrm{fix}(g' t_{i,\xi}) = \mathrm{fix}(g') - 2$; otherwise, $t_{i,\xi}$ inserts~$\omega$ into the cycle containing~$\omega'$, so that $\mathrm{fix}(g' t_{i,\xi}) = \mathrm{fix}(g') - 1$.
  In both cases, \eqref{eqn:symmetric:star-change} yields $\ell(g' t_{i,\xi}, \Sigma_S) > \ell(g', \Sigma_S)$, which contradicts the maximality of~$g'$, since $g' t_{i,\xi} = g \, (h' t_{i,\xi})$ with $h' t_{i,\xi} \in H$.

  Consequently, for $h \in H$, the fixed points of~$g'h$ in $\Omega \setminus \{\omega_\star\}$ are those points $\omega = (i, \xi, \zeta)$ with
  \begin{equation}\label{eqn:symmetric:fixed-points}
    h(i, \xi) = 1
    \quad\text{and}\quad
    \omega g' = (i, \xi, 1 - \zeta).
  \end{equation}
  Indeed, as $h = h^{-1}$, we have $\omega g' h = \omega$ if and only if $\omega g' = \omega h$.
  Moreover, since $\mathrm{fix}(g') = 0$, this can only hold if $\omega h \neq \omega$, meaning that $h(i, \xi) = 1$ where $\omega = (i, \xi, \zeta)$, so that $\omega g' = \omega h = (i, \xi, 1 - \zeta)$.
  Conversely, if $h(i, \xi) = 1$ and $\omega g' = (i, \xi, 1 - \zeta)$, then $\omega g' = \omega h$, so that $\omega$ is fixed by~$g'h$.

  To state the second consequence, consider the multigraph~$\mathcal G$ whose vertex set is the set~$C$ of cycles of~$g'$ and whose edge set is~$I$, where the edge $(i, \xi) \in I$ joins the cycles containing $(i, \xi, 0)$ and~$(i, \xi, 1)$.
  To avoid confusion with the cycles of permutations, we refer to cycles in the multigraph~$\mathcal G$ and its subgraphs as \emph{$\mathcal G$-cycles}.
  Let $\rho_\star \in C$ be the cycle containing~$\omega_\star$.
  Since $\mathrm{fix}(g') = 0$, every other~$\rho \in C$ contains at least two points, so that its degree in~$\mathcal G$ is at least two (where loops are counted twice).
  Moreover, $g'$ has $\abs{C} = \mathrm{cyc}(g') + 1$ cycles, so that $\ell(g', \Sigma_S) = 4d + \abs{C} - 1$ by \cref{lem:symmetric:star-length}.

  The second consequence is that, for every loop $(i, \xi) \in I$ at some $\rho \in C$, the two points moved by~$t_{i,\xi}$ are consecutive in~$\rho$, that is, $(i, \xi, \zeta) g' = (i, \xi, 1 - \zeta)$ for some $\zeta \in \{0, 1\}$.
  Indeed, $t_{i,\xi}$ otherwise splits~$\rho$ into two cycles with at least two points each, so that $\mathrm{fix}(g' t_{i,\xi}) = 0$.
  But, by~\eqref{eqn:symmetric:star-change}, this implies $\ell(g' t_{i,\xi}, \Sigma_S) = \ell(g', \Sigma_S) + 1$, contradicting the maximality of~$g'$.

  \medskip

  We now choose a set $F \sse I$ of edges of~$\mathcal G$ as follows: in the connected component of~$\mathcal G$ containing~$\rho_\star$, we choose a spanning tree, and in every other connected component, we choose a spanning tree together with one additional edge.
  The latter exists since all vertices of such a component have degree at least two, so that it has at least as many edges as vertices.
  Thus, the connected component of $(C, F)$ containing~$\rho_\star$ is a tree, and every other connected component contains exactly one $\mathcal G$-cycle, which may also be a loop or consist of two parallel edges.
  In particular, we have $\abs{F} = \abs{C} - 1$.

  Next, we choose the pairing~$\mathcal P$ of~$I$: for every $\mathcal G$-cycle of length at least two in $(C, F)$, we choose two of its edges and add them as a $2$-set to~$\mathcal P$; these $2$-sets are pairwise disjoint, as these $\mathcal G$-cycles lie in different connected components.
  Then, we pair the remaining elements of~$I$ arbitrarily.
  Finally, we let $b' \in \F_2^{\mathcal P}$ be given by $b'(P) = 1$ if $\abs{P \cap F}$ is even, and $b'(P) = 0$ otherwise.

  To show~\eqref{eqn:symmetric:star-fibre-goal}, let $h \in H$ with $\sigma_{\mathcal P}(h) = b'$, and let $J = \supp(h)$.
  For every $P \in \mathcal P$, we then have
  \[
    \abs{J \cap P} \equiv \sigma_{\mathcal P}(h)(P) = b'(P) \equiv 1 + \abs{F \cap P} \pmod 2
  \]
  and, hence, $\abs{(J \mathbin{\triangle} F) \cap P} \equiv 1 \pmod 2$ for the symmetric difference $J \mathbin{\triangle} F = (J \setminus F) \cup (F \setminus J)$.
  In~fact, we even have $\abs{(J \mathbin{\triangle} F) \cap P} = 1$ since $\abs{P} = 2$.
  Summing over the $2$-sets in~$\mathcal P$ then yields
  \begin{equation}\label{eqn:symmetric:symmetric-difference}
    \abs{J \setminus F} + \abs{F \setminus J} = \abs{\mathcal P} = d.
  \end{equation}
  Moreover, the subgraph $(C, J \cap F)$ contains no $\mathcal G$-cycle of length at least two.
  Indeed, such a $\mathcal G$-cycle would be the unique $\mathcal G$-cycle in its connected component of $(C, F)$, so that $J \cap F$ would contain the $2$-set $P \in \mathcal P$ chosen from it, contradicting $\abs{(J \mathbin{\triangle} F) \cap P} = 1$.
  Hence, if $J_1 \sse J \cap F$ denotes the set of loops in~$J \cap F$, then the remaining edges in $J_2 = (J \cap F) \setminus J_1$ form a forest $(C, J_2)$.

  To bound $\ell(g'h, \Sigma_S)$, we multiply~$g'$ by the transpositions $t_{i,\xi}$ with $(i, \xi) \in J$ one after another, first those with $(i, \xi) \in J_2$, then those with $(i, \xi) \in J_1$, and finally those with $(i, \xi) \in J \setminus F$, and we keep track of the length with respect to~$\Sigma_S$ using~\eqref{eqn:symmetric:star-change}, starting from $\ell(g', \Sigma_S) = 4d + \abs{C} - 1 = 4d + \abs{F}$.
  \begin{itemize}
    \item Multiplying successively by the transpositions $t_{i,\xi}$ with $(i, \xi) \in J_2$ merges the cycles of~$g'$ along the edges of the forest $(C, J_2)$.
      As this forest contains no $\mathcal G$-cycles, each of these multiplications merges two distinct cycles of the current permutation.
      Moreover, by~\eqref{eqn:symmetric:fixed-points}, none of them creates a fixed point, as no edge in~$J_2$ is a loop.
      Hence, each of the $\abs{J_2}$ multiplications in this phase decreases the length by one.
    \item Multiplying by the transpositions $t_{i,\xi}$ with $(i, \xi) \in J_1$ splits off new fixed points.
      Indeed, as the two points moved by each such $t_{i,\xi}$ are consecutive in a cycle of~$g'$, each of these multiplications creates at least one new fixed point by~\eqref{eqn:symmetric:fixed-points}, whereas no fixed point disappears.
      Hence, each of the $\abs{J_1}$ multiplications in this phase decreases the length by at least one.
    \item Multiplying by the transpositions $t_{i,\xi}$ with $(i, \xi) \in J \setminus F$ may merge or split cycles of the current permutation arbitrarily.
      However, by~\eqref{eqn:symmetric:fixed-points}, none of these multiplications destroys a fixed point, so that each of the $\abs{J \setminus F}$ multiplications increases the length by at most one.
  \end{itemize}
  Using $\abs{J \cap F} = \abs{J_1} + \abs{J_2}$ and $\abs{F} - \abs{J \cap F} = \abs{F \setminus J}$, as well as \eqref{eqn:symmetric:symmetric-difference}, we finally obtain
  \[
    \smash{
    \ell(g'h, \Sigma_S)
    \leq {\underbrace{\ell(g', \Sigma_S)}_{4d + \abs{F}}} - (\underbrace{\abs{J_1} + \abs{J_2}}_{\abs{J \cap F}}) + \abs{J \setminus F}
    = 4d + \abs{F \setminus J} + \abs{J \setminus F}
    = 5d.
    }
  \]
\end{proof}

\subsection{Proof of \texorpdfstring{\cref{thm:symmetric:involutions}}{Theorem~\ref{thm:symmetric:involutions}}}
\label{sec:symmetric:proofs}

For the lower bound in \cref{pro:symmetric:diameter-identity}, we show that the lengths of the elements of~$\F_2^d$ with respect to~$\Gamma$ reappear, shifted by~$5d$, as the lengths of certain permutations with respect to~$\Sigma$.

\begin{lemma}\label{lem:symmetric:hard-permutations}
For every $b\in\F_2^d$, there is some $g\in\Sym(\Omega)$ with $\ell(g,\Sigma) = \ell(b,\Gamma) + 5d$.
\end{lemma}

\begin{proof}
  We consider the canonical pairing $\mathcal P = \{\{(i, 0), (i, 1)\} : i \in [1, d]\}$ of~$I$, and we identify $\F_2^{\mathcal P}$ with~$\F_2^d$ via $\{(i, 0), (i, 1)\} \mapsto i$, so that $\sigma_{\mathcal P}(h)(i) = h(i, 0) + h(i, 1)$ for all $h \in H$ and $i \in [1, d]$.

  Every lift~$\tilde a$ of an element $a \in \Gamma$ contains exactly one of the transpositions $t_{i,0}$ and~$t_{i,1}$ for every $i \in \supp(a)$, and neither of them for $i \notin \supp(a)$; hence, we have $\sigma_{\mathcal P}(\tilde a) = a$ under the above identification.
  Consequently, replacing every letter~$\tilde a$ of a word over~$\Sigma_\Gamma$ that represents $h \in H$ by~$\sigma_{\mathcal P}(\tilde a)$ yields a word over~$\Gamma$ of the same length that represents~$\sigma_{\mathcal P}(h)$.
  Conversely, replacing every letter of a word over~$\Gamma$ that represents $z \in \F_2^d$ by an arbitrary lift yields a word over~$\Sigma_\Gamma$ of the same length that represents some $h \in H$ with $\sigma_{\mathcal P}(h) = z$.
  Together, we obtain
  \begin{equation}\label{eqn:symmetric:projected-length}
    \min\{\ell(h, \Sigma_\Gamma) : h \in H,\ \sigma_{\mathcal P}(h) = z\} = \ell(z, \Gamma)
    \quad\text{for all } z \in \F_2^d.
  \end{equation}

  Now let $b \in \F_2^d$, and let $g \in \Sym(\Omega)$ be the permutation that fixes~$\omega_\star$ and satisfies
  \[
    (i, 0, \zeta) g = (i, 1, \zeta)
    \quad\text{and}\quad
    (i, 1, \zeta) g = (i, 0, \zeta + b(i) + 1)
    \quad\text{for all } i \in [1, d] \text{ and } \zeta \in \{0, 1\},
  \]
  where we compute the last coordinate in~$\F_2$.
  For $h \in H$, the permutation~$gh$ then satisfies
  \[
    (i, 0, \zeta) gh = \bigl(i, 1, \zeta + h(i, 1)\bigr)
    \quad\text{and}\quad
    (i, 1, \zeta) gh = \bigl(i, 0, \zeta + b(i) + 1 + h(i, 0)\bigr),
  \]
  so that $(gh)^2$ maps $(i, \xi, \zeta)$ to $(i, \xi, \zeta + b(i) + \sigma_{\mathcal P}(h)(i) + 1)$.
  Hence, for every $i \in [1, d]$, the permutation~$gh$ permutes the four points $(i, \xi, \zeta)$ with $\xi, \zeta \in \{0, 1\}$ without fixed points, namely as two cycles of length two if $b(i) + \sigma_{\mathcal P}(h)(i) = 1$, and as a single cycle of length four otherwise.
  Thus, $gh$ fixes~$\omega_\star$, has no other fixed points, and has $d + \wt(b + \sigma_{\mathcal P}(h))$ nontrivial cycles, so that \cref{lem:symmetric:star-length} yields $\ell(gh, \Sigma_S) = \wt(b + \sigma_{\mathcal P}(h)) + 5d$.
  Together with \cref{lem:symmetric:normal-form} and~\eqref{eqn:symmetric:projected-length}, we obtain
  \[
    \ell(g, \Sigma)
    = \min_{h \in H} \bigl(\wt(b + \sigma_{\mathcal P}(h)) + \ell(h, \Sigma_\Gamma)\bigr) + 5d
    = \min_{z \in \F_2^d} \bigl(\wt(b + z) + \ell(z, \Gamma)\bigr) + 5d.
  \]
  Finally, since $\Gamma$ contains all unit vectors, we have $\ell(b, \Gamma) \leq \ell(z, \Gamma) + \wt(b + z)$ for every $z \in \F_2^d$, with equality for $z = b$.
  Hence, the last minimum equals $\ell(b, \Gamma)$, which shows $\ell(g, \Sigma) = \ell(b, \Gamma) + 5d$.
\end{proof}

\begin{proof}[Proof of \cref{pro:symmetric:diameter-identity}]
  For the upper bound, let $g \in \Sym(\Omega)$.
  By \cref{lem:symmetric:short-star-fibre}, there are a pairing~$\mathcal P$ of~$I$ and $b \in \F_2^{\mathcal P}$ such that $\ell(gh, \Sigma_S) \leq 5d$ for every $h \in H$ with $\sigma_{\mathcal P}(h) = b$, and, by \cref{lem:symmetric:short-fibres}, some such $h$ satisfies $\ell(h, \Sigma_\Gamma) \leq D$.
  By \cref{lem:symmetric:normal-form}, we thus have $\ell(g, \Sigma) \leq \ell(gh, \Sigma_S) + \ell(h, \Sigma_\Gamma) \leq D + 5d$.

  For the lower bound, we choose $b \in \F_2^d$ with $\ell(b, \Gamma) = D$, which exists by the definition of the diameter.
  Then \cref{lem:symmetric:hard-permutations} yields a permutation $g \in \Sym(\Omega)$ with $\ell(g, \Sigma) = D + 5d$.
\end{proof}

\begin{proof}[Proof of \cref{thm:symmetric:involutions}]
  For membership in~$\Pi_2^{\P}$, let $(\Sigma, k)$ be an instance of \UDiameter\ such that $\Sigma \sse \Sym(\Omega)$ generates the group $G = \Sym(\Omega)$.
  This condition can be checked in polynomial time by \cref{lem:bsgs-algorithms}, testing whether $\langle \Sigma \rangle$ contains some specific generating set of $\Sym(\Omega)$, such as the set of all transpositions.
  Whether $\Sigma$ consists only of involutions can be checked directly in polynomial time as well.

  By definition, we then have $\diam(G, \Sigma) \leq k$ if and only if 
  \[
    \forall g \in \Sym(\Omega) : \ell(g, \Sigma) \leq k.
  \]
  For a given permutation $g \in \Sym(\Omega)$, whether $\ell(g, \Sigma) \leq k$ can be checked in~\NP\ by simply guessing a word over~$\Sigma$ of length at most~$k$, which is polynomial in the input size since $k$ is encoded in unary, and then verifying that it represents~$g$.
  Since permutations $g \in \Sym(\Omega)$ have polynomial size encodings, it follows that the problem is indeed contained in~$\Pi_2^{\P}$.

  For $\Pi_2^{\P}$-hardness, we reduce from the problem in \cref{pro:symmetric:source}.
  Given an instance $(\Gamma, k)$ of this problem, our reduction computes the set $\Sigma \sse \Sym(\Omega)$ from \cref{sec:symmetric:construction}, which is possible in polynomial time since $\wt(a) \leq 5$ for all $a \in \Gamma$.
  It then outputs the instance $(\Sigma, k + 5d)$.
  As shown in \cref{sec:symmetric:construction}, the set~$\Sigma$ generates $\Sym(\Omega)$ and consists of involutions.
  Finally, we have
  \[
    \diam(\F_2^d, \Gamma) \leq k
    \iff
    \diam(\Sym(\Omega), \Sigma) \leq k + 5d
  \]
  by \cref{pro:symmetric:diameter-identity}, which shows the correctness of our reduction.
\end{proof}

\section{Conclusion and Open Problems}
\label{sec:conclusion}

We showed that \BLength\ is \NP-complete for $2$-step nilpotent groups and \PSPACE-complete for $3$-step nilpotent groups that are also metabelian.
For \BDiameter, our classification is less sharp: the problem is $\Pi_2^{\P}$-complete for $2$-step nilpotent groups and \PSPACE-complete for solvable groups of derived length at most three (see \cref{tab:results}).
In particular, it remains open whether \BDiameter\ is \PSPACE-complete for $3$-step nilpotent groups, as the wreath product construction in \cref{thm:diameter-reduction} destroys nilpotency, and whether it is \PSPACE-complete for metabelian groups.

Another natural parameter in these problems is the number of generators.
Jerrum~\cite{Jerrum1985} showed that \BLength\ is \PSPACE-complete already for two generators, whereas our reductions in \cref{thm:diameter-reduction,thm:length:pspace} produce polynomially many.
It therefore remains open whether \BDiameter\ is \PSPACE-complete for generating sets of bounded size, and whether \BLength\ remains \PSPACE-complete for bounded generating sets of $3$-step nilpotent, metabelian groups.

Our upper bounds for $2$-step nilpotent groups rest on the fact that \textsc{Binary $2$-Realization} belongs to~\NP\ (\cref{pro:nilpotent:profile-realization}), a combinatorial problem on words that is of independent interest.
The complexity of realizing $k$-profiles for arbitrary~$k$, given in unary or binary, will be the subject of a separate paper, in which we will prove both lower and upper bounds.

Another open problem is the complexity of \BLength\ and \BDiameter\ for transitive permutation groups.
Helfgott and Seress~\cite{HeSe14} showed that the diameter of every transitive group $G \leq \Sym(d)$ with respect to every generating set is at most $\exp(\mathcal{O}((\log d)^4 \log \log d))$ and, hence, quasipolynomial in~$d$, and it is conjectured that this bound can be improved to $\mathcal{O}(d^c)$ for a universal constant~$c$~\cite{BabaiS92,EberhardMainiSabatiniTracey2026}.
If this is true, then \BLength\ is \NP-complete and \BDiameter\ is $\Pi_2^{\P}$-complete for transitive groups, where \NP-hardness of the former can be shown by modifying the construction of Even and Goldreich~\cite{EvenG81}, as remarked by Jerrum~\cite[Sec.~4]{Jerrum1985}.
Recall that we showed that the unary variant of the latter problem is $\Pi_2^{\P}$-complete already for generating sets of the full symmetric group (\cref{thm:symmetric-intro}).

\appendix

\section{The Groups in Jerrum's Construction}
\label{appendix-jerrum}

In this appendix, we show that the groups in Jerrum's \PSPACE-hardness proof~\cite[pp.~270--274]{Jerrum1985} for \BLength\ are, in general, neither metabelian nor nilpotent, and not even solvable.
Here, we consider the group generated by the generating set that Jerrum's reduction outputs, rather than the larger ambient group used to describe the construction.

Let $\mathcal M$ be the Turing machine simulated by Jerrum's reduction, with state set~$Q$ and initial state $q_0 \in Q$.
We fix an instance of \BLength\ produced by Jerrum's reduction, and we let $\Sigma_{\mathrm J} \sse \Sym(\Omega_{\mathrm J})$ be its generating set and $G_{\mathrm J} = \langle \Sigma_{\mathrm J} \rangle$.
The domain~$\Omega_{\mathrm J}$ contains the set $Q \times X$ for a nonempty finite set~$X$; the remaining points of~$\Omega_{\mathrm J}$ represent the tape contents and the head position.
For $q \in Q$, we call $B_q = \{q\} \times X$ the \emph{block} of~$q$.
Jerrum's construction then represents a configuration of~$\mathcal M$ in state~$q$ by a permutation that maps $B_{q_0}$ onto~$B_q$.
Moreover, an additional permutation $\sigma \in \Sym(X)$ serves as a counter increment within the blocks.

To describe the generators on $Q \times X$, we let $\tau_q \in \Sym(Q)$ be the transposition $(q_0, q)$ for $q \in Q \setminus \{q_0\}$, and $\tau_{q_0} = 1$.
Moreover, for $q \in Q$, we define $\hat\tau_q, \kappa_q \in \Sym(Q \times X)$ by
\[
  (r, x)\hat\tau_q = (r\tau_q, x)
  \quad\text{and}\quad
  (r, x)\kappa_q =
  \begin{cases}
    (r, x\sigma) & \text{if } r = q, \\
    (r, x) & \text{otherwise},
  \end{cases}
\]
for all $(r, x) \in Q \times X$.
Thus, for $q \neq q_0$, the permutation~$\hat\tau_q$ exchanges the blocks $B_{q_0}$ and~$B_q$, whereas $\kappa_q$ increments the counter in the block~$B_q$; each other block is mapped onto itself.

For every transition of~$\mathcal M$ from a state~$q$ to a state~$q'$, the set~$\Sigma_{\mathrm J}$ contains a \emph{transition generator}~$a$ from~$q$ to~$q'$ with
\begin{equation}\label{eqn:jerrum:transition}
  a\rest_{Q \times X} = \kappa_q \hat\tau_q \hat\tau_{q'},
\end{equation}
whose remaining factors act on the points representing the tape and fix $Q \times X$ pointwise.
Since permutations act on the right, the restriction~\eqref{eqn:jerrum:transition} first increments the counter in~$B_q$, then maps $B_q$ onto~$B_{q_0}$, and finally maps $B_{q_0}$ onto~$B_{q'}$.
Besides the transition generators, the set~$\Sigma_{\mathrm J}$ contains two generators for adjusting the counter at the end of the simulation, whose restrictions to $Q \times X$ are the identity and~$\kappa_q$ for some $q \in Q$; in particular, both map every block onto itself.

Consequently, every element of~$G_{\mathrm J}$ maps $Q \times X$ onto itself and permutes the blocks.
This yields a homomorphism $\varphi \colon G_{\mathrm J} \to \Sym(Q)$ with $B_q g = B_{q\varphi(g)}$ for all $g \in G_{\mathrm J}$ and $q \in Q$.
As every~$\kappa_q$ maps each block onto itself, \eqref{eqn:jerrum:transition} yields $\varphi(a) = \tau_q \tau_{q'}$ for every transition generator~$a$ from~$q$ to~$q'$, whereas the two remaining generators lie in the kernel of~$\varphi$.

Finally, the \emph{state graph} of~$\mathcal M$ is the directed graph with vertex set~$Q$ and an edge from~$q$ to~$q'$ whenever $\mathcal M$ has a transition from~$q$ to~$q'$, regardless of the tape symbols and the head movement.

\begin{proposition}\label{prop:jerrum-quotient}
  If every state of~$\mathcal M$ is reachable from~$q_0$ in the state graph, then $G_{\mathrm J}$ has a quotient isomorphic to~$\Sym(Q)$.
  In particular, $G_{\mathrm J}$ is not solvable if $\abs{Q} \geq 5$.
\end{proposition}

\begin{proof}
  We show that $\varphi$ is surjective, so that $G_{\mathrm J} / \ker \varphi \cong \Sym(Q)$.
  Since $\Sym(Q)$ is generated by the transpositions $(q_0, q)$ with $q \in Q \setminus \{q_0\}$, it suffices to show that each of them lies in the image of~$\varphi$.
  For such a state~$q$, there is a path $q_0, q_1, \dotsc, q_m = q$ with $m \geq 1$ in the state graph, and, for each $i \in [1, m]$, we choose a transition generator~$a_i$ from~$q_{i-1}$ to~$q_i$.
  These transitions need not form a computation of~$\mathcal M$; we only use their product as an element of~$G_{\mathrm J}$.
  As $\tau_{q_0} = 1$ and each~$\tau_{q_i}$ is trivial or an involution, the inner factors cancel, and we obtain
  \[
    \varphi(a_1 \dotsm a_m)
    = (\tau_{q_0} \tau_{q_1}) (\tau_{q_1} \tau_{q_2}) \dotsm (\tau_{q_{m-1}} \tau_{q_m})
    = \tau_{q_0} \tau_{q_m}
    = (q_0, q).
  \]

  The second statement follows, since the class of solvable groups is closed under taking quotients (\cref{sec:preliminaries-groups-1}), whereas the group $\Sym(Q)$ is not solvable if $\abs{Q} \geq 5$ (its derived series stabilizes at the nontrivial subgroup $\Alt(Q)$).
\end{proof}

The reachability assumption in \cref{prop:jerrum-quotient} is no restriction, since deleting all states that are not reachable from~$q_0$ in the state graph changes neither the computations of~$\mathcal M$ nor the language it accepts.
Moreover, we may let~$\mathcal M$ pass through a fixed sequence of new states that leave the tape contents and the head position unchanged before the original computation starts; this ensures $\abs{Q} \geq 5$ without changing the accepted language or the space bound.
Consequently, every \PSPACE-complete language is decided by a polynomial-space machine for which all groups produced by Jerrum's reduction have a nonsolvable quotient and are, therefore, neither metabelian nor nilpotent.
In particular, the restrictions in \cref{thm:length:pspace} do not follow from Jerrum's construction.

\section{Generating Sets of Bounded Hamming Weight}
\label{appendix-source}

In this appendix, we prove \cref{pro:symmetric:source} by showing that the reduction from the proof of~\cite[Thm.~3.1]{LohreyRosowski2023} only produces instances $(\Gamma, k)$ that satisfy the additional conditions of \cref{pro:symmetric:source}.
To this end, we first recall this reduction, using the notation from \cref{sec:symmetric:construction}.

The reduction is from the $\Pi_2^{\P}$-complete problem \textsc{$\forall\exists$Sat}~\cite{stockmeyer}, which asks whether a given quantified Boolean formula $\Phi_0 = \forall X_0 \, \exists Y_0 \, F_0$ is true, where $X_0$ and~$Y_0$ are disjoint sets of variables, the quantifiers $\forall X_0$ and $\exists Y_0$ quantify all variables in these sets, and $F_0 = \bigwedge_{c \in C_0} c$ is in conjunctive normal form.
The formula~$\Phi_0$ is then preprocessed in three steps.
\begin{itemize}
  \item First, a new existentially quantified variable~$y_\star$ is introduced, and $F_0$ is replaced by
    \[
      \neg y_\star \wedge \bigwedge_{c \in C_0} (y_\star \vee c).
    \]
    Every satisfying assignment must set~$y_\star$ to false, so that this preserves the truth value of~$\Phi_0$.
    Moreover, every clause now contains an existential literal.
  \item Second, every variable~$z$ with $s \geq 4$ occurrences\footnote{In~\cite{LohreyRosowski2023}, this step is formulated for variables that occur in at least four clauses; counting occurrences instead makes a difference only for clauses that contain both a variable and its negation; the proof in~\cite{LohreyRosowski2023} applies verbatim.} is replaced by distinct new variables $z_1, \dotsc, z_s$, one for each of its occurrences and preserving the sign of the occurrence, and the clauses
    \[
      (\neg z_1 \vee z_2), \ (\neg z_2 \vee z_3), \ \dotsc, \ (\neg z_{s-1} \vee z_s), \ (\neg z_s \vee z_1)
    \]
    are added, which force all copies to have the same truth value.
    If $z$ is universally quantified, then $z_1$ is universally quantified and $z_2, \dotsc, z_s$ are existentially quantified; otherwise, all copies are existentially quantified.
    Placing all existential quantifiers after the universal ones, the resulting formula has the same truth value as~$\Phi_0$.
    Afterwards, every variable occurs in at most three clauses, and every clause still contains an existential literal (as $y_\star$ and all its copies are existentially quantified, and each added clause contains an existential copy).
  \item Third, the tautological clause $c_x = (x \vee \neg x)$ is added for every universally quantified variable~$x$.
\end{itemize}

Let $X$ and~$Y$ be the resulting sets of universally and existentially quantified variables, respectively, and let~$C$ be the resulting set of clauses.
For a literal~$\lambda$, let $C(\lambda) \sse C$ be the set of clauses containing~$\lambda$; by the second and third step, we have $\abs{C(\lambda)} \leq 4$.
The reduction then works with the elementary abelian $2$-group~$\F_2^A$ for $A = X \uplus Y \uplus C$, as well as the sets defined for $x \in X$ and $y \in Y$ by
\begin{gather*}
  \Gamma_x = \{ e_x + \one_B : B \sse C(x) \} \cup \{ \one_B : \emptyset \neq B \sse C(\neg x) \} \\
  \shortintertext{and}
  \Gamma_y = \{ \one_B : \emptyset \neq B \sse \{y\} \cup C(y) \} \cup \{ \one_B : \emptyset \neq B \sse \{y\} \cup C(\neg y) \}.
\end{gather*}
The output of the reduction is the instance $(\Gamma, k)$ with
\[
  \Gamma = \bigcup_{x \in X} \Gamma_x \cup \bigcup_{y \in Y} \Gamma_y
  \quad\text{and}\quad
  k = \abs{X} + \abs{Y},
\]
where $k$ is encoded in unary.

It remains to verify the additional conditions of \cref{pro:symmetric:source}.
By definition, all elements of~$\Gamma$ are nonzero, and their supports contain at most one variable and at most four clauses each, so that $1 \leq \wt(a) \leq 5$ for all $a \in \Gamma$.
Moreover, $\Gamma$ contains the unit vector~$e_v$ for every $v \in A$, which strengthens the fact that $\Gamma$ generates~$\F_2^A$~\cite[Claim~3.2]{LohreyRosowski2023}.
Indeed, for $x \in X$ and $y \in Y$, choosing $B = \emptyset$ and $B = \{y\}$ yields $e_x \in \Gamma_x$ and $e_y \in \Gamma_y$, respectively.
For a clause $c \in C$ that contains an existential literal $y$ or~$\neg y$, choosing $B = \{c\}$ yields $e_c \in \Gamma_y$.
The only remaining clauses are the tautological clauses $c_x \in C(\neg x)$ for $x \in X$, for which choosing $B = \{c_x\}$ yields $e_{c_x} \in \Gamma_x$.

\section*{AI Methodology}
\phantomsection
\addcontentsline{toc}{section}{AI Methodology}
\label{sec:ai-methodology}

The problems studied in this paper and the statements of our results were either conjectured by the first author and Rosowski~\cite{LohreyRosowski2023} or formulated by the authors.
We used GPT-6 Astra as a research assistant to obtain the proofs of \cref{thm:diameter-reduction,thm:length:pspace}, \cref{lem:nilpotent:profile-compression}, and \cref{thm:symmetric:involutions}.
The tool produced complete proofs, which the authors checked in full and substantially rewrote and simplified; the rewritten proofs are presented in \crefrange{sec:diameter}{sec:symmetric}.
We used Claude Opus 5.5 to help revise the authors' draft of the paper.
The authors verified the correctness and originality of all content, including references, and take full responsibility for the final version of the paper.

\bibliographystyle{alphaurl}
\bibliography{main}

\newcommand{\etalchar}[1]{$^{#1}$}
\begin{thebibliography}{MRUV10}

\bibitem[AB09]{AroraBarak2009}
Sanjeev Arora and Boaz Barak.
\newblock {\em Computational Complexity: A Modern Approach}.
\newblock Cambridge University Press, Cambridge, 2009.

\bibitem[ARMS03]{AloulRamaniMarkovSakallah2003}
Fadi~A. Aloul, Arathi Ramani, Igor~L. Markov, and Karem~A. Sakallah.
\newblock Solving difficult instances of {Boolean} satisfiability in the presence of symmetry.
\newblock {\em IEEE Transactions on Computer-Aided Design of Integrated Circuits and Systems}, 22(9):1117--1137, 2003.
\newblock \href {https://doi.org/10.1109/TCAD.2003.816218} {\path{doi:10.1109/TCAD.2003.816218}}.

\bibitem[Bab16]{Babai16}
L{\'{a}}szl{\'{o}} Babai.
\newblock Graph isomorphism in quasipolynomial time [extended abstract].
\newblock In {\em Proceedings of the 48th Annual {ACM} {SIGACT} Symposium on Theory of Computing, {STOC} 2016}, pages 684--697. {ACM}, 2016.
\newblock \href {https://doi.org/10.1145/2897518.2897542} {\path{doi:10.1145/2897518.2897542}}.

\bibitem[BBS04]{BabaiBS04}
L{\'{a}}szl{\'{o}} Babai, Robert Beals, and {\'{A}}kos Seress.
\newblock On the diameter of the symmetric group: polynomial bounds.
\newblock In {\em Proceedings of the 15th Annual {ACM-SIAM} Symposium on Discrete Algorithms, {SODA} 2004}, pages 1108--1112. {SIAM}, 2004.
\newblock URL: \url{https://dl.acm.org/doi/10.5555/982792.982956}.

\bibitem[BH92]{BabaiH92}
L{\'{a}}szl{\'{o}} Babai and G{\'{a}}bor Hetyei.
\newblock On the diameter of random {Cayley} graphs of the symmetric group.
\newblock {\em Combinatorics, Probability \& Computing}, 1:201--208, 1992.
\newblock \href {https://doi.org/10.1017/S0963548300000237} {\path{doi:10.1017/S0963548300000237}}.

\bibitem[BH05]{BabaiH05}
L{\'{a}}szl{\'{o}} Babai and Thomas~P. Hayes.
\newblock Near-independence of permutations and an almost sure polynomial bound on the diameter of the symmetric group.
\newblock In {\em Proceedings of the 16th Annual {ACM-SIAM} Symposium on Discrete Algorithms, {SODA} 2005}, pages 1057--1066. {SIAM}, 2005.
\newblock URL: \url{http://dl.acm.org/citation.cfm?id=1070432.1070584}.

\bibitem[BHK{\etalchar{+}}90]{BabaiHKLS90}
L{\'{a}}szl{\'{o}} Babai, G{\'{a}}bor Hetyei, William~M. Kantor, Alexander Lubotzky, and {\'{A}}kos Seress.
\newblock On the diameter of finite groups.
\newblock In {\em Proceedings of the 31st Annual Symposium on Foundations of Computer Science, {FOCS} 1990, Volume {II}}, pages 857--865. {IEEE} Computer Society, 1990.
\newblock \href {https://doi.org/10.1109/FSCS.1990.89608} {\path{doi:10.1109/FSCS.1990.89608}}.

\bibitem[BKL89]{BabaiKL89}
L{\'{a}}szl{\'{o}} Babai, William~M. Kantor, and A.~Lubotzky.
\newblock Small-diameter {Cayley} graphs for finite simple groups.
\newblock {\em European Journal of Combinatorics}, 10(6):507--522, 1989.
\newblock \href {https://doi.org/10.1016/S0195-6698(89)80067-8} {\path{doi:10.1016/S0195-6698(89)80067-8}}.

\bibitem[BS88]{BabaiS88}
L{\'{a}}szl{\'{o}} Babai and {\'{A}}kos Seress.
\newblock On the diameter of {Cayley} graphs of the symmetric group.
\newblock {\em Journal of Combinatorial Theory, Series {A}}, 49(1):175--179, 1988.
\newblock \href {https://doi.org/10.1016/0097-3165(88)90033-7} {\path{doi:10.1016/0097-3165(88)90033-7}}.

\bibitem[BS92]{BabaiS92}
L{\'{a}}szl{\'{o}} Babai and {\'{A}}kos Seress.
\newblock On the diameter of permutation groups.
\newblock {\em European Journal of Combinatorics}, 13(4):231--243, 1992.
\newblock \href {https://doi.org/10.1016/S0195-6698(05)80029-0} {\path{doi:10.1016/S0195-6698(05)80029-0}}.

\bibitem[CEJS98]{ClarkeEmersonJhaSistla1998}
Edmund~M. Clarke, E.~Allen Emerson, Somesh Jha, and A.~Prasad Sistla.
\newblock Symmetry reductions in model checking.
\newblock In Alan~J. Hu and Moshe~Y. Vardi, editors, {\em Computer Aided Verification ({CAV} 1998)}, volume 1427 of {\em Lecture Notes in Computer Science}, pages 147--158. Springer, 1998.
\newblock \href {https://doi.org/10.1007/BFb0028741} {\path{doi:10.1007/BFb0028741}}.

\bibitem[DDE{\etalchar{+}}11]{DemaineEtAl2011}
Erik~D. Demaine, Martin~L. Demaine, Sarah Eisenstat, Anna Lubiw, and Andrew Winslow.
\newblock Algorithms for solving {Rubik's Cubes}.
\newblock In {\em Proceedings of the 19th Annual European Symposium on Algorithms (ESA 2011)}, pages 689--700. Springer, 2011.
\newblock URL: \url{https://erikdemaine.org/papers/Rubik_ESA2011/}, \href {https://doi.org/10.1007/978-3-642-23719-5_58} {\path{doi:10.1007/978-3-642-23719-5_58}}.

\bibitem[EG81]{EvenG81}
Shimon Even and Oded Goldreich.
\newblock The minimum-length generator sequence problem is {NP}-hard.
\newblock {\em Journal of Algorithms}, 2(3):311--313, 1981.
\newblock \href {https://doi.org/10.1016/0196-6774(81)90029-8} {\path{doi:10.1016/0196-6774(81)90029-8}}.

\bibitem[EMST26]{EberhardMainiSabatiniTracey2026}
Sean Eberhard, Elena Maini, Luca Sabatini, and Gareth Tracey.
\newblock Diameter bounds for arbitrary finite groups and applications, 2026.
\newblock Preprint, arXiv:2604.15303.
\newblock URL: \url{https://arxiv.org/abs/2604.15303}, \href {https://arxiv.org/abs/2604.15303} {\path{arXiv:2604.15303}}, \href {https://doi.org/10.48550/arXiv.2604.15303} {\path{doi:10.48550/arXiv.2604.15303}}.

\bibitem[ES06]{ES2006}
Friedrich Eisenbrand and Gennady Shmonin.
\newblock Carath\'eodory bounds for integer cones.
\newblock {\em Operations Research Letters}, 34(5):564--568, 2006.
\newblock \href {https://doi.org/10.1016/j.orl.2005.09.008} {\path{doi:10.1016/j.orl.2005.09.008}}.

\bibitem[FGK{\etalchar{+}}15]{FreydenbergerEtAl2015}
Dominik~D. Freydenberger, Pawe{\l} Gawrychowski, Juhani Karhum{\"a}ki, Florin Manea, and Wojciech Rytter.
\newblock Testing {$k$}-binomial equivalence.
\newblock In {\em Multidisciplinary Creativity: Homage to {Gheorghe P{\u{a}}un} on His 65th Birthday}, pages 239--248. Spandugino, Bucharest, Romania, 2015.
\newblock URL: \url{https://arxiv.org/abs/1509.00622}.

\bibitem[FHL80]{FHL1980}
Merrick Furst, John~E. Hopcroft, and Eugene~M. Luks.
\newblock Polynomial-time algorithms for permutation groups.
\newblock In {\em Proceedings of the 21st Annual Symposium on Foundations of Computer Science, {FOCS} 1980}, pages 36--41. IEEE Computer Society, 1980.
\newblock \href {https://doi.org/10.1109/SFCS.1980.34} {\path{doi:10.1109/SFCS.1980.34}}.

\bibitem[Gla99]{Glasby1999}
S.~P. Glasby.
\newblock Subgroups of the upper-triangular matrix group with maximal derived length and a minimal number of generators.
\newblock In C.~M. Campbell, E.~F. Robertson, N.~Ru{\v{s}}kuc, and G.~C. Smith, editors, {\em Groups St Andrews 1997 in Bath, I}, volume 260 of {\em London Mathematical Society Lecture Note Series}, pages 275--281. Cambridge University Press, 1999.
\newblock \href {https://doi.org/10.1017/CBO9781107360228.019} {\path{doi:10.1017/CBO9781107360228.019}}.

\bibitem[HS14]{HeSe14}
Harald~A. Helfgott and {\'{A}}kos Seress.
\newblock On the diameter of permutation groups.
\newblock {\em Annals of Mathematics}, 179(2):611--658, 2014.
\newblock \href {https://doi.org/10.4007/annals.2014.179.2.4} {\path{doi:10.4007/annals.2014.179.2.4}}.

\bibitem[Huf98]{Huffman1998CodesAndGroups}
W.~C. Huffman.
\newblock Codes and groups.
\newblock In V.~S. Pless and W.~C. Huffman, editors, {\em Handbook of Coding Theory}, volume~2, chapter~17, pages 1345--1440. Elsevier, Amsterdam, 1998.

\bibitem[Jer85]{Jerrum1985}
Mark Jerrum.
\newblock The complexity of finding minimum-length generator sequences.
\newblock {\em Theoretical Computer Science}, 36:265--289, 1985.
\newblock \href {https://doi.org/10.1016/0304-3975(85)90047-7} {\path{doi:10.1016/0304-3975(85)90047-7}}.

\bibitem[Joy08]{Joyner2008}
David Joyner.
\newblock {\em Adventures in Group Theory: {Rubik's Cube}, {Merlin's Machine}, and Other Mathematical Toys}.
\newblock Johns Hopkins University Press, 2 edition, 2008.
\newblock URL: \url{https://www.press.jhu.edu/books/title/9554/adventures-group-theory}.

\bibitem[KMS84]{KMS84}
D.~Kornhauser, G.~Miller, and P.~Spirakis.
\newblock Coordinating pebble motion on graphs, the diameter of permutation groups, and applications.
\newblock In {\em Proceedings of the 25th IEEE Symposium on Foundations of Computer Science, {FOCS} 1984}, pages 241--250. IEEE Computer Society Press, 1984.
\newblock \href {https://doi.org/10.1109/SFCS.1984.715921} {\path{doi:10.1109/SFCS.1984.715921}}.

\bibitem[LR23]{LohreyRosowski2023}
Markus Lohrey and Andreas Rosowski.
\newblock On the complexity of diameter and related problems in permutation groups.
\newblock In {\em Proceedings of the 50th International Colloquium on Automata, Languages, and Programming, {ICALP} 2023}, volume 261 of {\em Leibniz International Proceedings in Informatics (LIPIcs)}, pages 134:1--134:18. Schloss Dagstuhl -- Leibniz-Zentrum f{\"u}r Informatik, 2023.
\newblock \href {https://doi.org/10.4230/LIPIcs.ICALP.2023.134} {\path{doi:10.4230/LIPIcs.ICALP.2023.134}}.

\bibitem[LRR20]{LejeuneRigoRosenfeld2020}
Marie Lejeune, Michel Rigo, and Matthieu Rosenfeld.
\newblock The binomial equivalence classes of finite words.
\newblock {\em International Journal of Algebra and Computation}, 30(7):1375--1397, 2020.
\newblock \href {https://doi.org/10.1142/S0218196720500459} {\path{doi:10.1142/S0218196720500459}}.

\bibitem[Mag37]{Magnus1937}
Wilhelm Magnus.
\newblock {\"Uber} {Beziehungen} zwischen h\"oheren {Kommutatoren}.
\newblock {\em Journal für die reine und angewandte Mathematik}, 177:105--115, 1937.
\newblock \href {https://doi.org/10.1515/crll.1937.177.105} {\path{doi:10.1515/crll.1937.177.105}}.

\bibitem[Mar02]{Margot2002}
Fran{\c{c}}ois Margot.
\newblock Pruning by isomorphism in branch-and-cut.
\newblock {\em Mathematical Programming}, 94(1):71--90, 2002.
\newblock \href {https://doi.org/10.1007/s10107-002-0358-2} {\path{doi:10.1007/s10107-002-0358-2}}.

\bibitem[McK84]{McKe84}
Pierre McKenzie.
\newblock Permutations of bounded degree generate groups of polynomial diameter.
\newblock {\em Information Processing Letters}, 19(5):253--254, 1984.
\newblock \href {https://doi.org/10.1016/0020-0190(84)90062-0} {\path{doi:10.1016/0020-0190(84)90062-0}}.

\bibitem[McK98]{McKay1998}
Brendan~D. McKay.
\newblock Isomorph-free exhaustive generation.
\newblock {\em Journal of Algorithms}, 26(2):306--324, 1998.
\newblock \href {https://doi.org/10.1006/jagm.1997.0898} {\path{doi:10.1006/jagm.1997.0898}}.

\bibitem[McL84]{McLoughlin1984}
Aileen~M. McLoughlin.
\newblock The complexity of computing the covering radius of a code.
\newblock {\em IEEE Transactions on Information Theory}, 30(6):800--804, 1984.
\newblock \href {https://doi.org/10.1109/TIT.1984.1056978} {\path{doi:10.1109/TIT.1984.1056978}}.

\bibitem[MRUV10]{MyasnikovRomankovUshakovVershik2010}
Alexei Myasnikov, Vitaly Roman'kov, Alexander Ushakov, and Anatoly Vershik.
\newblock The word and geodesic problems in free solvable groups.
\newblock {\em Transactions of the American Mathematical Society}, 362(9):4655--4682, 2010.
\newblock \href {https://doi.org/10.1090/S0002-9947-10-04959-7} {\path{doi:10.1090/S0002-9947-10-04959-7}}.

\bibitem[O'B04]{OBryant2004}
Kevin O'Bryant.
\newblock A complete annotated bibliography of work related to {S}idon sequences.
\newblock {\em The Electronic Journal of Combinatorics}, 2004.
\newblock Dynamic Survey DS11.
\newblock \href {https://doi.org/10.37236/32} {\path{doi:10.37236/32}}.

\bibitem[Pap94]{Papadimitriou1994}
Christos~H. Papadimitriou.
\newblock {\em Computational Complexity}.
\newblock Addison-Wesley, Reading, MA, 1994.

\bibitem[Par92]{Parry1992}
Walter Parry.
\newblock Growth series of some wreath products.
\newblock {\em Transactions of the American Mathematical Society}, 331(2):751--759, 1992.
\newblock \href {https://doi.org/10.1090/S0002-9947-1992-1062874-3} {\path{doi:10.1090/S0002-9947-1992-1062874-3}}.

\bibitem[RD14]{RokickiDavidson2014}
Tomas Rokicki and Morley Davidson.
\newblock {God's Number} is 26 in the {Quarter-Turn Metric}.
\newblock Online announcement, 2014.
\newblock Accessed 31 October 2026.
\newblock URL: \url{https://cube20.org/qtm/}.

\bibitem[RKDD13]{RokickiEtAl2013}
Tomas Rokicki, Herbert Kociemba, Morley Davidson, and John Dethridge.
\newblock The diameter of the {Rubik's Cube} group is twenty.
\newblock {\em SIAM Journal on Discrete Mathematics}, 27(2):1082--1105, 2013.
\newblock \href {https://doi.org/10.1137/120867366} {\path{doi:10.1137/120867366}}.

\bibitem[Rob96]{Robinson1996}
Derek J.~S. Robinson.
\newblock {\em A Course in the Theory of Groups}, volume~80 of {\em Graduate Texts in Mathematics}.
\newblock Springer, New York, 2 edition, 1996.
\newblock \href {https://doi.org/10.1007/978-1-4419-8594-1} {\path{doi:10.1007/978-1-4419-8594-1}}.

\bibitem[RS62]{RosserSchoenfeld1962}
J.~Barkley Rosser and Lowell Schoenfeld.
\newblock Approximate formulas for some functions of prime numbers.
\newblock {\em Illinois Journal of Mathematics}, 6(1):64--94, 1962.
\newblock \href {https://doi.org/10.1215/ijm/1255631807} {\path{doi:10.1215/ijm/1255631807}}.

\bibitem[Ser03]{Seress2003}
{\'A}kos Seress.
\newblock {\em Permutation Group Algorithms}, volume 152 of {\em Cambridge Tracts in Mathematics}.
\newblock Cambridge University Press, Cambridge, 2003.
\newblock \href {https://doi.org/10.1017/CBO9780511546549} {\path{doi:10.1017/CBO9780511546549}}.

\bibitem[Sim70]{Sims1970}
Charles~C. Sims.
\newblock Computational methods in the study of permutation groups.
\newblock In John Leech, editor, {\em Computational Problems in Abstract Algebra}, pages 169--183. Pergamon Press, Oxford, 1970.
\newblock \href {https://doi.org/10.1016/B978-0-08-012975-4.50020-5} {\path{doi:10.1016/B978-0-08-012975-4.50020-5}}.

\bibitem[Sto76]{stockmeyer}
Larry~J. Stockmeyer.
\newblock The polynomial-time hierarchy.
\newblock {\em Theoretical Computer Science}, 3(1):1--22, 1976.
\newblock \href {https://doi.org/10.1016/0304-3975(76)90061-X} {\path{doi:10.1016/0304-3975(76)90061-X}}.

\bibitem[Wil74]{Wilson1974}
Richard~M. Wilson.
\newblock Graph puzzles, homotopy, and the alternating group.
\newblock {\em Journal of Combinatorial Theory, Series B}, 16(1):86--96, 1974.
\newblock \href {https://doi.org/10.1016/0095-8956(74)90098-7} {\path{doi:10.1016/0095-8956(74)90098-7}}.

\end{thebibliography}

\end{document}